\documentclass[a4paper,11pt]{amsart}

\usepackage{graphicx}
\usepackage{amsmath}
\usepackage{amssymb}

\newtheorem{thm}{Theorem}[section]
\newtheorem{lem}[thm]{Lemma}%
\newtheorem{prop}[thm]{Proposition}%
\theoremstyle{definition}
\newtheorem{defn}{Definition}[section]
\theoremstyle{remark}
\newtheorem{remark}{Remark}[section] %

\theoremstyle{plain}

\numberwithin{equation}{section}

\def\QQ{{\mathbb Q}}
\def\PP{{\mathbb P}}
\def\RR{{\mathbb R}}

\def\ZZ{{\mathbb Z}}

\def\veca{{\text{\boldmath$a$}}}

\def\vece{{\text{\boldmath$e$}}}
\def\vech{{\text{\boldmath$h$}}}

\def\vecm{{\text{\boldmath$m$}}}
\def\vecn{{\text{\boldmath$n$}}}
\def\vecq{{\text{\boldmath$q$}}}
\def\vecQ{{\text{\boldmath$Q$}}}
\def\vecp{{\text{\boldmath$p$}}}
\def\vecr{{\text{\boldmath$r$}}}
\def\vecR{{\text{\boldmath$R$}}}
\def\vecs{{\text{\boldmath$s$}}}
\def\uvecs{\widehat{\text{\boldmath$s$}}}
\def\vecS{{\text{\boldmath$S$}}}
\def\uvecS{\widehat{\text{\boldmath$S$}}}

\def\vecu{{\text{\boldmath$u$}}}
\def\vecv{{\text{\boldmath$v$}}}
\def\vecV{{\text{\boldmath$V$}}}
\def\vecw{{\text{\boldmath$w$}}}
\def\uvecw{\widehat{\text{\boldmath$w$}}}
\def\vecW{{\text{\boldmath$W$}}}
\def\vecx{{\text{\boldmath$x$}}}

\def\vecy{{\text{\boldmath$y$}}}
\def\vecz{{\text{\boldmath$z$}}}
\def\vecalf{{\text{\boldmath$\alpha$}}}
\def\vecbeta{{\text{\boldmath$\beta$}}}

\def\vecomega{{\text{\boldmath$\omega$}}}

\def\vecxi{{\text{\boldmath$\xi$}}}

\def\vecnull{{\text{\boldmath$0$}}}

\def\scrA{{\mathcal A}}
\def\scrB{{\mathcal B}}

\def\scrD{{\mathcal D}}
\def\scrE{{\mathcal E}}
\def\scrF{{\mathcal F}}

\def\scrK{{\mathcal K}}
\def\scrL{{\mathcal L}}
\def\scrM{{\mathcal M}}
\def\scrN{{\mathcal N}}

\def\scrR{{\mathcal R}}
\def\scrS{{\mathcal S}}
\def\scrT{{\mathcal T}}
\def\scrU{{\mathcal U}}
\def\scrV{{\mathcal V}}
\def\scrW{{\mathcal W}}
\def\scrX{{\mathcal X}}

\def\fU{{\mathfrak U}}

\def\fZ{{\mathfrak Z}}

\def\adm{\operatorname{adm}}

\def\id{\operatorname{id}}

\def\C{\operatorname{C{}}}

\def\L{\operatorname{L{}}}

\def\S{\operatorname{S{}}}

\def\SL{\operatorname{SL}}
\def\ASL{\operatorname{ASL}}
\def\SO{\operatorname{SO}}

\def\O{\operatorname{O{}}}
\def\T{\operatorname{T{}}}

\def\supp{\operatorname{supp}}

\def\vol{\operatorname{vol}}

\def\ASLASL{\ASL(d,\ZZ)\backslash\ASL(d,\RR)}

\def\ASLR{\ASL(d,\RR)}

\def\SLR{\SL(d,\RR)}

\def\trans{\,^\mathrm{t}\!}

\def\Onder#1#2#3#4#5{#1 \setbox0=\hbox{$#1$}\setbox1=\hbox{$#2$}
       \dimen0=.5\wd0 \dimen1=\dimen0 \dimen2=\dp0 \dimen3=\dimen2
       \advance\dimen0 by .5\wd1 \advance\dimen0 by -#4
       \advance\dimen1 by -.5\wd1 \advance\dimen1 by -#4
       \advance\dimen2 by -#3 \advance\dimen2 by \ht1
       \advance\dimen2 by 0.3ex \advance\dimen3 by #5
        \kern-\dimen0\raisebox{-\dimen2}[0ex][\dimen3]{\box1}
       \kern\dimen1}

\newcommand{\XX}{{\mathcal X}}
\newcommand{\Ll}{\L^1_{\operatorname{loc{}}}}

\newcommand{\dto}{\,\widetilde{\to}\,}

\newcommand{\Q}{\mathbb{Q}}
\newcommand{\R}{\mathbb{R}}
\newcommand{\Z}{\mathbb{Z}}

\newcommand{\HS}{{{\S'_1}^{d-1}}}

\newcommand{\sfrac}[2]{{\textstyle \frac {#1}{#2}}}
\newcommand{\col}{\: : \:}

\newcommand{\bn}{\mathbf{0}}

\newcommand{\ttau}{{\tilde{\tau}}}
\newcommand{\trho}{{\tilde{\rho}}}
\newcommand{\tbe}{\widetilde{\vecbeta}}
\newcommand{\ttbe}{\widehat{\vecbeta}}
\newcommand{\tg}{\tilde{g}}
\newcommand{\tw}{{\tilde{\vecw}}}

\newcommand{\tm}{{\tilde{\vecm}}}

\newcommand{\ve}{\varepsilon}

\newcommand{\matr}[4]{\left( \begin{matrix} #1 & #2 \\ #3 & #4 \end{matrix} \right) }
\newcommand{\smatr}[4]{\left( \begin{smallmatrix} #1 & #2 \\ #3 & #4 \end{smallmatrix} \right) }

\title{The Boltzmann-Grad limit of the periodic Lorentz gas}
\author{Jens Marklof}
\author{Andreas Str\"ombergsson}
\address{School of Mathematics, University of Bristol,
Bristol BS8 1TW, U.K.\newline
\rule[0ex]{0ex}{0ex} \hspace{8pt}{\tt j.marklof@bristol.ac.uk}}
\address{Department of Mathematics, Box 480, Uppsala University,
SE-75106 Uppsala, Sweden\newline
\rule[0ex]{0ex}{0ex} \hspace{8pt}{\tt astrombe@math.uu.se}}
\date{3 January 2008}
\thanks{J.M.\ has been supported by EPSRC Research Grants GR/T28058/01 and GR/S87461/01, and a Philip Leverhulme Prize. A.S.\ is a Royal Swedish Academy of Sciences Research Fellow supported by
a grant from the Knut and Alice Wallenberg Foundation.}

\begin{document}

\begin{abstract}
We study the dynamics of a point particle in a periodic array of spherical scatterers, and construct a stochastic process that governs the time evolution for random initial data in the limit of low scatterer density (Boltzmann-Grad limit). A generic path of the limiting process is a piecewise linear curve whose consecutive segments are generated by a Markov process with memory two.
\end{abstract}

\maketitle
\tableofcontents

\section{Introduction}\label{secIntro}

The Lorentz gas describes an ensemble of non-interacting point particles in an infinite array of spherical scatterers. It was originally developed by Lorentz \cite{Lorentz05} in 1905 to model, in the limit of low scatterer density (Boltzmann-Grad limit), the stochastic properties of the motion of electrons in a metal. In the present paper we consider the case of a periodic array of scatterers, and construct a stochastic process that indeed governs the macroscopic dynamics of a particle cloud in the Boltzmann-Grad limit. The corresponding result has been known for some time in the case of a Poisson-distributed (rather than periodic) configuration of scatterers. Here the limiting process corresponds to a solution of the linear Boltzmann equation, see Gallavotti \cite{Gallavotti69}, Spohn \cite{Spohn78}, and Boldrighini, Bunimovich and Sinai \cite{Boldrighini83}. It already follows from the estimates in \cite{Bourgain98,Golse00} that the linear Boltzmann equation does not hold in the periodic set-up; this was pointed out recently by Golse \cite{Golse06,Golse07}. 

Our results complement classical studies in ergodic theory that characterize the stochastic properties of the periodic Lorentz gas in the limit of long times, see \cite{Bunimovich80,Bleher92,Chernov94,Melbourne05,Melbourne07,Balint07,Szasz07,dolgopyat} for details. 

To state our main results, consider an ensemble of non-interacting point particles moving in an array of spherical scatterers which are placed at the vertices of a euclidean lattice $\scrL\subset\RR^d$ of covolume one (Figure \ref{figLorentz}). The dynamics of each particle is governed by the billiard flow
\begin{equation}
	\varphi_t : \T^1(\scrK_\rho) \to \T^1(\scrK_\rho), \qquad (\vecq_0,\vecv_0) \mapsto (\vecq(t),\vecv(t))
\end{equation}
where $\scrK_\rho\subset\RR^d$ is the complement of the set $\scrB^d_\rho + \scrL$ (the ``billiard domain''), and $\T^1(\scrK_\rho)=\scrK_\rho\times\S_1^{d-1}$ is its unit tangent bundle (the ``phase space''). $\scrB^d_\rho$ denotes the open ball of radius $\rho$, centered at the origin. A point in $\T^1(\scrK_\rho)$ is parametrized by $(\vecq,\vecv)$, with $\vecq\in\scrK_\rho$ denoting the position and $\vecv\in\S_1^{d-1}$ the velocity of the particle. The Liouville measure on $\T^1(\scrK_\rho)$ is 
\begin{equation} \label{LIOUVILLEDEF}
	d\nu(\vecq,\vecv)=d\!\vol_{\RR^d}(\vecq)\, d\!\vol_{\S_1^{d-1}}(\vecv)
\end{equation}
where $\vol_{\RR^d}$ and $\vol_{\S_1^{d-1}}$ refer to the Lebesgue measures on $\RR^d$ %
and $\S_1^{d-1}$, respectively. For the purpose of this introduction we will restrict our attention to Lorentz' classical set-up, where the scatterers are assumed to be hard spheres. Our results in fact also hold for scattering processes described by smooth potentials, see Sec.~\ref{secScatt} for details.

If the initial condition $(\vecq_0,\vecv_0)$ is random, the billiard flow gives rise to the stochastic process
\begin{equation}\label{LP}
	\{ (\vecq(t),\vecv(t)) : t\in\RR_{>0} \}
\end{equation}
which we will refer to as the Lorentz process. The central result of this paper is the existence of a limiting stochastic process $\{\Xi(t): t\in\RR_{>0}\}$ of the Lorentz process in the Boltzmann-Grad limit $\rho\to 0$.
We begin with a study of the distribution of path segments of the billiard flow between collisions.

\begin{figure}
\begin{center}
\framebox{
\begin{minipage}{0.4\textwidth}
\unitlength0.1\textwidth
\begin{picture}(10,10)(0,0)
\put(0.5,1){\includegraphics[width=0.9\textwidth]{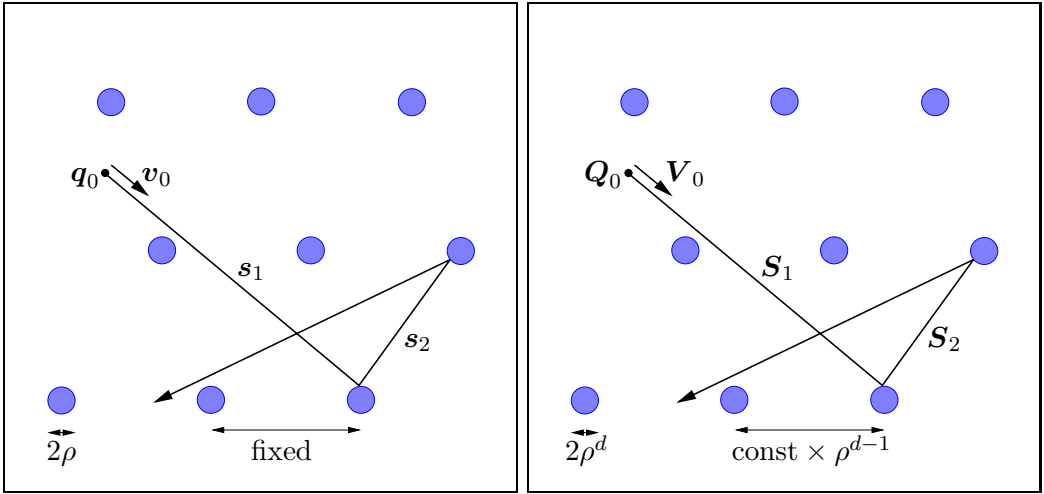}}
\put(1,6.4){$\vecq_0$} \put(2.5,6.4){$\vecv_0$}
\put(4.5,4.4){$\vecs_1$} \put(8,2.9){$\vecs_2$}
\put(0.5,0.5){$2\rho$} \put(4.8,0.5){fixed}
\end{picture}
\end{minipage}
}
\framebox{
\begin{minipage}{0.4\textwidth}
\unitlength0.1\textwidth
\begin{picture}(10,10)(0,0)
\put(0.5,1){\includegraphics[width=0.9\textwidth]{lorentzgas.eps}}
\put(0.8,6.4){$\vecQ_0$} \put(2.5,6.4){$\vecV_0$}
\put(4.5,4.4){$\vecS_1$} \put(8,2.9){$\vecS_2$}
\put(0.4,0.5){$2\rho^d$} \put(3.9,0.5){$\text{const}\times\rho^{d-1}$}
\end{picture}
\end{minipage}
}
\end{center}
\caption{Left: The periodic Lorentz gas in ``microscopic'' coordinates---the lattice $\scrL$ remains fixed as the radius $\rho$ of the scatterer tends to zero. Right: The periodic Lorentz gas in ``macroscopic'' coordinates ---both the lattice constant and the radius of each scatter tend to zero, in such a way that the mean free path length remains finite. The vectors $\vecs_1,\vecs_2,\ldots$ (resp. $\vecS_1,\vecS_2,\ldots$) represent the segments of the billiard path between collisions.} \label{figLorentz}
\end{figure}

\subsection{The joint distribution of path segments}\label{intrjointdistrsec}

The billiard flow $\varphi_t$ induces a billiard map on the boundary $\partial\T^1(\scrK_\rho)$, 
\begin{equation} \label{MNWNDEF}
(\vecq_{n-1},\vecv_{n-1}) \mapsto (\vecq_{n},\vecv_{n}) ,
\end{equation}
where $\vecq_n,\vecv_n$ denote position and velocity at the $n$th collision in the {\em outgoing} configuration, i.e.,
\begin{equation} \label{OUTGOING}
	(\vecq_n,\vecv_{n})=\lim_{\epsilon\to 0+}\varphi_{\tau_1(\vecq_{n-1},\vecv_{n-1};\rho)+\epsilon}(\vecq_{n-1},\vecv_{n-1}).
\end{equation}
Here $\tau_1$ denotes the free path length, 
defined by
\begin{equation} \label{TAU1DEF}
\tau_1(\vecq,\vecv;\rho)=\inf\{t>0\col\vecq+t\vecv\notin\scrK_\rho\}.
\end{equation}
We will later also use the parametrization $\partial\T^1(\scrK_\rho)=(\S_\rho^{d-1}+\scrL)\times \S_1^{d-1}$, so that $\vecq_n=\vecm_{n}+\rho\vecw_{n}$, where $\vecw_n\in\S_1^{d-1}$ and $\vecm_n\in\scrL$ are the position on the ball and ball label at the $n$th collision.

The time elapsed between the $(n-1)$th and $n$th hit is defined as the $n$th collision time
\begin{equation} \label{TAUNDEF}
	\tau_n(\vecq_0,\vecv_0;\rho)= \tau_1(\vecq_{n-1},\vecv_{n-1};\rho) .
\end{equation}
We express the $n$th path segment by the vector
\begin{equation}\label{pathseg}
	\vecs_n(\vecq_0,\vecv_0;\rho):=\tau_n(\vecq_0,\vecv_0;\rho) \vecv_{n-1}(\vecq_0,\vecv_0;\rho) .
\end{equation}
The central result of \cite{partI} is the proof of a limiting distribution for the first collision time $\tau_1$, and further refined versions that also take into account the particle's direction after the reflection. We will here extend these results to find a joint limiting distribution for the first $n$ segments of the billiard path with initial coordinates $(\vecq_0,\vecv_0)$, where the position $\vecq_0$ is fixed and the velocity $\vecv_0$ random. The precise statement is the following.

Here and in the remainder of this paper we will use the standard representation $\scrL=\ZZ^d M_0$, where $M_0\in\SL(d,\RR)$.
We will also use the notation $\uvecS:=\|\vecS\|^{-1}\vecS$. We set
\begin{equation}
\scrB_n:=\bigl\{(\vecS_1,\ldots,\vecS_n)\in(\RR^{d}\setminus\{\vecnull\})^n:\; \uvecS_{j+1}\neq\uvecS_{j}\; (j=1,\ldots,n-1) \bigr\} .
\end{equation}
\begin{thm}\label{secThmMicro}
Fix a lattice $\scrL=\Z^dM_0$ and a point $\vecq_0\in\RR^d\setminus\scrL$, and write $\vecalf=-\vecq_0 M_0^{-1}$. Then for each $n\in\ZZ_{>0}$ there exists a %
function $P_{\vecalf}^{(n)}:\scrB_n\to\R_{\geq 0}$ 
such that, 
for any Borel probability measure $\lambda$ on $\S_1^{d-1}$ which is 
absolutely continuous with respect to $\vol_{\S_1^{d-1}}$,
and for any set $\scrA\subset\RR^{nd}$ with boundary of Lebesgue measure zero,
\begin{multline} \label{secThm-eq}
	\lim_{\rho\to 0}\lambda\big( \big\{ \vecv_0\in \S_1^{d-1} : (\vecs_1(\vecq_0,\vecv_0;\rho),\ldots,\vecs_n(\vecq_0,\vecv_0;\rho)) \in \rho^{-(d-1)}\scrA \big\} \big) \\
= \int_{\scrA} P_{\vecalf}^{(n)}(\vecS_1,\ldots,\vecS_n)\,  \lambda'(\uvecS_1)\, d\!\vol_{\RR^d}(\vecS_1)\cdots d\!\vol_{\RR^d}(\vecS_n),
\end{multline}
where $\lambda'\in \L^1(\S_1^{d-1})$ is the Radon-Nikodym derivative of
$\lambda$ with respect to $\vol_{\S_1^{d-1}}$.
Furthermore, there is a function $\Psi:\scrB_3\to\R_{\geq 0}$ 
such that
\begin{equation} \label{jointlimdens}
	P_{\vecalf}^{(n)}(\vecS_1,\ldots,\vecS_n) = P_{\vecalf}^{(2)}(\vecS_1,\vecS_2)
	\prod_{j=3}^n \Psi(\vecS_{j-2},\vecS_{j-1},\vecS_j) 
\end{equation}
for all $n\geq 3$ and all $(\vecS_1,\ldots,\vecS_n)\in\scrB_n$.
\end{thm}

The above condition $\vecq_0\in\RR^d\setminus\scrL$ ensures that $\tau_1$ is defined for $\rho$ sufficiently small. In Sec.~\ref{ThmMicroproofsec} we also consider variants of Theorem \ref{secThmMicro} where the initial position is near $\scrL$, e.g., $\vecq_0\in\partial\scrK_\rho$.

We define the probability measure corresponding to \eqref{secThm-eq} by 
\begin{equation} \label{MUPROBMEAS}
\mu_{\vecalf,\lambda}^{(n)}(\scrA)
:= \int_{\scrA} P_{\vecalf}^{(n)}(\vecS_1,\ldots,\vecS_n)\,  \lambda'(\uvecS_1)\, d\!\vol_{\RR^d}(\vecS_1)\cdots d\!\vol_{\RR^d}(\vecS_n).	
\end{equation}
Note in particular that $\mu_{\vecalf,\lambda}^{(n+1)}(\scrA\times\RR^d)=\mu_{\vecalf,\lambda}^{(n)}(\scrA)$.

\begin{remark}
In probabilistic terms, Theorem \ref{secThmMicro} states that the discrete-time stochastic process $\{\rho^{d-1} \vecs_n(\vecq_0,\vecv_0;\rho) : n\in\ZZ_{>0} \}$
converges in the limit $\rho\to 0$ to 
\begin{equation}\label{Markovprocess}
\{ \vecS_n : n\in\ZZ_{>0} \}	,
\end{equation}
a Markov process with memory two. As we shall see, $\Psi(\vecS_1,\vecS_2,\vecS_3)$ is in fact independent of $\|\vecS_1\|$. %
\end{remark}

\begin{remark}
If $d\geq 3$ then %
$P_\vecalf^{(n)}$ is %
continuous on all of $\scrB_n$. If $d=2$ then $P_\vecalf^{(n)}$ is continuous
except possibly at points $(\vecS_1,\ldots,\vecS_n)\in\scrB_n$ with
$\uvecS_2=-\uvecS_1$ or 
$\uvecS_{j+2}=\uvecS_jR_{\vecS_{j+1}}$ for some $1\leq j\leq n-2$,
where $R_\vecS\in\O(2)$ denotes reflection in the line $\R\vecS$.
Cf.\ Remark \ref{secThmMicrocontrem} below.
\end{remark}

\begin{remark}\label{indi-remark}
Note that $\Psi$ is independent of $\scrL$ and $\vecq_0$, and $P_{\vecalf}^{(n)}$
depends only on the choice of $\vecalf$. This means in particular that $\Psi$ and
$P_{\vecalf}^{(n)}$ are rotation-invariant, i.e., for any $K\in\O(d)$ we have
\begin{equation}
        \Psi(\vecS_{1} K,\vecS_{2} K,\vecS_3 K) = \Psi(\vecS_{1},\vecS_{2},\vecS_3)
\end{equation}
and
\begin{equation}
        P_{\vecalf}^{(n)}(\vecS_1 K,\ldots,\vecS_n K) =
P_{\vecalf}^{(n)}(\vecS_1,\ldots,\vecS_n) .
\end{equation}
For $\vecalf\in\RR^d\setminus\QQ^d$ also $P_{\vecalf}^{(n)}=:P^{(n)}$ is in fact
independent of $\vecalf$, cf.\ Remark \ref{indi-remarkGS} below.
Explicit formulas and asymptotic properties of the limiting distributions will be
presented in \cite{partIII}.
\end{remark}

\begin{remark}
The case $n=1$ of course leads to the distribution of the free path length discussed in \cite{partI}; cf. also \cite{Dahlqvist97,Bourgain98,Golse00,Caglioti03,Boca07} for earlier results.
\end{remark}

\subsection{A limiting stochastic process for the billiard flow}\label{secStochastic}

In Theorem \ref{secThmMicro} we have identified a Markov process with memory two that describes the limiting distribution of billiard paths with random initial data $(\vecq_0,\vecv_0)$. Let us denote by
\begin{equation}\label{SP}
	\{ \Xi(t) : t\in\RR_{>0} \},
\end{equation}
the continuous-time stochastic process that is obtained by moving with unit speed along the random paths of the Markov process \eqref{Markovprocess}. The process is fully specified by the probability 
\begin{equation}\label{PProb}
	\PP_{\vecalf,\lambda}\big(\Xi(t_1)\in\scrD_1,\ldots, \Xi(t_M)\in\scrD_M\big)
\end{equation}
that $\Xi(t)$ visits the sets $\scrD_1,\ldots,\scrD_M\subset\T^1(\R^d)$ at times $t=t_1,\ldots,t_M$, with $M$ arbitrarily large. To give a precise definition of \eqref{PProb} set
$T_0:=0$, 
$T_n:=\sum_{j=1}^n \|\vecS_j\|$, and define the probability that $\Xi(t)$ is in the set $\scrD_1$ at time $t_1$ after exactly $n_1$ hits, in the set $\scrD_2$ at time $t_2$ after exactly $n_2$ hits, etc., by
\begin{multline} \label{PNLQ0LDEF}
	\PP_{\vecalf,\lambda}^{(\vecn)}\big(\Xi(t_1)\in\scrD_1,\ldots, \Xi(t_M)\in\scrD_M \text{ and } T_{n_1}\leq t_1< T_{n_1+1},\ldots, T_{n_M}\leq t_M< T_{n_M+1}  \big) \\
	:=
\mu_{\vecalf,\lambda}^{(n+1)}\big(\big\{(\vecS_1,\ldots,\vecS_{n+1}): \Xi_{n_j}(t_j) \in\scrD_j,\; T_{n_j}\leq t_j< T_{n_j+1}\;(j=1,\ldots,M) \big\}\big) 
\end{multline}
with $\vecn:=(n_1,\ldots,n_M)$, $n:=\max(n_1,\ldots,n_M)$, and
\begin{equation}\label{Xn}
	\Xi_n(t):= \bigg( \sum_{j=1}^n \vecS_j + (t-T_n) \uvecS_{n+1}, \uvecS_{n+1}\bigg) .	
\end{equation}
Note that the choice $T_{n}\leq t< T_{n+1}$ of semi-open intervals is determined by the use of the outgoing configuration, recall \eqref{OUTGOING}.
The formal definition of \eqref{PProb} is thus
\begin{multline} \label{PPLQ0LDEF}
		\PP_{\vecalf,\lambda}\big(\Xi(t_1)\in\scrD_1,\ldots, \Xi(t_M)\in\scrD_M\big) \\ := \sum_{\vecn\in\ZZ_{\geq 0}^M} \PP_{\vecalf,\lambda}^{(\vecn)}\big(\Xi(t_1)\in\scrD_1,\ldots, \Xi(t_M)\in\scrD_M \text{ and } T_{n_1}\leq t_1< T_{n_1+1},\ldots, T_{n_M}\leq t_M< T_{n_M+1}  \big).
\end{multline}

The following theorem shows that the Lorentz process \eqref{LP}, suitably rescaled, converges to the stochastic process \eqref{SP} as $\rho\to 0$. 
Given any set $\scrD\subset\T^1(\R^d)$ we say that $t\geq 0$ is
\textit{$\scrD$-admissible} if 
\begin{align}
\vol_{\S_1^{d-1}}\bigl(\bigl\{\uvecS_1\in\S_1^{d-1}\col
(t\uvecS_1,\uvecS_1)\in\partial\scrD\bigr\}\bigr)=0.
\end{align}
We write $\adm(\scrD)$ for the set of all 
$\scrD$-admissible numbers $t\geq 0$.
\begin{thm} \label{secThmMicro2}
Fix a lattice $\scrL=\ZZ^d M_0$ and a point $\vecq_0\in\RR^d\setminus\scrL$, 
set $\vecalf=-\vecq_0 M_0^{-1}$,
and let $\lambda$ be a Borel probability measure on $\S_1^{d-1}$ which is absolutely continuous with respect to $\vol_{\S_1^{d-1}}$. 
Then, for any subsets $\scrD_1,\ldots,\scrD_M\subset \T^1(\RR^{d})$ 
with boundary of Lebesgue measure zero, and any numbers
$t_j\in\adm(\scrD_j)$ ($j=1,\ldots,M$),
\begin{multline} \label{secThm-eq2}
	\lim_{\rho\to 0}\lambda\big( \big\{ \vecv_0\in \S_1^{d-1} : (\rho^{d-1}\vecq(\rho^{-(d-1)}t_j),\vecv(\rho^{-(d-1)} t_j)) \in\scrD_j, \; j=1,\ldots,M \big\} \big) \\
= \PP_{\vecalf,\lambda}\big(\Xi(t_1)\in\scrD_1,\ldots, \Xi(t_M)\in\scrD_M\big) .
\end{multline}
The convergence is uniform for $(t_1,\ldots,t_M)$ in compact subsets of 
$\adm(\scrD_1)\times\cdots\times\adm(\scrD_M)$. 
\end{thm}

\begin{remark} \label{ADMREMARK}
The condition $t_j\in\adm(\scrD_j)$ %
cannot be disposed with.
For example, \eqref{secThm-eq2} is in general \textit{false} in the
case $M=1$, $\scrD_1=\scrB_{t_1}^d\times\S_1^{d-1}$.
We prove this in  Section~\ref{ADMREMARKSEC}.
Note however that no admissibility condition is required in the macroscopic
analogue of Theorem~\ref{secThmMicro2}, see Theorem \ref{secThmMacro2} below.
\end{remark}

\subsection{Macroscopic initial conditions\label{secMacroscopic}}

In view of the rescaling applied in the previous section it is natural to consider the ``macroscopic'' billiard flow
\begin{align}
	F_t : & \T^1(\rho^{d-1}\scrK_\rho) \to \T^1(\rho^{d-1}\scrK_\rho) \\ 
	 & (\vecQ_0,\vecV_0) \mapsto  (\vecQ(t),\vecV(t))=(\rho^{d-1}\vecq(\rho^{-(d-1)}t),\vecv(\rho^{-(d-1)}t)) ,\notag
\end{align}
and take random initial conditions $(\vecQ_0,\vecV_0)$ with respect to some fixed probability measure $\Lambda$. We will establish the analogous limit laws as in the previous sections. Although the macroscopic versions are less general (they are obtained by averaging over $\vecq_0$), they appear more natural from a physical viewpoint, where one is interested in the time evolution of a macroscopic particle cloud; cf. the discussion at end of this section.

The $n$th path segment in these macroscopic coordinates is 
\begin{equation}
	\vecS_n(\vecQ_0,\vecV_0;\rho):=\rho^{d-1}\vecs_n(\rho^{-(d-1)}\vecQ_0,\vecV_0;\rho).
\end{equation}

\begin{thm}\label{secThmMacro}
Fix a lattice $\scrL$ and let $\Lambda$ be a Borel probability measure on $\T^1(\RR^d)$ which is absolutely continuous with respect to Lebesgue measure. Then, for each $n\in\ZZ_{>0}$, and for any set $\scrA\subset \RR^d\times\RR^{nd}$ with boundary of Lebesgue measure zero,
\begin{multline} \label{secThm-eq-macro}
	\lim_{\rho\to 0}\Lambda\big( \big\{ (\vecQ_0,\vecV_0)\in \T^1(\rho^{d-1}\scrK_\rho) : (\vecQ_0,\vecS_1(\vecQ_0,\vecV_0;\rho),\ldots,\vecS_n(\vecQ_0,\vecV_0;\rho)) \in \scrA \big\} \big) \\
= \int_{\scrA} P^{(n)}(\vecS_1,\ldots,\vecS_n)\, \Lambda'\big(\vecQ_0,\uvecS_1\big)\, d\!\vol_{\RR^d}(\vecQ_0)\, d\!\vol_{\RR^d}(\vecS_1)\cdots d\!\vol_{\RR^d}(\vecS_n) ,
\end{multline}
with $P^{(n)}$ as in Remark \ref{indi-remark},
and where $\Lambda'$ is the Radon-Nikodym derivative of
$\Lambda$ with respect to $\vol_{\R^d}\times\vol_{\S_1^{d-1}}$.
\end{thm}

The probability measure corresponding to the above limiting distribution is defined by
\begin{equation} \label{muLambdandef}
	\mu_{\Lambda}^{(n)}(\scrA):= \int_{\scrA} P^{(n)}(\vecS_1,\ldots,\vecS_n)\, \Lambda'\big(\vecQ_0,\uvecS_1\big)\, d\!\vol_{\RR^d}(\vecQ_0)\, d\!\vol_{\RR^d}(\vecS_1)\cdots d\!\vol_{\RR^d}(\vecS_n) .
\end{equation}

We redefine the stochastic process \eqref{SP} by specifying the probability
\begin{equation}\label{PProb-macro}
	\PP_{\Lambda}\big(\Xi(t_1)\in\scrD_1,\ldots, \Xi(t_M)\in\scrD_M\big)
\end{equation}
via the measure $\mu_\Lambda^{(n)}$ by the same construction as in Section \ref{secStochastic}. The only essential difference is that we need to replace \eqref{Xn} by
\begin{equation} \label{XNDEFMACRO}
	\Xi_n(t):= \bigg( \vecQ_0+ \sum_{j=1}^n \vecS_j + (t-T_n) \uvecS_{n+1}, \uvecS_{n+1}\bigg) .
\end{equation}
Note that formally $\PP_\Lambda=\PP_{\vecalf,\lambda}$ if $\Lambda'(\vecQ,\vecV)=\delta(\vecQ) \lambda'(\vecV)$, and $\vecalf\in\RR^d\setminus\QQ^d$.

\begin{thm}\label{secThmMacro2}
Fix a lattice $\scrL$ and let $\Lambda$ be a Borel probability measure on $\T^1(\RR^d)$ which is absolutely continuous with respect to Lebesgue measure. Then, for any $t_1,\ldots,t_M\in\RR_{\geq 0}$, and any subsets $\scrD_1,\ldots,\scrD_M\subset \T^1(\RR^{d})$ with boundary of Lebesgue measure zero,
\begin{multline} \label{secThm-eq2-macro}
	\lim_{\rho\to 0}\Lambda\big(\big\{(\vecQ_0,\vecV_0)\in \T^1(\rho^{d-1}\scrK_\rho) : (\vecQ(t_1),\vecV(t_1))\in \scrD_1,\ldots, (\vecQ(t_M),\vecV(t_M))\in \scrD_M \big\}\big) \\
= \PP_{\Lambda}\big(\Xi(t_1)\in\scrD_1,\ldots, \Xi(t_M)\in\scrD_M\big) .
\end{multline}
The convergence is uniform for $t_1,\ldots,t_M$ in compact subsets of $\RR_{\geq 0}$. 
\end{thm}

The time evolution of an initial particle cloud $f\in\L^1(\T^1(\rho^{d-1}\scrK_\rho))$ in the periodic Lorentz gas is described by the operator $L_{t,\rho}$ defined by
\begin{equation}
	[L_{t,\rho}f](\vecQ,\vecV)=f(F_t^{-1}(\vecQ,\vecV)) .
\end{equation}
To allow a $\rho$-independent choice of the initial density $f$, it is convenient to extend the action of $F_t$ from $\T^1(\rho^{d-1}\scrK_\rho)$ to $\T^1(\RR^d)$ by setting $F_t=\id$ on $\T^1(\RR^d)\setminus\T^1(\rho^{d-1}\scrK_\rho)$.
We fix the Liouville measure on $\T^1(\RR^d)$ to be the standard 
Lebesgue measure
\begin{align}
d\nu(\vecQ,\vecV)=d\!\vol_{\RR^d}(\vecQ)\, d\!\vol_{\S_1^{d-1}}(\vecV).
\end{align}
Theorem \ref{secThmMacro2} now implies the existence of a limiting operator $L_t$ that describes the evolution of the particle cloud in the Boltzmann-Grad limit. More precisely, for every set $\scrD$ with boundary of Lebesgue measure zero, we have
\begin{equation}\label{Boltzlimit}
	\lim_{\rho\to 0} \int_\scrD [L_{t,\rho}f](\vecQ,\vecV)\, d\nu(\vecQ,\vecV)= \int_\scrD [L_t f](\vecQ,\vecV)\,d\nu(\vecQ,\vecV) ,
\end{equation}
uniformly for $t$ on compacta in $\RR_{\geq 0}$, and $L_t$ is defined by the relation
\begin{equation}
	\int_\scrD [L_t f](\vecQ,\vecV)\,d\nu(\vecQ,\vecV)
	= \PP_{\Lambda}\big(\Xi(t)\in\scrD \big),
\end{equation}
for any absolutely continuous $\Lambda$, any Borel subset $\scrD\subset\T^1(\RR^d)$,
and $f=\Lambda'$. We note that $L_t$ commutes with the translation operators $\{
T_\vecR : \vecR\in\RR^d\}$,
\begin{equation}
        [T_\vecR f](\vecQ,\vecV) := f(\vecQ-\vecR,\vecV),
\end{equation}
and, in view of Remark \ref{indi-remark}, with the rotation operators $\{ R_K : K\in
\O(d)\}$,
\begin{equation}
        [R_K f](\vecQ,\vecV) := f(\vecQ K,\vecV K) .
\end{equation}

It was already pointed out by Golse \cite{Golse07}, that the weak-$*$ limit of any converging subsequence $L_{t,\rho_i}f$ ($\rho_i\to 0$) does not satisfy the linear Boltzmann equation. His arguments use the a priori estimates in \cite{Bourgain98,Golse00}, and do not require knowledge of the existence of the limit \eqref{Boltzlimit}. The fundamental reason behind the failure of the linear Boltzmann equation is that, perhaps surprisingly,  $\{L_t: t\geq 0\}$ is not a semigroup. We will show in Section \ref{secExtended} how to overcome this problem by considering an extended stochastic process that keeps track not only of position $\vecQ$ and current velocity $\vecV$, but also the free path length $\scrT$ until the next collision, and the velocity $\vecV_+$ thereafter. We will establish that the extended process is Markovian, and derive the corresponding Fokker-Planck-Kolmogorov equation describing the evolution of the particle density in the extended phase space. A similar approach has recently been explored by Caglioti and Golse in the two-dimensional case \cite{Caglioti07}. Their result is however conditional on an independence hypothesis, which is equivalent to the Markov property established by our Theorem \ref{secThmMacro} above.

\subsection{Outline of the paper} 
The key ingredient in the present work is Theorem 4.8 of \cite{partI} (restated as Theorem \ref{exactpos2-1hit} below for general scattering maps), which yields the joint limiting distribution for the free path length and velocity after the next collision, given that the initial position and velocity are taken at random with respect to a fixed probability measure. The proofs of Theorems \ref{secThmMicro} and \ref{secThmMacro} are based on a uniform version of Theorem \ref{exactpos2-1hit}, where the fixed probability measures are replaced by certain equismooth families, see Section \ref{secFirst} for details. Section \ref{reflmapssec} provides technical information on the $n$th iterate of the scattering maps, which in conjunction with the uniform version of Theorem \ref{exactpos2-1hit} yields the proof of Theorems \ref{secThmMicro} and \ref{secThmMacro} (Section \ref{secLoss}). In Section \ref{secConvergence} we prove that the dynamics in the periodic Lorentz gas converges in the Boltzmann-Grad limit to a stochastic process $\Xi(t)$, and thus establish Theorems \ref{secThmMicro2} and \ref{secThmMacro2}. We finally derive the substitute for the linear Boltzmann equation in Section \ref{secExtended}, by extending $\Xi(t)$ to a Markov process and calculating its Fokker-Planck-Kolmogorov equation.

\section{First collision}\label{secFirst}

We begin by reviewing the central result of \cite{partI}.

\subsection{Location of the first collision}\label{firstcoll}

We fix a lattice $\scrL=\Z^d M_0$ with $M_0\in\SL(d,\R)$,
once and for all. Recall that $\scrK_\rho\subset\R^d$ is the complement of
the set $\scrB_\rho^d+\scrL$ and that the free path length for the initial 
condition $(\vecq,\vecv)\in\T^1(\scrK_\rho)$ is defined as
\begin{align}
\tau_1(\vecq,\vecv;\rho)=\inf\{t>0\col\vecq+t\vecv\notin\scrK_\rho\}.
\end{align}
Note that $\tau_1(\vecq,\vecv;\rho)=\infty$ can only happen for a set of
$\vecv$'s of measure zero with respect to $\vol_{\S_1^{d-1}}$. 
In fact we have $\tau_1(\vecq,\vecv;\rho)<\infty$ whenever the $d$
coordinates of $\vecv M_0^{-1}\in\R^d$ are linearly independent over $\Q$
(note that this condition is independent of $\vecq$),
since then each orbit of the linear flow $\vecx\mapsto\vecx+t\vecv$ is 
dense on $\R^d/\scrL$.

The position of the particle when hitting the first scatterer is
\begin{equation}
	\vecq_1(\vecq,\vecv;\rho) := \vecq+\tau_1(\vecq,\vecv;\rho) \vecv .
\end{equation}
As in Sec.\ \ref{intrjointdistrsec} we write 
$\vecq_1(\vecq,\vecv;\rho)=\vecm_1+\rho\vecw_1$ with
$\vecm_1\in\scrL$ and $\vecw_1=\vecw_1(\vecq,\vecv;\rho)\in \S_1^{d-1}$.

Let us fix a map $K:\S_1^{d-1}\to\SO(d)$ such that
$\vecv K(\vecv)=\vece_1$ for all $\vecv\in\S_1^{d-1}$;
we assume that $K$ is smooth when restricted to $\S_1^{d-1}$ minus
one point. 
For example, we may choose $K$ as 
$K(\vece_1)=I$, $K(-\vece_{1})=-I$ and
\begin{equation}
	K(\vecv)=
E\Bigl(-\frac{2\arcsin\bigl(\|\vecv-\vece_1\|/2\bigr)}
{\|\vecv_\perp\|} \vecv_\perp\Bigr)
\end{equation}
for
$\vecv\in\S_1^{d-1}\setminus\{\vece_1,-\vece_1\}$,
where 
\begin{equation}
	\vecv_\perp:=(v_2,\ldots,v_d)\in\R^{d-1}, \qquad E(\vecw)=\exp\matr 0\vecw{-\trans\vecw}\bn\in\SO(d).
\end{equation} 
Then $K$ is smooth when restricted to $\S_1^{d-1}\setminus\{-\vece_1\}$.

It is evident that $-\vecw_1 K(\vecv)\in \HS$, with the hemisphere
\begin{equation}
 \HS=\{\vecv=(v_1,\ldots,v_d)\in\S_1^{d-1} \col v_1>0\}. 	
\end{equation}

Let $\vecbeta$ be a continuous function $\S_1^{d-1}\to\RR^d$.
If $\vecq\in\scrL$ we assume that
$(\vecbeta(\vecv)+\R_{>0}\vecv)\cap\scrB_1^d=\emptyset$
for all $\vecv\in\S_1^{d-1}$.
We will consider initial conditions of the form
$(\vecq_{\rho,\vecbeta}(\vecv),\vecv)\in\T^1(\scrK_\rho)$, where
$\vecq_{\rho,\vecbeta}(\vecv)=\vecq+\rho\vecbeta(\vecv)$
and where $\vecv$ is picked at random in $\S_1^{d-1}$.
Note that for fixed $\vecq$ and $\vecbeta$ we indeed have
$\vecq_{\rho,\vecbeta}(\vecv)\in\scrK_\rho$ for all $\vecv\in\S_1^{d-1}$,
so long as $\rho$ is sufficiently small.

For the statement of the theorem below, we recall 
the definition of the manifolds $X_q(\vecy)$ and $X(\vecy)$
from \cite[Sec.\ 7]{partI}:
If $q\in\Z_{>0}$ and $\vecalf\in q^{-1}\Z^d$, then we set
$X_q=\Gamma(q)\backslash\SLR$, and define,
for each $\vecy\in\R^d\setminus\{\bn\}$,
\begin{align}
& X_q(\vecy):=\bigl\{M\in X_q \col \vecy\in (\Z^d+\vecalf)M\bigr\}.
\end{align}
We also set $X=\ASLASL$ where $\ASLR=\SLR\ltimes \RR^d$ is the
semidirect product group with multiplication law
$(M,\vecxi)(M',\vecxi')=(MM',\vecxi M' +\vecxi')$;
we let $\ASLR$ act on $\RR^d$ through
$\vecy \mapsto \vecy(M,\vecxi):=\vecy M+\vecxi$. 
Now for each $\vecy\in\R^d$ we define
\begin{align}
& X(\vecy):=\bigl\{g\in X \col \vecy\in \Z^d g\bigr\}.
\end{align}
The spaces $X_q(\vecy)$ and $X(\vecy)$ carry natural probability 
measures $\nu_\vecy$ whose properties are discussed in \cite[Sec.\ 7]{partI}.

We will also use the notation 
\begin{align}
\vecx_\perp=\vecx-(\vecx\cdot\vece_1)\vece_1,
\qquad\text{for }\:\vecx\in\R^d.
\end{align}

The following %
is a restatement of \cite[Theorem 4.4]{partI}.

\begin{thm}\label{exactpos1}
Fix a lattice $\scrL=\Z^d M_0$.
Let $\vecq\in\R^d$ and $\vecalf=-\vecq M_0^{-1}$.
There exists a function 
$\Phi_{\vecalf}:\R_{>0}\times(\{0\}\times\scrB_1^{d-1})\times(\{0\}\times\R^{d-1})\to\R_{\geq 0}$ such that for any Borel probability measure 
$\lambda$ on $\S_1^{d-1}$ absolutely 
continuous with respect to $\vol_{\S_1^{d-1}}$, any
subset $\fU\subset\HS$ with $\vol_{\S_1^{d-1}}(\partial\fU)=0$, 
and $0\leq \xi_1<\xi_2$, we have
\begin{multline} \label{exactpos1eq}
\lim_{\rho\to 0}  \lambda\bigl(\bigl\{ \vecv\in\S_1^{d-1} \col 
\rho^{d-1} \tau_1(\vecq_{\rho,\vecbeta}(\vecv),\vecv;\rho)\in [\xi_1,\xi_2), \:  
-\vecw_1(\vecq_{\rho,\vecbeta}(\vecv),\vecv;\rho)K(\vecv)\in\fU\bigr\}\bigr) \\
=\int_{\xi_1}^{\xi_2} \int_{\fU_\perp} \int_{\S_1^{d-1}} 
\Phi_{\vecalf}\bigl(\xi,\vecw,(\vecbeta(\vecv)K(\vecv))_\perp\bigr) 
\, d\lambda(\vecv) d\vecw \, d\xi,
\end{multline}
where $d\vecw$ denotes the $(d-1)$-dimensional
Lebesgue volume measure on $\{0\}\times\R^{d-1}$.
The function $\Phi_\vecalf$ is explicitly given by
\begin{align} \label{exactpos1limitrat}
\Phi_\vecalf(\xi,\vecw,\vecz)
=\begin{cases}
\nu_\vecy\bigl(\bigl\{M\in X_q(\vecy) \col 
(\Z^d+\vecalf)M \cap (\fZ(0,\xi,1)+\vecz)=\emptyset\bigr\}\bigr)
& \text{if } \: \vecalf\in q^{-1}\Z^d
\\
\nu_\vecy\bigl(\bigl\{g\in X(\vecy) \col \Z^d g \cap (\fZ(0,\xi,1)+\vecz)=\emptyset
\bigr\}\bigr)
& \text{if } \: \vecalf\notin \Q^d,
\end{cases}
\end{align}
where $\vecy=\xi\vece_1+\vecw+\vecz$, and 
\begin{align} \label{FZC1C2DEF}
	\fZ(c_1,c_2,\sigma) =\big\{(x_1,\ldots,x_d)\in\RR^d : c_1 < x_1 < c_2, \|(x_2,\ldots,x_d)\|< \sigma \big\} .
\end{align}
\end{thm}

\begin{remark} \label{PALFBETSPECREM}
For $\vecalf\in \Q^d$
the function $\Phi_\vecalf(\xi,\vecw,\vecz)$ is Borel measurable,
and in fact only depends on $\vecalf$ and
the four real numbers $\xi,\|\vecw\|,\|\vecz\|,\vecz\cdot\vecw$.
Also for $\vecalf\in \Q^d$,
if we restrict to $\|\vecz\|\leq 1$ [and if $d=2$: $\vecz+\vecw\neq \bn$],
then $\Phi_\vecalf(\xi,\vecw,\vecz)$ is jointly continuous
in the three variables $\xi,\vecw,\vecz$.
If $\vecalf\notin\Q^d$ then $\Phi_\vecalf(\xi,\vecw,\vecz)$
is everywhere continuous in the three variables,
and it is independent of both $\vecalf$ and $\vecz$;
in fact it only depends on $\xi$ and $\|\vecw\|$.
We have, for all $\vecalf\in\R^d$ and all $\vecz\in\{0\}\times\R^{d-1}$,
\begin{align} \label{PALFINTEQ1}
\int_0^\infty \int_{\{0\}\times\scrB_1^{d-1}} 
\Phi_\vecalf(\xi,\vecw,\vecz)\,d\vecw\,d\xi=1.
\end{align}
The convergence in this integral is uniform, i.e.\ 
\begin{equation} \label{PALFBETSPECREMFORMULA}
\int_T^\infty \int_{\{0\}\times\scrB_1^{d-1}} 
\Phi_\vecalf(\xi,\vecw,\vecz)\,d\vecw\,d\xi \to 0
\end{equation}
uniformly with respect to $\vecalf$ and $\vecz$ as $T\to\infty$.
Cf.\ \cite[Remark 4.5, (8.37) and Lemma 8.15]{partI}.
\end{remark}

\begin{remark} \label{XI0REM}
We may extend the function $\Phi_\vecalf(\xi,\vecw,\vecz)$ 
to the larger set
$\R_{\geq 0}\times(\{0\}\times\scrB_1^{d-1})\times(\{0\}\times\R^{d-1})$
by letting $\Phi_\vecalf(0,\vecw,\vecz):=1$ for all $\vecw,\vecz$.
This definition is natural, since it makes $\Phi_\vecalf(\xi,\vecw,\vecz)$ 
continuous (jointly in the three variables) at each point with $\xi=0$.
The proof of this is an immediate extension of the discussion in 
\cite[Sec.\ 8.1, 8.2]{partI}.
\end{remark}

\subsection{Scattering maps}\label{secScatt}

As indicated above, the results of this paper extend to the case of a Lorentz gas, where the scattering process at a hard sphere is replaced by a smooth radial potential. To obtain the correct scaling in the Boltzmann-Grad limit, we assume the scattering potential is of the form $V(\vecq/\rho)$, where $V(\vecq)$ has compact support in the unit ball $\scrB_1^d$. We will refer to the rescaled ball $\scrB_\rho^d=\rho\scrB_1^d$ as the {\em interaction region}.

It is most convenient to phrase the required assumptions in terms of a scattering map that describes the dynamics at each scatterer. Let 
\begin{align}
\scrS_-:=\{(\vecv,\vecw)\in\S_1^{d-1}\times\S_1^{d-1}\col \vecv\cdot\vecw<0\}
\end{align}
be the set of {\em incoming data} $(\vecw_-,\vecv_-)$, i.e., the relative position and velocity with which the particle enters the interaction region. The corresponding {\em outgoing data} is parametrized by the set 
\begin{align}
\scrS_+:=\{(\vecv,\vecw)\in\S_1^{d-1}\times\S_1^{d-1}\col \vecv\cdot\vecw>0\}.
\end{align}
We define the scattering map by
\begin{align}
\Theta:\scrS_-\to\scrS_+,\qquad
(\vecv_-,\vecw_-)\mapsto(\vecv_+,\vecw_+).
\end{align}
In the case of the original Lorentz gas the scattering map is 
given by specular reflection,
\begin{align} \label{LORENTZSCATTERING}
\Theta(\vecv,\vecw)=(\vecv-2(\vecv\cdot\vecw)\vecw,\vecw) ;
\end{align}
scattering maps corresponding to smooth potentials can be readily obtained from classical results, see e.g. \cite[Chapter 5]{Newton82}.

In the following we will treat the scattering process as instantaneous. In the case of potentials the particle will of course spend a non-zero amount of time in the interaction region, but---under standard assumptions on the potential---this time will tend to zero when $\rho\to 0$.
 
Let $\Theta_1(\vecv,\vecw)\in\S_1^{d-1}$ and 
$\Theta_2(\vecv,\vecw)\in\S_1^{d-1}$ be the projection of 
$\Theta(\vecv,\vecw)\in\S_1^{d-1}$ onto the first and second component,
respectively.
We assume throughout this paper that 
\begin{enumerate}
	\item[(i)] the scattering map $\Theta$
is \textit{spherically symmetric}, i.e., if 
$(\vecv_+,\vecw_+)=\Theta(\vecv,\vecw)$ then
$(\vecv_+K,\vecw_+K)=\Theta(\vecv K,\vecw K)$ for all
$K\in \O(d)$;
	\item[(ii)] $\vecv_+$ and $\vecw_+$ are contained in the
subspace spanned by $\vecv$ and $\vecw$;
	\item[(iii)] if $\vecw=-\vecv$ then $\vecv_+=-\vecv$;
	\item[(iv)] $\Theta:\scrS_-\to\scrS_+$ is $\C^1$ and 
for each fixed $\vecv\in\S_1^{d-1}$ the map $\vecw\mapsto\Theta_1(\vecv,\vecw)$
is a $\C^1$ diffeomorphism from $\{\vecw\in\S_1^{d-1}\col\vecv\cdot\vecw<0\}$
onto some open subset of $\S_1^{d-1}$.
\end{enumerate}
The above conditions are for example satisfied for the scattering map of a ``muffin-tin'' Coulomb potential, $V(\vecq)=\alpha \max(\|\vecq\|^{-1}-1,0)$ with $\alpha\notin \{-2E,0\}$, where $E$ denotes the
total energy. Conditions (iii) and (iv) help to simplify the presentation of the proofs, but are not essential. It is for instance not necessary that in (iv) the map $\vecw\mapsto\Theta_1(\vecv,\vecw)$ is invertible as long as it has finitely many pre-images, which allows a larger class of scattering potentials; condition (iii) can be dropped entirely.

We will write $\varphi(\vecv,\vecu)\in [0,\pi]$ for the angle between
any two vectors $\vecv,\vecu\in\R^d\setminus\{\bn\}$.
Using the spherical symmetry 
and $\Theta_1(\vecv,-\vecv)=-\vecv$ one sees that 
there exists a constant $0\leq B_\Theta<\pi$ such that for each  
$\vecv\in\S_1^{d-1}$, the image of the diffeomorphism 
$\vecw\mapsto\Theta_1(\vecv,\vecw)$ is
\begin{align} \label{VPDEF}
\scrV_{\vecv}:=\{\vecu\in\S_1^{d-1}\col\varphi(\vecv,\vecu)>B_\Theta\}.
\end{align}
Let us write $\vecbeta_\vecv^-:\scrV_\vecv\to
\{\vecw\in\S_1^{d-1}\col\vecv\cdot\vecw<0\}$ for the inverse map.
Then $\vecbeta_\vecv^-$ is spherically symmetric in the sense that
$\vecbeta_{\vecv K}^-(\vecu K)=\vecbeta_{\vecv}^-(\vecu)K$ 
for all $\vecv\in\S_1^{d-1}$, $\vecu\in\scrV_\vecv$, $K\in\O(d)$,
and in particular $\vecbeta_\vecv^-(\vecu)$ is jointly $\C^1$ in $\vecv,\vecu$.

We also define
\begin{align}
\vecbeta^+_{\vecv}(\vecu)=\Theta_2(\vecv,\vecbeta^-_{\vecv}(\vecu))
\qquad (\vecv\in\S_1^{d-1},\:\vecu\in\scrV_\vecv).
\end{align}
The map $\vecbeta^+$ is also spherically symmetric and jointly $\C^1$ in 
$\vecv,\vecu$.
In terms of the original scattering situation, the point of our notation 
is the following:
Given any $\vecv_-,\vecv_+\in\S_1^{d-1}$, there exists
$\vecw_-,\vecw_+\in\S_1^{d-1}$ such that %
$\Theta(\vecv_-,\vecw_-)=(\vecv_+,\vecw_+)$ if and only if
$\varphi(\vecv_-,\vecv_+)>B_\Theta$, and in this case $\vecw_-$ and $\vecw_+$
are uniquely determined, as $\vecw_\pm=\vecbeta^\pm_{\vecv_-}(\vecv_+)$.

For example, in the case of specular reflection \eqref{LORENTZSCATTERING} we have $B_\Theta=0$ and 
\begin{equation}
	\vecw_+=\vecw_-=\frac{\vecv_+-\vecv_-}{\|\vecv_+-\vecv_-\|} .
\end{equation}

\begin{remark} \label{PRESERVELIOUVILLEREMARK}
In the case of specular reflection, the flow $F_t$ preserves the
Liouville measure $\nu$, %
but this does not hold in the case of a general scattering map 
satisfying (i)--(iv).
Indeed, a necessary and sufficient condition for 
$F_t$ to preserve $\nu$ is that the scattering map $\Theta$ is a %
diffeomorphism from $\scrS_-$ onto $\scrS_+$ which carries the volume measure
$|\vecv\cdot\vecw|\,d\!\vol_{\S_1^{d-1}}(\vecv)\,d\!\vol_{\S_1^{d-1}}(\vecw)$
on $\scrS_-$ to 
$(\vecv\cdot\vecw)\,d\!\vol_{\S_1^{d-1}}(\vecv)\,d\!\vol_{\S_1^{d-1}}(\vecw)$
on $\scrS_+$.
Maps with this property can be classified explicitly:
Define the functions $\vartheta_1,\vartheta_2:(-\frac{\pi}2,\frac{\pi}2)\to\R$ through
\begin{align} \label{TAUJDEF}
\Theta_j\bigl(\vece_1,-(\cos\varphi)\vece_1+(\sin\varphi)\vece_2\bigr)
=-(\cos\vartheta_j(\varphi))\vece_1+(\sin\vartheta_j(\varphi))\vece_2.
\end{align}
In view of (iii) we may then assume $\vartheta_j(0)=0$, and take $\vartheta_1$, $\vartheta_2$ to 
be continuous. (Then %
$\vartheta_1,\vartheta_2$ are both odd and $\C^1$,
and $\vartheta_1$ is a $\C^1$ diffeomorphism from
$(-\frac{\pi}2,\frac{\pi}2)$ onto %
\mbox{$(B_\Theta-\pi,\pi-B_\Theta)$}.)
In this notation, one checks by a computation that $\Theta$ carries
$|\vecv\cdot\vecw|\,d\!\vol_{\S_1^{d-1}}(\vecv)\,d\!\vol_{\S_1^{d-1}}(\vecw)$
to $(\vecv\cdot\vecw)\,d\!\vol_{\S_1^{d-1}}(\vecv)\,d\!\vol_{\S_1^{d-1}}(\vecw)$
if and only if, for all $\varphi\in(-\sfrac{\pi}2,\sfrac{\pi}2)\setminus\{0\}$,
\begin{align}
\Bigl|\frac{\sin(\vartheta_2(\varphi)-\vartheta_1(\varphi))}{\sin\varphi}\Bigr|^{d-2}
\cdot\Bigl|\frac{\cos(\vartheta_2(\varphi)-\vartheta_1(\varphi))}{\cos\varphi}\Bigr|
\cdot\bigl|\vartheta_2'(\varphi)-\vartheta_1'(\varphi)\bigr|=1.
\end{align}
This is seen to hold if and only if  
$\vartheta_2(\varphi)=\vartheta_1(\varphi)-\varphi$ for all 
$\varphi\in (-\sfrac{\pi}2,\sfrac{\pi}2)$
or $\vartheta_2(\varphi)=\vartheta_1(\varphi)+\varphi$ for all 
$\varphi\in (-\sfrac{\pi}2,\sfrac{\pi}2)$.
(In physical terms, this
reflects the preservation of the angular momentum $\vecw\wedge\vecv$, or 
its reversal, respectively.)
Translating this condition in terms of $\vecbeta^\pm$, we conclude that
\textit{$F_t$ preserves the Liouville measure if and only if}
\begin{align} \label{PLBETACRITERION}
\vecbeta_{\vecv_1}^+(\vecv_2)\equiv\vecbeta_{-\vecv_2}^-(-\vecv_1)
\qquad \text{or} \qquad
\vecbeta_{\vecv_1}^+(\vecv_2)\equiv \vecbeta_{\vecv_2}^-(\vecv_1)R_{\{\vecv_2\}^\perp},
\end{align}
where $R_{\{\vecv_2\}^\perp}\in \O(d)$ denotes orthogonal reflection in the 
hyperplane $\{\vecv_2\}^\perp\subset\R^d$.
\end{remark}

\subsection{Velocity after the first impact}

\begin{thm}\label{exactpos2-1hit}
Let $\lambda$ be a Borel probability measure on $\S_1^{d-1}$ absolutely 
continuous with respect to $\vol_{\S_1^{d-1}}$. For any bounded continuous function $f:\S_1^{d-1}\times \R_{>0} \times \S_1^{d-1} \to \RR$,
\begin{multline} \label{exactpos2eq-1hit}
\lim_{\rho\to 0}  \int_{\S_1^{d-1}} f\big(\vecv_0, \rho^{d-1} \tau_1(\vecq_{\rho,\vecbeta}(\vecv_0),\vecv_0;\rho), 
\vecv_1(\vecq_{\rho,\vecbeta}(\vecv_0),\vecv_0;\rho)\big) d\lambda(\vecv_0) \\
=%
\int_{\S_1^{d-1}} \int_{\R_{>0}} \int_{\S_1^{d-1}} 
f\big(\vecv_0,\xi,\vecv_1\big) 
p_{\vecalf,\vecbeta}(\vecv_0,\xi,\vecv_1) \, d\lambda(\vecv_0)\,
d\xi \, d\!\vol_{\S_1^{d-1}}(\vecv_1),
\end{multline}
where the probability density $p_{\vecalf,\vecbeta}$ is defined by
\begin{equation} \label{exactpos2-1hit-tpdef}
p_{\vecalf,\vecbeta}(\vecv_0,\xi,\vecv_1)\,d\!\vol_{\S_1^{d-1}}(\vecv_1)
=
\begin{cases}
\Phi_\vecalf\bigl(\xi,\vecw,(\vecbeta(\vecv_0)K(\vecv_0))_\perp\bigr)\,
\, d\vecw & \text{if $\vecv_1\in\scrV_{\vecv_0}$}\\
0 & \text{if $\vecv_1\notin\scrV_{\vecv_0}$,}
\end{cases}
\end{equation}
with
\begin{equation} \label{exactpos2-1hit-subst}
\vecw %
=-\vecbeta_{\vece_1}^-\bigl(\vecv_1K(\vecv_0)\bigr)_\perp
\in\{0\}\times\scrB_1^{d-1}.
\end{equation}
\end{thm}

\begin{remark} \label{PALFBETCONTREM}
The density $p_{\vecalf,\vecbeta}(\vecv_0,\xi,\vecv_1)$ is independent 
of the choice of the function $K(\vecv_0)$,
since $\Phi_\vecalf(\xi,\vecw,\vecz)$ only depends on the four real numbers 
$\xi$, $\|\vecw\|$, $\|\vecz\|$, $\vecw\cdot\vecz$,
cf.\ Remark \ref{PALFBETSPECREM}.
It also follows from Remark \ref{PALFBETSPECREM} and Remark \ref{XI0REM}
that if $\vecalf\notin\Q^d$ then
$p_{\vecalf,\vecbeta}(\vecv_0,\xi,\vecv_1)=:p(\vecv_0,\xi,\vecv_1)$ is independent of $\vecalf,\vecbeta$, and is continuous 
at each point 
$(\vecv_0,\xi,\vecv_1)\in\S_1^{d-1}\times\R_{\geq 0}\times\S_1^{d-1}$
with $\vecv_1\in\scrV_{\vecv_0}$.
The same continuity statement also holds for 
$p_{\vecalf,\vecbeta}(\vecv_0,\xi,\vecv_1)$ 
if $\vecalf\in\Q^d$ and $\sup\|\vecbeta\|\leq 1$,
except possibly when $d=2$, $\xi>0$ and
$\vecbeta_{\vece_1}^-\bigl(\vecv_1K(\vecv_0)\bigr)_\perp=
(\vecbeta(\vecv_0)K(\vecv_0))_\perp$.
\end{remark}
\begin{remark} \label{PALFBETEXPLICITREM}
The relationship between
$p_{\vecalf,\vecbeta}(\vecv_0,\xi,\vecv_1)$ and 
$\Phi_\vecalf\bigl(\xi,\vecw,\vecz)$ when $\vecv_1\in\scrV_{\vecv_0}$
can be expressed more explicitly as
\begin{align} \label{PALFBETEXPLICITREMFORMULA}
p_{\vecalf,\vecbeta}(\vecv_0,\xi,\vecv_1)=
J(\vecv_0,\vecv_1) \,
\Phi_\vecalf\Bigl(\xi,-\vecbeta_{\vece_1}^-\bigl(\vecv_1K(\vecv_0)\bigr)_\perp
,(\vecbeta(\vecv_0)K(\vecv_0))_\perp\Bigr),
\end{align}
where the factor $J(\vecv_0,\vecv_1)>0$ is a function of
$\varphi=\varphi(\vecv_1,-\vecv_2)$ given by 
\begin{align} \label{JFORMULA}
J(\vecv_0,\vecv_1)=\begin{cases}
\bigl| \frac{\sin\omega(\varphi)}{\sin\varphi}\bigr|^{d-2}
\, \bigl|\omega'(\varphi)\bigr| \cos\omega(\varphi)
& \text{if }\: \varphi>0
\\
|\omega'(0)|^{d-1} & \text{if }\:\varphi=0,\end{cases}
\end{align}
with $\omega=\vartheta_1^{-1}:
(B_\Theta-\pi,\pi-B_\Theta)\to (-\frac\pi 2,\frac\pi 2)$ being
the inverse of the map $\vartheta_1$ 
defined in Remark~\ref{PRESERVELIOUVILLEREMARK}.
In particular in the case of specular reflection \eqref{LORENTZSCATTERING}
we have $\vartheta_1^{-1}(\varphi)=\varphi/2$, and we recover the formula
\cite[(4.24)]{partI}.
\end{remark}

\begin{proof}[Proof of Theorem \ref{exactpos2-1hit}]
The proof follows exactly the same steps as the proof of Theorem 4.8 
in \cite[Sec.\ 9.3]{partI}: The left hand side of \eqref{exactpos2eq-1hit} equals
\begin{align}
\lim_{\rho\to 0}  \int_{\S_1^{d-1}} g\big(\vecv_0, \rho^{d-1} \tau_1(\vecq_{\rho,\vecbeta}(\vecv_0),\vecv_0;\rho), 
\vecw_1(\vecq_{\rho,\vecbeta}(\vecv_0),\vecv_0;\rho)\big)\, d\lambda(\vecv_0),
\end{align}
where $g(\vecv_0,\xi,\vecw_1)=f(\vecv_0,\xi,\Theta_1(\vecv_0,\vecw_1))$.
Using \cite[Cor.\ 4.7]{partI} 
(which in fact is a corollary of Theorem \ref{exactpos1}
in the present paper) we obtain
\begin{align}\notag
&=\int_{\HS}\int_{\R_{>0}}\int_{\S_1^{d-1}}
f(\vecv_0,\xi,\Theta_1(\vecv_0,-\vecomega K(\vecv_0)^{-1}))
\\ \label{exactpos2-1hit-proof1}
& \hspace{100pt}
\times\Phi_\vecalf(\xi,\vecomega_\perp,(\vecbeta(\vecv_0)K(\vecv_0))_\perp)\,
\omega_1\,d\lambda(\vecv_0)\,d\xi\,d\!\vol_{\S_1^{d-1}}(\vecomega).
\end{align}
Now change the order of integration by moving the integral over
$\vecomega\in\HS$ to the innermost position, and then apply the variable
substitution $\vecv_1=\Theta_1(\vece_1,-\vecomega)K(\vecv_0)^{-1}
=\Theta_1(\vecv_0,-\vecomega K(\vecv_0)^{-1})$ 
in the innermost integral; note that this gives
a diffeomorphism $\vecomega\mapsto\vecv_1$ from $\HS$ onto 
$\scrV_{\vecv_0}$,
with the inverse map given by 
$\vecomega=-\vecbeta^-_{\vece_1}(\vecv_1 K(\vecv_0))
=-\vecbeta^-_{\vecv_0}(\vecv_1)K(\vecv_0)$.
Recalling \eqref{exactpos2-1hit-subst} we then see that 
\eqref{exactpos2-1hit-proof1}
equals the right hand side of \eqref{exactpos2eq-1hit}, and we are done.
\end{proof}

\subsection{A uniform version of Theorem \ref{exactpos2-1hit}}
\label{exactpos2-1hit-unif-subsec}

In order to prove Theorem \ref{pathwayThm} we need a version of
Theorem \ref{exactpos2-1hit} which is uniform over certain families of
$\vecbeta$'s, $f$'s and $\lambda$'s.

Given any subset $\scrW\subset\S_1^{d-1}$ we let 
$\partial_\ve\scrW\subset\S_1^{d-1}$
denote the $\ve$-neighborhood of its boundary, i.e.\
\begin{align}
\partial_\ve\scrW:=\bigl\{\vecv\in\S_1^{d-1} \col
\exists \vecw\in\partial\scrW: \: \varphi(\vecv,\vecw)<\ve\bigr\}.
\end{align}

\begin{defn}
A family $F$ of Borel subsets of $\S_1^{d-1}$ is called
\textit{equismooth} if for every $\delta>0$ there is some $\ve>0$ such that
$\vol_{\S_1^{d-1}}(\partial_\ve\scrW)<\delta$ for all $\scrW\in F$.
\end{defn}

\begin{defn}
A family $F$ of measures on $\S_1^{d-1}$ is called
\textit{equismooth} if there exist an equicontinuous and uniformly bounded
family $F'$ of functions from $\S_1^{d-1}$ to $\R_{\geq 0}$ and
an equismooth family $F''$ of Borel subsets of $\S_1^{d-1}$, such that
each $\mu\in F$ can be expressed as
$\mu=(g\cdot \vol_{\S_1^{d-1}})_{|\scrW}$ for some $g\in F'$, $\scrW\in F''$.
\end{defn}

\begin{thm} \label{exactpos2-1hit-unif}
Fix a lattice $\scrL=\Z^dM_0$ and a point $\vecq\in\R^d$, and write $\vecalf=-\vecq M_0^{-1}$.
Let $F_1$ be an equismooth family of probability measures on $\S_1^{d-1}$,
let $F_2$ be a uniformly bounded and equicontinuous family
of functions $f:\S_1^{d-1}\times\R_{>0}\times\S_1^{d-1}\to\R$, 
and let $F_3$ be a uniformly bounded and equicontinuous family
of functions $\vecbeta:\S_1^{d-1}\to\R^d$ such that if $\vecq\in\scrL$ then
$(\vecbeta(\vecv)+\R_{>0}\vecv)\cap\scrB_1^d=\emptyset$
for all $\vecbeta\in F_3$, $\vecv\in\S_1^{d-1}$.
Then the limit relation %
\begin{multline} \label{exactpos2eq-1hit-repeted}
\lim_{\rho\to 0}  \int_{\S_1^{d-1}} f\big(\vecv_0, \rho^{d-1} \tau_1(\vecq_{\rho,\vecbeta}(\vecv_0),\vecv_0;\rho), 
\vecv_1(\vecq_{\rho,\vecbeta}(\vecv_0),\vecv_0;\rho)\big) d\lambda(\vecv_0) \\
=\int_{\S_1^{d-1}} \int_{\R_{>0}} \int_{\S_1^{d-1}} f\big(\vecv_0,\xi,\vecv_1\big) 
p_{\vecalf,\vecbeta}(\vecv_0,\xi,\vecv_1) \, d\lambda(\vecv_0)\, d\xi 
d\!\vol_{\S_1^{d-1}}(\vecv_1)
\end{multline}
holds uniformly with respect to all 
$\lambda\in F_1$, $f\in F_2$, $\vecbeta\in F_3$.
\end{thm}

This uniform version of Theorem \ref{exactpos2-1hit} will actually be derived
as a corollary of Theorem~\ref{exactpos2-1hit}.
We first need a lemma. Let $\fU_\eta$ denote a small neighborhood
of the boundary in $\HS$:
\begin{align} \label{FUEPSDEF}
\fU_\eta:=\bigl\{\vecw\in\HS\col\varphi(\vecw,\vece_1)>\sfrac{\pi}2-\eta\bigr\}
\qquad (0<\eta<1).
\end{align}

\begin{lem} \label{EXCSET1LEMMA}
Given $B>0$ there exists some number $\rho_0(B,\scrL,\vecq)>0$ such that
for every $0<\rho<\rho_0(B,\scrL,\vecq)$, 
$\vecv\in\S_1^{d-1}$, $K\in\SO(d)$, $\vecbeta,\vecbeta'\in\R^d$ and
$0<\eta<\frac 1{10}$ such that
\begin{equation}
	\vecv K=\vece_1, \qquad
	\|\vecbeta\|\leq B, \qquad 
	\|\vecbeta-\vecbeta'\|\leq\eta, \qquad
	\vecbeta\cdot\vecv\geq 2, \qquad 
	\vecbeta'\cdot\vecv\geq 2,
\end{equation}
one of the following holds:
\begin{enumerate}
	\item[(i)] $\tau_1(\vecq+\rho\vecbeta,\vecv,\rho)=
\tau_1(\vecq+\rho\vecbeta',\vecv,\rho)=\infty$,
	\item[(ii)] $-\vecw_1(\vecq+\rho\vecbeta,\vecv,(1+\eta)\rho)K\in\fU_{3\sqrt\eta}$,
	\item[(iii)] $\bigl |\tau_1(\vecq+\rho\vecbeta',\vecv;\rho)
-\tau_1(\vecq+\rho\vecbeta,\vecv;\rho)\bigr |<3\rho\sqrt\eta$ \\ and 
$\bigl\|\vecw_1(\vecq+\rho\vecbeta',\vecv;\rho)
-\vecw_1(\vecq+\rho\vecbeta,\vecv;\rho)\bigr\|<\sqrt{2\eta}$.
\end{enumerate}
\end{lem}

\begin{remark}
The conditions $\vecbeta\cdot\vecv\geq 2$, $\vecbeta'\cdot\vecv\geq 2$
are only needed in the case $\vecq\in\scrL$, so as to guarantee that the 
rays $\vecq+\rho\vecbeta+\R_{>0}\vecv$ and $\vecq+\rho\vecbeta'+\R_{>0}\vecv$ 
lie outside the ball $\vecq+\scrB_\rho^d$, and also outside the ball
$\vecq+\scrB_{(1+\eta)\rho}^d$.
\end{remark}

\begin{proof}[Proof of Lemma \ref{EXCSET1LEMMA}]
We choose $\rho_0=\rho_0(B,\scrL,\vecq)>0$ so small that all the balls
$\vecm+\scrB_{2\rho_0}^d$ ($\vecm\in\scrL\setminus\{\vecq\}$)
have disjoint closures, and are each disjoint from
$\vecq+\scrB^d_{(B+1)\rho_0}$. 
Now fix any $\rho,\vecv,K,\vecbeta,\vecbeta',\eta$
as in the statement of the lemma.

After an auxiliary rotation we may assume $\vecv=\vece_1$ and $K=1_d$.
Write $\tau=\tau_1(\vecq+\rho\vecbeta,\vece_1;\rho)$ and
$\vecw=\vecw_1(\vecq+\rho\vecbeta,\vece_1;\rho)$, let
$\tau'$ and $\vecw'$ be the corresponding data for $\vecbeta'$,
and write $\trho=(1+\eta)\rho$, 
$\ttau=\tau_1(\vecq+\rho\vecbeta,\vece_1,\trho)$ and
$\tw=\vecw_1(\vecq+\rho\vecbeta,\vece_1,\trho)$.
By our assumption on $\rho_0$ we see that $\tau,\tau',\ttau$
are well-defined numbers in $\R_{>0}\cup\{\infty\}$.

For any ray $\gamma\subset\R^d$ and any point $\vecm\in\R^d$ we
write $\delta(\gamma,\vecm)=\inf_{\vecp\in\gamma}\|\vecp-\vecm\|$
for the distance between $\gamma$ and $\vecm$.
If $\ttau=\infty$ then 
$\delta(\vecq+\rho\vecbeta+\R_{>0}\vece_1,\vecm)\geq\trho$ for all 
$\vecm\in\scrL$, and thus also
$\delta(\vecq+\rho\vecbeta'+\R_{>0}\vece_1,\vecm)\geq\trho-\rho\bigl\|\vecbeta
-\vecbeta'\bigr\|\geq\rho$ for all $\vecm\in\scrL$, so that
$\tau=\tau'=\infty$.
Hence from now on we may assume $\ttau<\infty$.

By the definition of $\ttau$, $\tw$, there is a unique 
$\tm\in\scrL\setminus\{\vecq\}$ such that 
\begin{align} \label{EXCSETTMREL}
\vecq+\rho\vecbeta+\ttau\vece_1=\tm+\trho\tw, %
\end{align}
and the ray $\vecq+\rho\vecbeta+\R_{>0}\vece_1$ does not intersect
any previous ball, i.e.\ 
\begin{align} \label{EXCSETTMREL2}
\forall \vecm\in\scrL: %
\qquad \vecm\cdot\vece_1<\tm\cdot\vece_1 \Longrightarrow
\delta(\vecq+\rho\vecbeta+\R_{>0}\vece_1,\vecm)\geq\trho. %
\end{align}

Now assume $-\tw\notin\fU_{3\sqrt\eta}$.
We will prove that (iii) holds.
From \eqref{EXCSETTMREL} we get
\begin{equation}
	\delta(\vecq+\rho\vecbeta+\R_{>0}\vece_1,\tm)=\trho\|\tw_\perp\|
\leq(1+\eta)\cos(3\sqrt\eta)\rho<(1-\eta)\rho
\end{equation}
and thus also
\begin{equation}
	\delta(\vecq+\rho\vecbeta'+\R_{>0}\vece_1,\tm)<\rho .
\end{equation}
Hence both the rays $\vecq+\rho\vecbeta+\R_{>0}\vece_1$ and
$\vecq+\rho\vecbeta'+\R_{>0}\vece_1$ intersect the ball $\tm+\scrB_\rho^d$.
Note also that 
$\delta(\vecq+\rho\vecbeta+\R_{>0}\vece_1,\vecm)>\trho$ and thus
$\delta(\vecq+\rho\vecbeta'+\R_{>0}\vece_1,\vecm)>\rho$ for all
$\vecm\in\scrL$ with $\vecm\cdot\vece_1<\tm\cdot\vece_1$,
by \eqref{EXCSETTMREL2}.
Hence neither $\tau$ nor $\tau'$ can arise from any 
previous intersection, i.e.\ we must have
$\vecq+\rho\vecbeta+\tau\vece_1=\tm+\rho\vecw$ and
$\vecq+\rho\vecbeta'+\tau'\vece_1=\tm+\rho\vecw'$. This implies
\begin{equation}
	\bigl\|(\vecw'-\vecw)_\perp\bigr\|
=\|(\vecbeta'-\vecbeta)_\perp\|\leq\eta
\end{equation}
and thus,
using $-\vecw'$, $-\vecw\in\HS$, we conclude
$\bigl\|\vecw'-\vecw\bigr\|<\sqrt{2\eta}$. 
Finally from
$(\tau-\tau')\vece_1=\rho(\vecw-\vecw')-\rho(\vecbeta-\vecbeta')$
we conclude 
\begin{equation}
	|\tau-\tau'|< \rho(\sqrt{2\eta}+\eta)<3\rho\sqrt{\eta} .
\end{equation}
Hence (iii) holds.
\end{proof}

\begin{proof}[Proof of Theorem \ref{exactpos2-1hit-unif}.]

\textit{Step 1: Proof in the case that $F_1$ and $F_2$ are singleton sets,
$F_1=\{\lambda\}$ and $F_2=\{f\}$, and furthermore $f$ is uniformly
continuous.}
Let $\ve>0$ be given. Let $R=1+\sup_{\S_1\times\R_{>0}\times\S_1} |f|$.
Set $B=3+\sup_{\vecbeta\in F_3}\sup_{\S_1^{d-1}} \|\vecbeta\|$ and
\begin{align}
F_3'=\bigl\{\vecv\mapsto\vecbeta(\vecv)+B\vecv\col\vecbeta\in F_3\bigr\}.
\end{align}
Then $\vecbeta_1(\vecv)\cdot\vecv\geq 3$ and $\|\vecbeta_1(\vecv)\|< 2B$
hold for all $\vecbeta_1\in F_3'$ and $\vecv\in\S_1^{d-1}$.

Because of Remark \ref{PALFBETSPECREM} and 
$0\leq\Phi_{\vecalf}(\xi,\vecw,\vecz)\leq 1$,
the following limit holds uniformly with respect to
all $\vecz\in\{0\}\times\R^{d-1}$:
\begin{align} \label{STEP1DELTAFACT}
\lim_{\delta\to 0} \Bigl(\int_0^\delta+\int_{\delta^{-1}}^\infty\Bigr) 
\int_{\{0\}\times\scrB_1^{d-1}}
\Phi_{\vecalf}(\xi,\vecw,\vecz)\,d\vecw\,d\xi=0.
\end{align}
Furthermore, by \cite[Lemma 8.8]{partI},
for any given $\delta>0$ there is some $0<\eta<\frac 1{10}$ so small that
\begin{align} \label{STEP1VEPRIMECHOICE}
\int_{\{0\}\times\scrB_1^{d-1}} \bigl | \Phi_\vecalf(\xi,\vecw,\vecz_1)-
\Phi_\vecalf(\xi,\vecw,\vecz_2)\bigr | \, d\vecw<\frac{\delta\ve}{R}
\end{align}
holds for all $\xi\in [T^{-1},T]$ and all
$\vecz_1,\vecz_2\in\{0\}\times\scrB_{2B}^{d-1}$
with $\|\vecz_1-\vecz_2\|\leq \eta$.
As a consequence of \eqref{STEP1DELTAFACT}
and \eqref{STEP1VEPRIMECHOICE}, we can now fix $\eta>0$ so small that,
for any two (measurable) functions
$\vecz_1,\vecz_2:\S_1^{d-1}\to\{0\}\times\scrB_{2B}^{d-1}$
satisfying $\|\vecz_1(\vecv)-\vecz_2(\vecv)\|\leq \eta$ 
($\forall\vecv\in\S_1^{d-1}$), we have
\begin{align} \label{STEP1VEPRIMECHOICEIMPL}
\int_0^\infty \int_{\{0\}\times\scrB^{d-1}_1} \int_{\S_1^{d-1}} &
\Bigl | \Phi_\vecalf\bigl(\xi,\vecw,\vecz_1(\vecv)\bigr)
-\Phi_\vecalf\bigl(\xi,\vecw,\vecz_2(\vecv)\bigr)\Bigr |
\,d\lambda(\vecv)\,d\vecw\,d\xi<\frac{2\ve}R.
\end{align}

Again using \eqref{STEP1DELTAFACT}
(and $\lim_{\eta\to 0}\vol_{\R^{d-1}} \bigl((\fU_{3\sqrt\eta})_\perp\bigr)=0$),
we see that by possibly further shrinking $\eta$, we may also assume that
\begin{align} \label{STEP1TCHOICEIMPL}
\int_0^\infty \int_{(\fU_{3\sqrt\eta})_\perp} \int_{\S_1^{d-1}} 
\Phi_{\vecalf}\bigl(\xi,\vecw,
(\vecbeta(\vecv)K(\vecv))_\perp\bigr)\,d\lambda(\vecv)\,d\vecw\,d\xi
<\frac{2\ve}{R}
\end{align}
holds for all (measurable) functions $\vecbeta:\S_1^{d-1}\to\R^d$.

Since $f$ is uniformly continuous there is some $\delta>0$ such that
$\bigl | f(\vecv_0,\xi,\vecv_1)-f(\vecv_0,\xi',\vecv_1')\bigr |<\ve$ holds
whenever $|\xi-\xi'|<\delta$ and $\|\vecv_1-\vecv_1'\|<\delta$.
Furthermore, our assumptions on the scattering function $\Theta$ imply that 
$\Theta_1(\vece_1,\vecw)$ is uniformly continuous in $\vecw\in -\HS$;
hence by possibly further shrinking $\eta$ we may assume that
\begin{align} \label{THETA1UNIFCONT}
\forall \vecw,\vecw'\in-\HS: \quad
\|\vecw-\vecw'\|<\sqrt{2\eta} \Longrightarrow
\|\Theta_1(\vece_1,\vecw)-\Theta_1(\vece_1,\vecw')\|<\delta.
\end{align}

Since the family $F_3'$ is uniformly bounded and equicontinuous, there is a 
finite set $F_3''$ of continuous functions 
$\vecbeta_2:\S_1^{d-1}\to\scrB_{2B}^d$ such
that for each $\vecbeta_1\in F_3'$ there is some $\vecbeta_2\in F_3''$
with $\sup_{\S_1^{d-1}} \|\vecbeta_1-\vecbeta_2\|<\eta$.

Given $\vecbeta_2\in F_3''$, let us write $\tbe_2$ for the function
$\tbe_2(\vecv):=(1+\eta)^{-1}{\vecbeta_2}(\vecv)$.
Now since $F_3''$ is finite, 
by Theorem~\ref{exactpos2-1hit} and Theorem \ref{exactpos1} combined
with \eqref{STEP1TCHOICEIMPL}
there is some $\rho_0>0$ such that the following two
inequalities hold for all $\rho\in (0,2\rho_0)$ and all $\vecbeta_2\in F_3''$:
\begin{align} \label{STEP1ASSONRHO1}
& \biggl |
\int_{\S_1^{d-1}} f\big(\vecv_0, \rho^{d-1} \tau_1(\vecq_{\rho,{\vecbeta_2}}(\vecv_0),\vecv_0;\rho), 
\vecv_1(\vecq_{\rho,{\vecbeta_2}}(\vecv_0),\vecv_0;\rho)\big) 
\, d\lambda(\vecv_0) 
\\ \notag
& \hspace{50pt}
-\int_{\S_1^{d-1}} \int_{\R_{>0}} \int_{\S_1^{d-1}} f\big(\vecv_0,\xi,\vecv_1\big) 
p_{\vecalf,{\vecbeta_2}}(\vecv_0,\xi,\vecv_1) \, d\lambda(\vecv_0)\, d\xi 
d\!\vol_{\S_1^{d-1}}(\vecv_1)\biggr | <\ve;
\\ \label{STEP1ASSONRHO2}
& \lambda\bigl(\bigl\{\vecv_0\in\S_1^{d-1}\col 
-\vecw_1(\vecq_{\rho,\tbe_2}(\vecv_0),\vecv_0;\rho)K(\vecv_0)\in\fU_{3\sqrt\eta}
\bigr\}\bigr)<\frac{3\ve}{R}.
\end{align}
By possibly further shrinking $\rho_0$ we may assume that
$(3\sqrt{\eta}+B)\rho_0^{d}<\delta$, and
$\rho_0<\rho_0(2B,\scrL,\vecq)$ (the constant from Lemma \ref{EXCSET1LEMMA}),
and that every ball $\vecm+\scrB_{\rho_0}^d$
with $\vecm\in\scrL\setminus\{\vecq\}$ lies outside $\vecq+\scrB_{2B\rho_0}^d$.

We now claim that for all $\rho\in(0,\rho_0)$ and all $\vecbeta\in F_3$
we have
\begin{align} \label{STEP1CLAIM}
& \Bigl |
\int_{\S_1^{d-1}} f\big(\vecv_0, \rho^{d-1} \tau_1(\vecq_{\rho,\vecbeta}(\vecv_0),\vecv_0;\rho), 
\vecv_1(\vecq_{\rho,\vecbeta}(\vecv_0),\vecv_0;\rho)\big) d\lambda(\vecv_0) 
\\ \notag
& \hspace{50pt}
-\int_{\S_1^{d-1}} \int_{\R_{>0}} \int_{\S_1^{d-1}} f\big(\vecv_0,\xi,\vecv_1\big) 
p_{\vecalf,\vecbeta}(\vecv_0,\xi,\vecv_1) \, d\lambda(\vecv_0)\, d\xi 
d\!\vol_{\S_1^{d-1}}(\vecv_1)\Bigr | <10\ve.
\end{align}
This will complete the proof of Step 1, since $\ve>0$ was arbitrary.

To prove \eqref{STEP1CLAIM}, let $\rho\in(0,\rho_0)$ and 
$\vecbeta\in F_3$
be given. Set $\vecbeta_1(\vecv):\equiv\vecbeta(\vecv)+B\vecv$;
then $\vecbeta_1\in F_3'$, and hence there is some $\vecbeta_2\in F_3''$
such that $\sup_{\S_1^{d-1}} \|\vecbeta_1-\vecbeta_2\|<\eta$.
Now for each $\vecv_0\in\S_1^{d-1}$, Lemma \ref{EXCSET1LEMMA} applies to 
$\rho$, $\vecv_0$, $K(\vecv_0)$, $\vecbeta_2(\vecv_0)$, $\vecbeta_1(\vecv_0)$,
$\eta$,
and hence either
\begin{enumerate}
	\item[(i)] $\tau_1(\vecq_{\rho,\vecbeta_2}(\vecv_0),\vecv_0,\rho)=
\tau_1(\vecq_{\rho,\vecbeta_1}(\vecv_0),\vecv_0,\rho)=\infty$, or
	\item[(ii)] $-\vecw_1(\vecq_{\trho,\tbe_2}(\vecv_0),\vecv_0,\trho)K(\vecv_0)
\in\fU_{3\sqrt\eta}$ with $\trho=\rho(1+\eta)$, or
	\item[(iii)] $\bigl |\tau_1(\vecq_{\rho,\vecbeta_1}(\vecv_0),\vecv_0;\rho)
-\tau_1(\vecq_{\rho,\vecbeta_2}(\vecv_0),\vecv_0;\rho)\bigr |<3\rho\sqrt{\eta}$ \\ and
$\bigl\|\vecw_1(\vecq_{\rho,\vecbeta_1}(\vecv_0),\vecv_0;\rho)
-\vecw_1(\vecq_{\rho,\vecbeta_2}(\vecv_0),\vecv_0;\rho)\bigr\|<\sqrt{2\eta}$.
\end{enumerate}
But 
\begin{equation}
	\tau_1(\vecq_{\rho,\vecbeta_1}(\vecv_0),\vecv_0;\rho)
=\tau_1(\vecq_{\rho,\vecbeta}(\vecv_0),\vecv_0;\rho)-B\rho
\end{equation}
and 
\begin{equation}
	\vecw_1(\vecq_{\rho,\vecbeta_1}(\vecv_0),\vecv_0;\rho)
=\vecw_1(\vecq_{\rho,\vecbeta}(\vecv_0),\vecv_0;\rho),
\end{equation}
since every ball $\vecm+\scrB_\rho^d$ ($\vecm\in\scrL\setminus\{\vecq\}$) 
lies outside $\vecq+\scrB_{2B\rho}^d$ and 
$(\vecbeta(\vecv_0)+\R_{>0}\vecv_0)\cap\scrB_1^d=\emptyset$ if $\vecq\in\scrL$.
Hence if (iii) holds then we have
(using our assumptions on $\rho_0$)
\begin{align}
\bigl | \rho^{d-1}\tau_1(\vecq_{\rho,\vecbeta_2}(\vecv_0),\vecv_0;\rho)
-\rho^{d-1}\tau_1(\vecq_{\rho,\vecbeta}(\vecv_0),\vecv_0;\rho) \bigr |
<(3\sqrt{\eta}+B)\rho^d<\delta
\end{align}
and also, via \eqref{THETA1UNIFCONT},
\begin{equation}
	\bigl\|\vecv_1(\vecq_{\rho,\vecbeta_2}(\vecv_0),\vecv_0;\rho)
-\vecv_1(\vecq_{\rho,\vecbeta}(\vecv_0),\vecv_0;\rho)\bigr\|<\delta .
\end{equation}
Therefore, by our choice of $\delta$,
\begin{align} \label{STEP1FDIFFSMALL}
& \Bigl | f\big(\vecv_0, \rho^{d-1} \tau_1(\vecq_{\rho,\vecbeta_2}(\vecv_0),\vecv_0;\rho), 
\vecv_1(\vecq_{\rho,\vecbeta_2}(\vecv_0),\vecv_0;\rho)\big)
\\ \notag
& \hspace{50pt}
-f\big(\vecv_0, \rho^{d-1} \tau_1(\vecq_{\rho,\vecbeta}(\vecv_0),\vecv_0;\rho),
\vecv_1(\vecq_{\rho,\vecbeta}(\vecv_0),\vecv_0;\rho)\big)\Bigr |<\ve.
\end{align}

Regarding the other possibilities (i) and (ii), the set
$\{\vecv_0\in\S_1^{d-1}\col
\tau_1(\vecq_{\rho,\vecbeta_1}(\vecv_0),\vecv_0,\rho)=\infty\}$ 
has measure zero with respect to
$\vol_{\S_1^{d-1}}$ (see Sec.\ \ref{firstcoll}), and hence also with respect 
to $\lambda$. Furthermore we have
\begin{equation}
	\lambda(\{\vecv_0\in\S_1^{d-1}\col
-\vecw_1(\vecq_{\trho,\tbe_2}(\vecv_0),\vecv_0,\trho)K(\vecv_0)\in\fU_{3\sqrt\eta}\})
<\frac{3\ve}R
\end{equation}
using \eqref{STEP1ASSONRHO2} and $\trho<2\rho_0$.
Hence, since the difference in \eqref{STEP1FDIFFSMALL} is \textit{always}
$<2R$, we conclude
\begin{align} \label{STEP1INTFDIFFSMALL}
& \Bigl |
\int_{\S_1^{d-1}} f\big(\vecv_0, \rho^{d-1} \tau_1(\vecq_{\rho,{\vecbeta}}(\vecv_0),\vecv_0;\rho), 
\vecv_1(\vecq_{\rho,{\vecbeta}}(\vecv_0),\vecv_0;\rho)\big) d\lambda(\vecv_0) 
\\ \notag
& \hspace{50pt}
-\int_{\S_1^{d-1}} f\big(\vecv_0, \rho^{d-1} \tau_1(\vecq_{\rho,{\vecbeta_2}}(\vecv_0),\vecv_0;\rho), 
\vecv_1(\vecq_{\rho,{\vecbeta_2}}(\vecv_0),\vecv_0;\rho)\big) d\lambda(\vecv_0) 
\Bigr | <7\ve.
\end{align}
Using \eqref{STEP1VEPRIMECHOICEIMPL} and the definition of 
$p_{\vecalf,\vecbeta}$ in Theorem \ref{exactpos2-1hit} we also get
\begin{align} \label{STEP1PABDIFFSMALL}
\int_{\S_1^{d-1}} \int_{\R_{>0}} \int_{\S_1^{d-1}}
\Bigl |p_{\vecalf,{\vecbeta_2}}(\vecv_0,\xi,\vecv_1)-
p_{\vecalf,{\vecbeta}}(\vecv_0,\xi,\vecv_1)\Bigr | 
\,d\lambda(\vecv_0)\, d\xi \, d\!\vol_{\S_1^{d-1}}(\vecv_1)<\frac{2\ve}R.
\end{align}
Combining \eqref{STEP1ASSONRHO1}, %
\eqref{STEP1INTFDIFFSMALL} and \eqref{STEP1PABDIFFSMALL},
we conclude that \eqref{STEP1CLAIM} holds, thus completing the proof of Step 1.

\hspace{5pt}

\textit{Step 2: Proof of Theorem \ref{exactpos2-1hit-unif} in the case
that $F_1$ is a singleton set, $F_1=\{\lambda\}$, 
and $F_2$ has the property that $\cup_{f\in F_2} \supp(f)$ is compact.} 
In this case, since the family $F_2$ is also uniformly bounded
and equicontinuous, for any given $\ve>0$ there exists a finite
subfamily $F_2'\subset F_2$ such that for every $f\in F_2$ there is some
$f_0\in F_2'$ with $\sup_{\S_1\times\R_{>0}\times\S_1} |f-f_0|<\ve$.
Hence the desired result follows from Step 1 by a standard approximation 
argument.

\hspace{5pt}

\textit{Step 3: Proof of Theorem \ref{exactpos2-1hit-unif} in the case
that $F_1$ is a singleton set, $F_1=\{\lambda\}$.}
We first prove that the following limit holds uniformly with respect to
all $\vecbeta\in F_3$:
\begin{align} \label{STEP3NB}
\lim_{(\rho,\delta)\to (0,0)}\lambda\bigl(\bigl\{\vecv_0\in\S_1^{d-1}\col
\rho^{d-1}\tau_1(\vecq_{\rho,\vecbeta}(\vecv_0),\vecv_0;\rho)
\notin [\delta,\delta^{-1}] \bigr\}\bigr)=0.
\end{align}
To prove this,
for each $0<\delta<\frac 12$ we fix $c_\delta:\R_{>0}\to [0,1]$ to be some
continuous function with 
$\chi_{[2\delta,(2\delta)^{-1}]}\leq c_\delta\leq\chi_{[\delta,\delta^{-1}]}$,
and view $c_\delta$ as a function on $\S_1^{d-1}\times\R_{>0}\times\S_1^{d-1}$
via projection onto the second component.
Applying Step 1 for the families
$F_1$, $F_2'=\{c_\delta\}$, $F_3$, one proves that for any
$\ve,\delta>0$ there exists some $\rho_0>0$ such that
\begin{align} 
& \lambda\bigl(\bigl\{\vecv_0\in\S_1^{d-1}\col
\rho^{d-1}\tau_1(\vecq_{\rho,\vecbeta}(\vecv_0),\vecv_0;\rho)
\notin [\delta,\delta^{-1}] \bigr\}\bigr)
\\ \notag
& \qquad <\ve+\bigl(\int_0^{2\delta}+\int_{(2\delta)^{-1}}^\infty\bigr)
\int_{\{0\}\times\scrB_1^{d-1}}\int_{\S_1^{d-1}}
\Phi_{\vecalf}(\xi,\vecw,(\vecbeta(\vecv)K(\vecv))_\perp)\,d\lambda(\vecv)
\,d\vecw\,d\xi
\end{align}
for all $\rho\in(0,\rho_0)$ and all $\vecbeta\in F_3$.
Thus \eqref{STEP3NB} follows using \eqref{STEP1DELTAFACT} and %
monotonicity in $\delta$.

Now note that for each fixed $\delta$, 
the family $F_2'':=\{c_\delta\cdot f\col f\in F_2\}$ is uniformly bounded and 
equicontinuous, and all functions in $F_2''$ have support contained in the
compact set $\S_1^{d-1}\times[\delta,\delta^{-1}]\times\S_1^{d-1}$.
Hence Step 2 applies to the families $F_1,F_2'',F_3$.
Now the desired claim follows upon letting $\delta\to 0$ and using
\eqref{STEP1DELTAFACT} and \eqref{STEP3NB} to control the error caused by 
replacing
$f$ by $c_\delta\cdot f$ in both sides of \eqref{exactpos2eq-1hit-repeted}.

\hspace{5pt}

\textit{Step 4: Proof of Theorem \ref{exactpos2-1hit-unif} in the 
general case.}
Since $F_1$ is equismooth there exist
an equicontinuous and uniformly bounded family $F'$ of functions 
from $\S_1^{d-1}$ to $\R_{\geq 0}$ and an equismooth family $F''$
of Borel subsets of $\S_1^{d-1}$ such that each $\lambda\in F_1$ can be 
expressed as $\lambda=(g\cdot \vol_{\S_1^{d-1}})_{|\scrW}$ with
$g\in F'$, $\scrW\in F''$.
Set $R=1+\sup_{f\in F_2}\sup_{\S_1\times\R_{>0}\times\S_1} |f|$ and
$S=1+\sup_{g\in F'}\sup_{\S_1^{d-1}} g$.

Let $\ve>0$ be given.
Fix $\eta>0$ so small that
\begin{align} \label{ETACHOICEREG}
\forall g\in F': \:\forall \vecv_0,\vecv_0'\in\S_1^{d-1}: \:
\varphi(\vecv_0,\vecv_0')<\eta \Longrightarrow
\|g(\vecv_0)-g(\vecv_0')\|<\frac{\ve}{R\vol(\S_1^{d-1})}
\end{align}
and 
\begin{align}
\forall\scrW\in F'':\:
\vol_{\S_1^{d-1}}(\partial_\eta\scrW)<\frac{\ve}{RS}.
\end{align}
Choose a partition $\S_1^{d-1}=\bigsqcup_{j=1}^n D_j$ of $\S_1^{d-1}$
into Borel subsets $D_1,\ldots,D_n$ which are pathwise connected, of
positive volume %
and each has diameter $<\eta$
(with respect the metric $\varphi$ on $\S_1^{d-1}$).
Let $\lambda_j$ be the probability measure
$\lambda_j=\vol_{\S_1^{d-1}}(D_j)^{-1}\bigl(\vol_{\S_1^{d-1}}\bigr)_{|D_j}$
for $j=1,\ldots,n$.
By Step 3 applied $n$ times, there exists some $\rho_0>0$ such that
for all $\rho\in(0,\rho_0)$, $j\in\{1,\ldots,n\}$, $f\in F_2$ and
$\vecbeta\in F_3$, we have
\begin{align} \label{STEP3ASSONRHO1}
\biggl |
\int_{\S_1^{d-1}} f\big(\vecv_0, \rho^{d-1} \tau_1(\vecq_{\rho,{\vecbeta}}(\vecv_0),\vecv_0;\rho), 
\vecv_1(\vecq_{\rho,{\vecbeta}}(\vecv_0),\vecv_0;\rho)\big) d\lambda_j(\vecv_0)
\hspace{100pt} &
\\ \notag
-\int_{\S_1^{d-1}} \int_{\R_{>0}} \int_{\S_1^{d-1}} f\big(\vecv_0,\xi,\vecv_1\big) 
p_{\vecalf,{\vecbeta}}(\vecv_0,\xi,\vecv_1) \, d\lambda_j(\vecv_0)\, d\xi 
d\!\vol_{\S_1^{d-1}}(\vecv_1)\biggr | <\frac{\ve}{S\vol(\S_1^{d-1})}.&
\end{align}

We now claim that for all $\rho\in(0,\rho_0)$, $\lambda\in F_1$,
$f\in F_2$ and $\vecbeta\in F_3$, we have
\begin{align} \label{STEP3CLAIM}
& \biggl |
\int_{\S_1^{d-1}} f\big(\vecv_0, \rho^{d-1} \tau_1(\vecq_{\rho,{\vecbeta}}(\vecv_0),\vecv_0;\rho), 
\vecv_1(\vecq_{\rho,{\vecbeta}}(\vecv_0),\vecv_0;\rho)\big) d\lambda(\vecv_0) 
\\ \notag
& \hspace{50pt}
-\int_{\S_1^{d-1}} \int_{\R_{>0}} \int_{\S_1^{d-1}} f\big(\vecv_0,\xi,\vecv_1\big) 
p_{\vecalf,{\vecbeta}}(\vecv_0,\xi,\vecv_1) \, d\lambda(\vecv_0)\, 
d\xi\,d\!\vol_{\S_1^{d-1}}(\vecv_1)\biggr | <5\ve.
\end{align}
This will complete the proof of %
Theorem \ref{exactpos2-1hit-unif},
since $\ve>0$ was arbitrary.

To prove \eqref{STEP3CLAIM}, let $\rho\in(0,\rho_0)$, $\lambda\in F_1$,
$f\in F_2$ and $\vecbeta\in F_3$ be given. Take $g\in F'$ and $\scrW\in F''$
such that $\lambda=(g\cdot \vol_{\S_1^{d-1}})_{|\scrW}$.
Set 
\begin{align}
M:=\{j\col D_j\subset \scrW\}, \qquad
M':=\{j\col D_j\not\subset \scrW \text{ and } D_j\not\subset 
\S_1^{d-1}\setminus\scrW\}.
\end{align}
Since $D_j$ is pathwise connected, for each $j\in M'$ there is
some point $\vecp\in D_j\cap\partial\scrW$. Therefore, and because $D_j$
has diameter $<\eta$, we have $D_j\subset\partial_\eta\scrW$.
Hence $\sum_{j\in M'}\vol_{\S_1^{d-1}}(D_j)\leq
\vol_{\S_1^{d-1}}(\partial_\eta\scrW)<\frac{\ve}{RS}$.

For each $j\in \{1,\ldots,n\}$ we fix a point $\vecp_j\in D_j$ and set
$g_j:=g(\vecp_j)$; then note that 
$|g(\vecv_0)-g_j|<\frac\ve{R\vol(\S_1^{d-1})}$ for all
$\vecv_0\in D_j$, by our choice of $\eta$ in \eqref{ETACHOICEREG}. Hence
\begin{align} \notag
\biggl |
\int_{\S_1^{d-1}} f\big(\vecv_0, \rho^{d-1} \tau_1(\vecq_{\rho,{\vecbeta}}(\vecv_0),\vecv_0;\rho), 
\vecv_1(\vecq_{\rho,{\vecbeta}}(\vecv_0),\vecv_0;\rho)\big)\, d\lambda(\vecv_0)
\hspace{115pt} &
\\ \notag
-\sum_{j\in M} g_j \vol_{\S_1^{d-1}}(D_j) \int_{\S_1^{d-1}}
f\big(\vecv_0, \rho^{d-1} \tau_1(\vecq_{\rho,{\vecbeta}}(\vecv_0),\vecv_0;\rho), 
\vecv_1(\vecq_{\rho,{\vecbeta}}(\vecv_0),\vecv_0;\rho)\big) \,
d\lambda_j(\vecv_0) \biggr | &
\\ \label{STEP3FACT1}
\leq \sum_{j\in M} \int_{D_j} R\frac{\ve}{R\vol(\S_1^{d-1})}
\,d\!\vol_{\S_1^{d-1}}(\vecv_0)
+\sum_{j\in M'} \int_{D_j} RS\,d\!\vol_{\S_1^{d-1}}(\vecv_0)
< 2\ve. \hspace{30pt} &
\end{align}
Combining this with \eqref{STEP3ASSONRHO1}, applied for each $j\in M$, we get
\begin{align} \notag
\biggl |
\int_{\S_1^{d-1}} f\big(\vecv_0, \rho^{d-1} \tau_1(\vecq_{\rho,{\vecbeta}}(\vecv_0),\vecv_0;\rho), 
\vecv_1(\vecq_{\rho,{\vecbeta}}(\vecv_0),\vecv_0;\rho)\big) d\lambda(\vecv_0)
\hspace{115pt} &
\\ \notag
-\sum_{j\in M} g_j \vol_{\S_1^{d-1}}(D_j) 
\int_{\S_1^{d-1}} \int_{\R_{>0}} \int_{\S_1^{d-1}} f\big(\vecv_0,\xi,\vecv_1\big) 
p_{\vecalf,{\vecbeta}}(\vecv_0,\xi,\vecv_1) \, d\lambda_j(\vecv_0)
\, d\xi \,d\!\vol_{\S_1^{d-1}}(\vecv_1)\biggr | &
\\ \label{STEP3FACT2}
<2\ve+\sum_{j\in M} S \vol_{\S_1^{d-1}}(D_j)\frac{\ve}{S\vol(\S_1^{d-1})}\leq
3\ve. &
\end{align}
On the other hand, using 
$\int_{\S_1^{d-1}} \int_{\R_{>0}} p_{\vecalf,{\vecbeta}}(\vecv_0,\xi,\vecv_1) 
\,d\xi\,d\!\vol_{\S_1^{d-1}}(\vecv_1)=1$
(true for all $\vecv_0\in\S_1^{d-1}$, cf.\ \eqref{exactpos2-1hit-tpdef} 
and \eqref{PALFINTEQ1}) and the same bounds as in \eqref{STEP3FACT1},
we get
\begin{align} \notag
\biggl | 
\int_{\S_1^{d-1}} \int_{\R_{>0}} \int_{\S_1^{d-1}} f\big(\vecv_0,\xi,\vecv_1\big)
p_{\vecalf,{\vecbeta}}(\vecv_0,\xi,\vecv_1) \, d\lambda(\vecv_0)\,d\xi
\,d\!\vol_{\S_1^{d-1}}(\vecv_1) \hspace{120pt} &
\\ \notag
-\sum_{j\in M} g_j \vol_{\S_1^{d-1}}(D_j) 
\int_{\S_1^{d-1}} \int_{\R_{>0}} \int_{\S_1^{d-1}} f\big(\vecv_0,\xi,\vecv_1\big)
p_{\vecalf,{\vecbeta}}(\vecv_0,\xi,\vecv_1) \, d\lambda_j(\vecv_0)\,d\xi
\,d\!\vol_{\S_1^{d-1}}(\vecv_1) \biggr| &
\\ \label{STEP3FACT3}
\leq \sum_{j\in M} \int_{D_j} R\frac{\ve}{R\vol(\S_1^{d-1})}
\,d\!\vol_{\S_1^{d-1}}(\vecv_0)
+\sum_{j\in M'} \int_{D_j} RS\,d\!\vol_{\S_1^{d-1}}(\vecv_0)<2\ve. &
\end{align}
Combining \eqref{STEP3FACT2} and \eqref{STEP3FACT3} we obtain 
\eqref{STEP3CLAIM}, and we are done.
\end{proof}

\section{Iterated scattering maps}\label{reflmapssec}

To prepare for the proof of the main result of this paper, Theorem \ref{pathwayThm}, this section provides a detailed study of the map obtained by iterated scattering in a given sequence of balls 
$\vecm_j+\scrB_\rho^d$, when $\rho$ is small. The central result of this section is
Proposition \ref{REFLPROP} below.

\subsection{Two lemmas}

Given $\vecv\in\S_1^{d-1}$, the tangent space $\T_\vecv(\S_1^{d-1})$
is naturally identified with $\{\vecv\}^\perp\subset\R^d$.
For $\vech\in \T_\vecv(\S_1^{d-1})$ we write $D_\vech$ for the corresponding
derivative. We use the standard Riemannian metric for $\S_1^{d-1}$, and denote by $\T^1_\vecv(\S_1^{d-1})$ the set of unit vectors in $\T_\vecv(\S_1^{d-1})$. For any open subset $\scrV\subset\S_1^{d-1}$ we write $\T^1(\scrV)=\bigsqcup_{\vecv\in\scrV} \T^1_\vecv(\S_1^{d-1})$
for the unit tangent bundle of $\scrV$.

For $\vecw\in\R^d\setminus\{\bn\}$ and $\eta\geq 0$ 
we define $\scrV_\vecw:=\scrV_{\uvecw}$ (cf.\ \eqref{VPDEF}) and
\begin{align} \label{WVPDEF}
\scrV^\eta_\vecw:=\Bigl\{\vecu\in\S_1^{d-1}\col\varphi(\vecu,\vecw)>
B_\Theta+\eta\Bigr\}\subseteq\scrV_\vecw;
\end{align}
thus in particular $\scrV^0_\vecw=\scrV_\vecw$. %
Set
\begin{align} \label{C1THETAETADEF}
C_{\eta}:=1+\max\Bigl(
\sup_{\vech\in \T^1(\scrV_{\vecv}^\eta)} 
\bigl\|D_\vech\vecbeta^+_{\vecv}\bigr\|,
\sup_{\vech\in \T^1(\scrV_{\vecv}^\eta)} 
\bigl\|D_\vech\vecbeta^-_{\vecv}\bigr\|
\Bigr).
\end{align}
Then $C_\eta$ is independent of $\vecv$, depends continuously on 
$\eta\in (0,\pi-B_\Theta)$, and may approach infinity as $\eta\to 0$.

For any $\vecs\in\R^d\setminus\{\bn\}$ we let $\lambda_\vecs$ be the 
probability measure on $\S_1^{d-1}$ which gives
the direction of a ray after it has been scattered 
in the ball $\scrB_1^{d}$, given that the
incoming ray has direction $\vecs$ and is part of the line
$\vecx+\R\vecs$ with $\vecx$ picked at random in 
the $(d-1)$-dimensional unit ball $\{\vecs\}^\perp\cap \scrB_1^d$, with respect to 
the $(d-1)$-dimensional Lebesgue measure.
In particular for $\vecs=\vece_1$ we have 
\begin{align} \label{LAMBDA0DEF}
\lambda_{\vece_1}=\vol(\scrB_1^{d-1})^{-1}\cdot
(V_0)_*\bigl({\vol_{\R^{d-1}}}\bigr), %
\end{align}
where $V_0:\scrB_1^{d-1}\to\S_1^{d-1}$ is the map
\begin{align} \label{V0DEF}
V_0(\vecx)=\Theta_1\bigl(\vece_1,(-\sqrt{1-\|\vecx\|^2},\vecx)\bigr).
\end{align}
Thus $V_0$ is a diffeomorphism from $\scrB_1^{d-1}$ onto $\scrV_{\vece_1}$.
For general $\vecs\neq\bn$ we have $\lambda_\vecs=K_*(\lambda_{\vece_1})$ where $K\in\SO(d)$ is any
rotation such that $\uvecs=\vece_1 K$.
The following lemma shows that for any subset $M\subset\scrV_\vecs^\eta$, the renormalized 
$\S_1^{d-1}$-volume measure on $M$ and the renormalized $\lambda_\vecs$-measure
on $M$ are comparable with a controllable distortion factor.

\begin{lem} \label{COMPARABILITYLEM}
Given $\eta,\ve>0$, there exist constants $K_{\eta}>1$ and
$c_{\eta,\ve}>0$ such that for every 
Borel subset $M\subset\scrV_\vecs^\eta$ with
$\vol_{\S_1^{d-1}}(M)>0$, we have
\begin{align}
\lambda_\vecs(M)^{-1}\lambda_{\vecs |M}=g\cdot
(\vol_{\S_1^{d-1}}(M))^{-1}\cdot \vol_{\S_1^{d-1} |M}
\end{align}
for some continuous function $g:M\to\R_{>0}$ with $g(\vecu)\in
[K_{\eta}^{-1},K_{\eta}]$ for all $\vecu\in M$.

If furthermore $M$ has diameter $\leq c_{\eta,\ve}$,
i.e.\  if $\varphi(\vecu_1,\vecu_2)\leq c_{\eta,\ve}$
for all $\vecu_1,\vecu_2\in M$, then $g(\vecu)\in[1-\ve,1+\ve]$
for all $\vecu\in M$.
\end{lem}

\begin{proof}
Without loss of generality we may assume $\vecs=\vece_1$.
Set $\nu=(\vol_{\S_1^{d-1}}(M))^{-1}\cdot \vol_{\S_1^{d-1}|M}$ and
$\lambda=\lambda_{\vece_1}(M)^{-1}\lambda_{\vece_1|M}$.
By \eqref{LAMBDA0DEF} we have $\lambda=C \cdot f_{|M} \cdot \nu$, 
where $f:\scrV_{\vece_1}\to\R_{>0}$ is the Jacobian determinant of the 
inverse map $V_0^{-1}$ from $\scrV_{\vece_1}$ to $\scrB_1^{d-1}$,
and $C>0$ is a constant determined by $C\int_M f\, d\nu=1$
(in particular if the ratio of $\sup_M f$ and $\inf_M f$ is close to $1$ then 
the function $C\cdot f_{|M}$ is uniformly close to $1$).
Now the lemma follows from the fact that both $f$ and $f^{-1}$ 
are bounded and uniformly continuous on $\scrV_{\vece_1}^\eta$,
since $\scrV_{\vece_1}^\eta$ has compact closure in $\scrV_{\vece_1}$.
\end{proof}

We define $\lambda_\vecs^\eta$ to be the probability measure
which is obtained by restricting $\lambda_\vecs$ to $\scrV_\vecs^\eta$, and normalizing appropriately.

Given $\vecr,\vecs\in\R^d\setminus\{\bn\}$,
a number $\rho>0$ and a continuous function
$\vecbeta:\scrV_{\vecr}^{\eta}\to\R^d$, we set 
\begin{align} \label{OMEGARHOMBDEF}
\Omega%
=\Bigl\{\vecv\in\scrV_{\vecr}^{\eta}\col (\rho\vecbeta(\vecv)+\R_{>0}\vecv)\cap
(\vecs+\scrB_\rho^d)\neq\emptyset\Bigr\}.
\end{align}
For $\vecv\in\Omega$ we set
\begin{equation}
	\tau(\vecv)=\tau_{\rho,\vecs,\vecbeta}(\vecv)
	:=\inf\{t>0\col
\rho\vecbeta(\vecv)+t\vecv\in\vecs+\scrB_\rho^d\},
\end{equation}
let $W(\vecv)=W_{\rho,\vecs,\vecbeta}(\vecv)$ be the impact location on $\S_1^{d-1}$, i.e., the point for 
which $\rho\vecbeta(\vecv)+\tau(\vecv)\vecv=\vecs+\rho W(\vecv)$,
and let
\begin{align} \label{SHINELEM2VDEF}
V(\vecv)=V_{\rho,\vecs,\vecbeta}(\vecv)
:=\Theta_1(\vecv,W(\vecv))\in\S_1^{d-1},
\end{align}
the outgoing direction after the ray
$\rho\vecbeta(\vecv)+\R_{>0}\vecv$ is scattered in the sphere 
$\vecs+\S_\rho^{d-1}$.

\begin{lem} \label{SHINELEM1}
Given any $0<\eta<\frac {\pi-B_\Theta}{100}$, $C\geq 10$ and $\ve>0$, there exists a constant
$\trho_0=\trho_0(\eta,C,\ve)>0$ such that all the following 
statements hold for any $\rho\in (0,\trho_0)$,
any $\vecr,\vecs\in\R^d\setminus\{\bn\}$ with 
$\|\vecs\|\geq C^{-1}$ and $\varphi(\vecr,\vecs)>B_\Theta+2\eta$,
and for any $\C^1$-function
$\vecbeta:\scrV_{\vecr}^{\eta}\to\R^d$ 
with $\sup_{\scrV_{\vecr}^{\eta}} \|\vecbeta\|\leq C$ and
$\sup_{\vech\in \T^1(\scrV_{\vecr}^{\eta})} \|D_\vech\vecbeta\|\leq C$:
\begin{enumerate}
\item[(i)] Let $\overline{V}=\overline{V}_{\rho,\vecs,\vecbeta}$ be the 
restriction of $V=V_{\rho,\vecs,\vecbeta}$ to $V^{-1}(\scrV_\vecs^\eta)$;
then $\overline{V}$ is a $\C^1$ diffeomorphism onto $\scrV_\vecs^\eta$.
\item[(ii)] If $M\subset\scrV^{\eta}_\vecs$ 
is any Borel subset with $\lambda_\vecs(M)>0$ and if $\mu$ denotes
the measure $\vol_{\S_1^{d-1}}$ restricted to $\overline{V}^{-1}(M)$ and
rescaled to be a probability measure, then
$\overline{V}_*\mu=g\cdot \lambda_{\vecs}(M)^{-1}{\lambda_{\vecs}}_{|M}$
for some continuous function $g:M\to [1-\ve,1+\ve]$.
\item[(iii)] Define the $\C^1$ maps $B^\pm=B^\pm_{\rho,\vecs,\vecbeta}:\scrV^\eta_\vecs
\to \S_1^{d-1}$ through
$B^\pm(\vecu)=\vecbeta^\pm_{\overline{V}^{-1}(\vecu)}(\vecu).$ Then 
$\bigl\| B^\pm(\vecu)-\vecbeta^\pm_{\uvecs}(\vecu)\bigr\|<\ve$
for all $\vecu\in\scrV_\vecs^\eta$ and
$\|D_\vech B^\pm\|<C_{\eta}$ for all 
$\vech\in \T^1(\scrV_\vecs^\eta)$.
\end{enumerate}
\end{lem}

\begin{proof}
Since $\Omega$, $W_{\rho,\vecs,\vecbeta}$ and 
$V_{\rho,\vecs,\vecbeta}$ are invariant under 
$\langle\rho,\vecs\rangle\mapsto\langle c\rho,c\vecs\rangle$ for any $c>0$,
it suffices to treat the case $\|\vecs\|=1$. After an auxiliary rotation
we may then assume $\vecs=\vece_1$.
From now on we will always keep $\rho<\frac{\eta}{20C}$.
Then $\rho\vecbeta(\vecv)$ certainly lies outside
the ball $\vece_1+\scrB_\rho^d$ for all $\vecv\in\scrV_\vecr^\eta$.
Let us write $\veca=\veca(\vecv)=\rho^{-1}\vece_1-\vecbeta(\vecv)$
and note that $\|\veca\|>\frac{1}{2\rho}>10^3$ for all $\vecv$. 
Now for each $\vecv\in\Omega$
the ray $\R_{>0}\vecv$ hits $\veca+\scrB_1^d$, and thus 
$\varphi(\vecv,\veca)\leq \arcsin\bigl(\|\veca\|^{-1}\bigr)
\leq \frac{\pi}{2\|\veca\|}<\pi\rho$.
Also $\varphi(\veca,\vece_1)
\leq\arcsin\bigl(\rho\|\vecbeta(\vecv)\|\bigr)
\leq \frac{\pi}{2}C\rho$.
Hence by the triangle inequality for $\varphi$,
using $\rho<\frac{\eta}{20C}$ we obtain
\begin{align} \label{VMCLOSEINOMEGA}
\varphi(\vecv,\vece_1)<
\bigl(\pi+\sfrac{\pi}2 C\bigr)\rho<2C\rho<\frac{\eta}{10},
\qquad \forall \vecv\in\Omega.
\end{align}
We now compute
\begin{align} \label{SHINE1OMEGADEF}
& \Omega=\Bigl\{\vecv\in\S_1^{d-1}\col 1-\|\veca\|^2+(\veca\cdot\vecv)^2>0\Bigr\}, %
\end{align}
which is automatically a subset of $\scrV_\vecr^\eta$, 
because of the assumption $\varphi(\vecr,\vece_1)>B_\Theta+2\eta$,
and furthermore, for all $\vecv\in\Omega$, $\vech\in\T_\vecv(\S_1^{d-1})$,
\begin{equation} \label{WVFORMULA}
W(\vecv)=-\veca+\bigl((\veca\cdot\vecv)-\sqrt{1-\|\veca\|^2
+(\veca\cdot\vecv)^2}\bigr) \vecv
\end{equation}
and
\begin{multline}\label{DHW}
D_\vech W =D_\vech\vecbeta+\biggl(\veca\cdot\vech-\vecv\cdot D_\vech\vecbeta
-\frac {\veca\cdot D_\vech\vecbeta+(\veca\cdot\vecv)(\veca\cdot\vech-
\vecv\cdot D_\vech\vecbeta)}{\sqrt{1-\|\veca\|^2+(\veca\cdot\vecv)^2}}\biggr)\vecv 
\\ 
+\Bigl(\veca\cdot\vecv-\sqrt{1-\|\veca\|^2+(\veca\cdot\vecv)^2}\Bigr) \vech,
\end{multline}
using the standard embedding $\vech\in \T_\vecv(\S_1^{d-1})\subset\R^d$.
Since $\Omega\subset\S_1^{d-1}$ is contained in a small neighborhood of
$\vece_1$ (see \eqref{VMCLOSEINOMEGA}), the map
$\vecv\mapsto\vecbeta(\vece_1)_\perp+\rho^{-1}\vecv_\perp$ gives a 
diffeomorphism from $\Omega$ onto
\begin{equation}
	\Omega_I:=\vecbeta(\vece_1)_\perp+\rho^{-1}\Omega_\perp
\subset\{0\}\times\R^{d-1},
\end{equation}
which transforms the $\S_1^{d-1}$ volume measure
into the standard $\R^{d-1}$ Lebesgue volume measure scaled by $\rho^{d-1}$ 
times a distortion factor %
which is uniformly close to $1$ when $\rho$ is small.
The inverse map is (using the identification $\{0\}\times\R^{d-1}=\R^{d-1}$)
\begin{align}
I:\Omega_I\to\Omega, \qquad
\vecx\mapsto Q\bigl(\rho(\vecx-\vecbeta(\vece_1)_\perp)\bigr),
\end{align} 
where $Q$ is the map
\begin{align}
Q:\scrB_1^{d-1}\to\S_1^{d-1}, \qquad
\vecy\mapsto (\sqrt{1-\|\vecy\|^2},\vecy).
\end{align}

Now, for (i) and (ii), it suffices to prove that (i) and (ii) hold with
the map $V$ replaced throughout by $V_I:=V\circ I:\Omega_I\to\S_1^{d-1}$, 
$\vol_{\S_1^{d-1}}$ replaced by $\vol_{\R^{d-1}}$, and $\ve$ replaced by $\ve/2$.

Recall the definition \eqref{V0DEF} of the diffeomorphism $V_0$
from $\scrB_1^{d-1}$ onto $\scrV_{\vece_1}$.
Using the spherical invariance of $\Theta_1$ it follows that
$V_0^{-1}(\scrV_{\vece_1}^{\eta/2})=\scrB_r^{d-1}$ for some 
$r=r(\Theta,\eta)<1$.

From \eqref{VMCLOSEINOMEGA} we get 
$\Omega_I\subset\{0\}\times\scrB_{3C}^{d-1}$.
Now for all $\vecx\in\scrB_{3C}^{d-1}$ we compute,
with $\vecv=Q(\rho(\vecx-\vecbeta(\vece_1)_\perp))$
and $\veca=\rho^{-1}\vece_1-\vecbeta(\vecv)$ as
before, 
\begin{align} \label{AVASYMPT}
& \veca\cdot\vecv=\rho^{-1}-\beta_1(\vecv)+\sfrac 12
\bigl(\|\vecbeta(\vece_1)_\perp\|^2-\|\vecx\|^2\bigr)\rho+O(\rho^2);
\\
& \|\veca\|^2=\rho^{-2}-2\beta_1(\vecv)\rho^{-1}+
\|\vecbeta(\vece_1)\|^2+O(\rho),
\end{align}
where $\beta_1(\vecv):=\vecbeta(\vecv)\cdot\vece_1$.
Here and in the rest of the proof, the implied constant in any big $O$
depends only on $\Theta,C,\eta,\ve$.
It follows that
\begin{align} \label{EXPRASYMPT}
1-\|\veca\|^2+(\veca\cdot\vecv)^2=1-\|\vecx\|^2+O(\rho),
\end{align}
and hence by \eqref{SHINE1OMEGADEF}, if $\rho$ is sufficiently small,
\begin{align} \label{BRCONTAINED}
V_0^{-1}(\scrV_{\vece_1}^{\eta/2})=\scrB_r^{d-1}\subset\Omega_I.
\end{align}

Let $W_I:=W\circ I:\Omega_I\to\S_1^{d-1}$.
By a computation using \eqref{WVFORMULA} and 
\eqref{AVASYMPT}--\eqref{EXPRASYMPT} we find 
that for each $\vecx\in\scrB_{r}^{d-1}$,
\begin{align} \label{WIW0CLOSE}
W_I(\vecx)=W(\vecv)=W_0(\vecx)+O(\rho),
\qquad \text{with } \: 
W_0(\vecx)=\Bigl(-\sqrt{1-\|\vecx\|^2},\vecx\Bigr).
\end{align}
By a similar computation using \eqref{DHW} we also obtain,
for $\rho$ sufficiently small,
\begin{equation}
	(D_\vech W)_\perp=\rho^{-1}\vech+O(\|\vech\|)
\end{equation}
for all $\vecv\in I(\scrB_{r}^{d-1})$ and $\vech\in\T_\vecv(\S_1^{d-1})$.
Therefore, for all $\vecx\in\scrB_r^{d-1}$ and
$\vecu\in\T_\vecx(\Omega_I)\cong\{0\}\times\R^{d-1}$,
\begin{align} \label{DWIAPPRPREL}
\bigl(D_\vecu W_I\bigr)_\perp %
=\rho^{-1}D_\vecu I + O(\|D_\vecu I\|)
=\vecu + O(\rho\|\vecu\|),
\end{align}
since $D_\vecu I=\rho\vecu+O(\rho^2\|\vecu\|)$.
We also have, trivially, $(D_\vecu W_0)_\perp=\vecu$.
But we know 
\begin{equation} \label{DUWIDUW0}
	D_\vecu W_I\in \T_{W_I(\vecx)}(\S_1^{d-1})
=\{W_I(\vecx)\}^\perp\subset\R^d,
\qquad
D_\vecu W_0\in \{W_0(\vecx)\}^\perp\subset\R^d .
\end{equation}
For $\rho$ sufficiently small, \eqref{WIW0CLOSE}
implies that $W_I(\vecx)$ and $W_0(\vecx)$ are close on $\S_1^{d-1}$
and both have $\vece_1$-component $<-\frac 12\sqrt{1-r^2}$.
Hence \eqref{DWIAPPRPREL}, $(D_\vecu W_0)_\perp=\vecu$ and
\eqref{DUWIDUW0} imply
\begin{align} \label{DWIAPPR}
D_\vecu W_I=D_\vecu W_0+O(\rho\|\vecu\|).
\end{align}
It follows from \eqref{WIW0CLOSE} and \eqref{DWIAPPR} that the function
$(W_I)_{|\scrB_{r}^{d-1}}$ converges to $(W_0)_{|\scrB_{r}^{d-1}}$ in 
$\C^1$-norm as $\rho\to 0$.

To continue, note that $V_I(\vecx)=\Theta_1(I(\vecx),W_I(\vecx))$.
The function $I:\Omega_I\to\Omega$ tends to the 
constant function $\vecx\mapsto\vece_1$ in $\C^1$-norm, 
as $\rho\to 0$.
Using that $\Theta_1:\scrS_-\to\S_1^{d-1}$ is $\C^1$ and that
$I(\scrB_{r}^{d-1})\times W_I(\scrB_{r}^{d-1})$ is a subset with compact
closure in $\scrS_-$ for $\rho$ sufficiently small,
we conclude that the map
$(V_I)_{|\scrB^{d-1}_{r}}:\scrB_{r}^{d-1} \to\S_1^{d-1}$ 
tends to $(V_0)_{|\scrB^{d-1}_{r}}$ in $\C^1$-norm as $\rho\to 0$.
Since the latter map is a $\C^1$ diffeomorphism from 
$\scrB^{d-1}_r$ onto $\scrV_{\vece_1}^{\eta/2}$ which is a restriction of
a diffeomorphism from all of $\scrB^{d-1}_1$,
it follows by a standard argument in differential 
geometry (cf.,\ e.g.,\ \cite[Sec.\ 2.1]{hirsch})
that, for $\rho$ small enough,
$(V_I)_{|\scrB^{d-1}_{r}}$ must be a diffeomorphism as well, mapping
$\scrB_{r}^{d-1}$ onto some open subset of $\S_1^{d-1}$ which contains
$\scrV_{\vece_1}^\eta$, and that (ii) holds with $\overline{V}$ 
replaced by $(V_I)_{|\scrB^{d-1}_{r}}$
(and $\vol_{\S_1^{d-1}}$ replaced by $\vol_{\R^{d-1}}$, and $\ve$
replaced by $\ve/2$).

Hence to conclude the proof of (i) and (ii), it only remains to prove that
the map $\overline{V}$ is injective, or equivalently that 
$V_I(\vecx)\notin\scrV_{\vece_1}^\eta$ for all 
$\vecx\in\Omega_I\setminus\scrB^{d-1}_r$.
A computation shows that if $\rho$ is small and
$\vecx\in\Omega_I\setminus\scrB^{d-1}_r$, then
$W_I(\vecx)\cdot\vece_1\geq -\sqrt{1-r^2}-O(\rho)$,
and thus
\begin{align}
\varphi(I(\vecx),W_I(\vecx))\leq
\varphi(\vece_1,W_I(\vecx))+O(\rho)\leq\sfrac{\pi}2+\arccos(r)+O(\rho).
\end{align}
It follows from our assumptions on $\Theta_1$ that
$\varphi(I(\vecx),V_I(\vecx))$ is an increasing $\C^1$ function of
$\varphi(I(\vecx),W_I(\vecx))\in(\frac{\pi}2,\pi]$, and 
$\varphi(I(\vecx),V_I(\vecx))=B_\Theta+\frac{\eta}2$
when $\varphi(I(\vecx),W_I(\vecx))=\frac{\pi}2+\arccos(r)$.
Hence
\begin{align}
\varphi(\vece_1,V_I(\vecx))\leq\varphi(I(\vecx),V_I(\vecx))+O(\rho)
\leq B_\Theta+\sfrac{\eta}2+O(\rho),
\qquad\forall\vecx\in\Omega_I\setminus\scrB^{d-1}_r,
\end{align}
and in particular if $\rho$ is sufficiently small then
$V_I(\vecx)\notin\scrV_{\vece_1}^\eta$ for all
$\vecx\in\Omega_I\setminus\scrB^{d-1}_r$, as desired.

Finally we turn to (iii).
Set $\overline{V}_I=(V_I)_{|V_I^{-1}(\scrV_{\vece_1}^\eta)}$.
Then, since $(V_I)_{|\scrB^{d-1}_{r}}$ tends to $(V_0)_{|\scrB^{d-1}_{r}}$ 
in $\C^1$-norm, we have $D_\vecu(\overline{V}_I^{-1})=O(\|\vecu\|)$
for all $\vecx\in\scrV_{\vece_1}^\eta$ and 
$\vecu\in\T_\vecx(\scrV_{\vece_1}^\eta)$, so long as $\rho$ is sufficiently 
small.
But 
\begin{equation}
	\overline{V}^{-1}=I\circ \overline{V}_I^{-1}:
\scrV_{\vece_1}^\eta\to\S_1^{d-1}
\end{equation}
and $I$ tends to the constant function $\vecx\to\vece_1$ in $\C^1$-norm.
Hence also $\overline{V}^{-1}$  
tends to the constant function $\vecx\to\vece_1$ in $\C^1$-norm as $\rho\to 0$.
Therefore, by continuity and the definition of $C_{\eta}$, \eqref{C1THETAETADEF}, claim (iii) is established.
\end{proof}

\subsection{The main proposition}

Let us consider a particle trajectory that follows the ray emerging at $\vecm_0+\rho\vecbeta(\vecv_0)$ in direction $\vecv_0$, is scattered at the ball $\vecm_1+\scrB_\rho^{d}$, exits with velocity $\vecv_1$ and moves with constant speed until scattered at $\vecm_2+\scrB_\rho^{d}$, and so on; after the final scattering at $\vecm_n+\scrB_\rho^{d}$ the particle exists with velocity $\vecv_n$. We call $\vec{\vecm}=(\vecm_0,\vecm_1,\ldots,\vecm_n)$ the {\em scattering sequence} associated with this trajectory.
We denote by $\overline{\vecbeta}^-(\vecv_n)$ and
$\overline{\vecbeta}^+(\vecv_n)\in\S_1^{d-1}$
the position of impact and exit on the last scatterer; thus
$\overline{\vecbeta}^-(\vecv_n)=\vecbeta_{\vecv_{n-1}}^-(\vecv_n)$
and
$\overline{\vecbeta}^+(\vecv_n)=\vecbeta_{\vecv_{n-1}}^+(\vecv_n)$.
We call $\vecv_0$ and $\vecv_n$ the {\em initial} and {\em final velocity}, respectively. Set furthermore
\begin{equation}
	\vecs_k:=\vecm_{k}-\vecm_{k-1}.
\end{equation}
Note that $\vecs_k$ differs from the path segment $\tau_k\vecv_{k-1}$ defined in \eqref{pathseg} by $\ll \rho$.

Given positive constants $\eta,C$, we say a sequence $\vec{\vecm}=(\vecm_0,\vecm_1,\ldots,\vecm_n)$ of scatterer positions is $(\eta,C)$-{\em admissible}, if 
\begin{equation}
	\|\vecs_k\|\geq C^{-1}, \qquad (k=1,\ldots,n)
\end{equation}
and
\begin{equation}
	\varphi\bigl(\vecs_{k+1},\vecs_{k}\bigr)>B_\Theta+2\eta, \qquad (k=1,\ldots,n-1)	.
\end{equation}

\begin{prop} \label{REFLPROP}
Given any $N\in\Z_{>0}$, $0<\eta<\frac{\pi-B_\Theta}{100}$, 
$C\geq 10$ and $\ve>0$, there exists a constant
$\rho_0=\rho_0(N,\eta,C,\ve)>0$ such that for any 
$\rho\in (0,\rho_0)$, any $(\eta,C)$-admissible sequence 
$\vec{\vecm}=(\vecm_0,\vecm_1,\ldots,\vecm_n)$ with $n\leq N$,
and any $\C^1$ function $\vecbeta:\S_1^{d-1}\to\R^d$ with
$\sup_{\S_1^{d-1}} \|\vecbeta\|\leq C$ and
$\sup_{\vech\in \T^1(\S_1^{d-1})} \|D_\vech\vecbeta\|\leq C$,
the following holds:
\begin{enumerate}
	\item[(i)] There exists a $\C^1$ diffeomorphism 
$\Phi=\Phi_{\rho,\vec{\vecm},\vecbeta}$ from an open subset 
$\Delta=\Delta_{\rho,\vec{\vecm},\vecbeta}\subset\S_1^{d-1}$ onto
$\scrV_{\vecs_n}^{\eta}$ such that, for every $\vecv\in\S_1^{d-1}$, $\vecu\in\scrV_{\vecs_n}^{\eta}$, the following statements are equivalent:
\begin{enumerate}
	\item[(a)] There is a particle trajectory with scattering sequence $\vec{\vecm}$, initial velocity $\vecv_0=\vecv$ and final velocity $\vecv_n=\vecu$.
	\item[(b)] $\vecv\in\Delta$ and $\Phi(\vecv)=\vecu$.
\end{enumerate}
	\item[(ii)] The position of impact and exit on the last scatterer, $\overline{\vecbeta}^-=\overline{\vecbeta}^-_{\rho,\vec{\vecm},\vecbeta}$ 
and $\overline{\vecbeta}^+=\overline{\vecbeta}^+_{\rho,\vec{\vecm},\vecbeta}$, 
are $\C^1$ maps from
$\scrV_{\vecs_n}^{\eta}$ to $\S_1^{d-1}$, satisfying 
$\bigl\| \overline{\vecbeta}^\pm(\vecu)
-\vecbeta^\pm_{\uvecs_n}(\vecu)\bigr\|<\ve$
for all $\vecu\in\scrV_{\vecs_n}^{\eta}$ and
$\sup_{\vech\in \T^1(\scrV_{\vecs_n}^{\eta})} 
\|D_\vech \overline{\vecbeta}^\pm\|<C_{\eta}$.
	\item[(iii)] If $\mu$ is the $\S_1^{d-1}$ volume measure restricted to
$\Delta$ and rescaled to be a probability measure, then
$\Phi_*\mu=g\cdot\lambda_{\vecs_n}^{\eta}$ for some continuous function 
$g:\scrV_{\vecs_n}^{\eta}\to [1-\ve,1+\ve]$.
\end{enumerate}
\end{prop}

\begin{proof}
We assume that $C\geq C_{\eta}$.
Fix $\ve_1\in (0,\ve)$ so small that
$(1+\ve_1)^{2N-1}<1+\ve$ and $(1-\ve_1)^{2N-1}>1-\ve$,
and then fix 
\begin{align}
\rho_0=\rho_0(N,\eta,C,\ve):=\min\bigl(
\trho_0(\eta,C,\ve_1),\sfrac{c_{\eta,\ve_1}}{2\pi C},
\sfrac{\eta}{20C^2}\bigr),
\end{align}
where $\trho_0(\eta,C,\ve_1)$ is as in Lemma \ref{SHINELEM1},
and $c_{\eta,\ve_1}$ is as in Lemma \ref{COMPARABILITYLEM}.

Now take arbitrary $\rho\in (0,\rho_0)$ and $\vecm_0,\ldots,\vecm_n$
and $\vecbeta$ satisfying the assumptions of the proposition.
Using the notation introduced in the context of Lemma \ref{SHINELEM1}
(setting $\vecr=-\vecs_1$, say, and considering the restriction of
$\vecbeta$ to $\scrV_\vecr^\eta$), we set 
\begin{align}
& \Omega^{(1)}=V_{\rho,\vecs_1,\vecbeta}^{-1}
(\scrV_{\vecs_1}^\eta),
\\
& V^{(1)}=\overline{V}_{\rho,\vecs_1,\vecbeta}
:\Omega^{(1)}\to\S_1^{d-1} ,
\\
& \vecbeta^{(1)}=B^+_{\rho,\vecs_1,\vecbeta}
:\scrV^\eta_{\vecs_1}\to\S_1^{d-1} ,
\end{align} 
and recursively for $k=1,2,\ldots,n-1$
($\vecs=\vecs_{k+1}=\vecm_{k+1}-\vecm_k$ and $\vecr=\vecs_k=\vecm_k-\vecm_{k-1}$, in the the context of Lemma \ref{SHINELEM1}),
\begin{align} \label{OMEGAKPDEF}
& \Omega^{(k+1)}=(V^{(k)})^{-1}\bigl(
V_{\rho,\vecs_{k+1},\vecbeta^{(k)}}^{-1}
(\scrV_{\vecs_{k+1}}^\eta)\bigr),
\\ \label{VKPDEF}
& V^{(k+1)}=\overline{V}_{\rho,\vecs_{k+1},\vecbeta^{(k)}}
\circ V^{(k)}:
\Omega^{(k+1)}\to\S_1^{d-1},
\\ \label{BETAKPDEF}
& \vecbeta^{(k+1)}=B^+_{\rho,\vecs_{k+1},\vecbeta^{(k)}}:
\scrV_{\vecs_{k+1}}^\eta\to\S_1^{d-1}.
\end{align}
Then $\Omega^{(n)}\subset\Omega^{(n-1)}\subset\ldots\subset\Omega^{(1)}$
by definition. 

Lemma \ref{SHINELEM1} implies, by induction over $k$, that
$V^{(k)}$ is a $\C^1$ diffeomorphism from $\Omega^{(k)}$ onto
$\scrV_{\vecs_{k}}^\eta$ and that $\vecbeta^{(k)}$ is a 
$\C^1$ map from $\scrV_{\vecs_{k}}^\eta$ to $\S_1^{d-1}$
satisfying
$\sup_{\vech\in\T^1(\scrV_{\vecs_{k}}^\eta)}\bigl\|D_\vech
\vecbeta^{(k)}\bigr\|<C_{\eta}\leq C$, for each $k=1,\ldots,n$.

Now set $\Delta=\Omega^{(n)}$ and $\Phi=V^{(n)}$. Then $\Phi$ is a 
$\C^1$ diffeomorphism from $\Delta$ onto $\scrV_{\vecs_n}^\eta$.

{\em Proof of} (i). The implication ``(b)$\Rightarrow$(a)'' in Proposition \ref{REFLPROP}
follows directly by construction. Indeed, suppose $\vecv\in\Delta$ and 
$\vecu=\Phi(\vecv)$.
Then since $\vecv=\vecv_0\in\Omega^{(1)}$ %
the ray $\vecm_0+\rho\vecbeta(\vecv)+\R_{>0}\vecv$ hits the ball
$\vecm_1+\scrB_\rho^d$ at the point
$\vecm_1+\rho W_{\rho,\vecs_1,\vecbeta}(\vecv)$,
and after scattering we obtain the ray
$\vecm_1+\rho \vecbeta^{(1)}(V^{(1)}(\vecv))+\R_{>0} V^{(1)}(\vecv)$.
Similarly it follows from our definitions that for each 
$k\in\{1,\ldots,n-1\}$, the ray
$\vecm_k+\rho\vecbeta^{(k)}(V^{(k)}(\vecv))+\R_{>0} V^{(k)}(\vecv)$
hits the ball $\vecm_{k+1}+\scrB_\rho^d$ and after scattering gives the
ray $\vecm_{k+1}+\rho \vecbeta^{(k+1)}\bigl(V^{(k+1)}(\vecv)\bigr)
+\R_{>0}V^{(k+1)}(\vecv)$. Using this for $k=1,\ldots,n-1$ we see that
(a) holds for the given vectors $\vecv=\vecv_0$, $\vecu=\vecv_n$.

Conversely, we now prove ``(a)$\Rightarrow$(b)''.
Suppose that $\vecv=\vecv_0\in\S_1^{d-1}$ and 
$\vecu=\vecv_n\in\scrV_{\vecs_n}^\eta$ satisfy the assumptions in (a),
i.e.\ the ray $\vecm_0+\rho\vecbeta(\vecv_0)+\R_{>0}\vecv_0$ hits 
$\vecm_1+\scrB_\rho^{d}$; after scattering in this ball we get a ray which hits
$\vecm_2+\scrB_\rho^{d}$, 
and so on for $\vecm_3+\scrB_\rho^{d},\ldots,\vecm_n+\scrB_\rho^{d}$, and
after the final scattering in $\vecm_n+\scrB_\rho^{d}$ we get a ray 
with direction $\vecv_n=\vecu$.
Let $\vecv_k\in\S_1^{d-1}$ be the
direction of the ray leaving the ball $\vecm_k+\scrB_\rho^d$
in this scenario.
For each $k\in\{1,\ldots,n-1\}$, by assumption there exists some ray 
with direction $\vecv_k$ and starting point on $\vecm_k+\S_\rho^{d-1}$ 
which hits the ball $\vecm_{k+1}+\scrB_\rho^d$; this implies
$\varphi(\vecv_k,\vecs_{k+1})<\frac{\eta}{10}$ by the
same argument that led to \eqref{VMCLOSEINOMEGA}, 
using $\rho<\sfrac{\eta}{20C^2}$.
Since we are also assuming $\varphi(\vecs_k,\vecs_{k+1})
>B_\Theta+2\eta$ we conclude $\vecv_k\in\scrV_{\vecs_k}^\eta$.
This is also true for $k=n$, since $\vecv_n=\vecu$ and
$\vecu\in\scrV_{\vecs_n}^\eta$ by assumption.

Now since the ray $\vecm_{0}+\rho\vecbeta(\vecv_0)+\R_{>0}\vecv_0$
hits $\vecm_{1}+\scrB_\rho^d$ and gives a ray with direction 
$\vecv_1\in\scrV_{\vecs_1}^\eta$ after scattering,
we must have, by our definitions, 
$\vecv_0\in V_{\rho,\vecs_1,\vecbeta}^{-1}
(\scrV_{\vecs_1}^\eta)=\Omega^{(1)}$ and
$V^{(1)}(\vecv_0)=\vecv_1$; in fact the ray obtained after the first scattering
equals
\begin{align} \label{BIMPLAFIRSTRAY2}
\vecm_{1}+\rho\vecbeta^{(1)}(\vecv_1)+\R_{>0}\vecv_1,
\qquad \vecv_1=V^{(1)}(\vecv).
\end{align}
Similarly one now proves by induction that for each $k\in\{1,\ldots,n\}$
we have $\vecv_0\in\Omega^{(k)}$, and the ray obtained after the
$k$th scattering in our scenario equals
\begin{align} \label{BIMPLAKTHRAY}
\vecm_{k}+\rho\vecbeta^{(k)}(\vecv_k)+\R_{>0}\vecv_k,\qquad
\vecv_k=V^{(k)}(\vecv).
\end{align}
In particular for $k=n$ we obtain $\vecv=\vecv_0\in\Delta$
and $\vecv_n=\Phi(\vecv)$, i.e.\ (b) holds for the given vectors $\vecv=\vecv_0$, $\vecu=\vecv_n$.
This completes the proof of the implication
``(a)$\Rightarrow$(b)''. %

{\em Proof of} (ii). It follows from the above discussion that the functions 
$\overline{\vecbeta}^\pm$ defined in Proposition \ref{REFLPROP} are the same as
\begin{equation}
	\vecbeta^{(n)}=B^+_{\rho,\vecs_n,\vecbeta^{(n-1)}}
:\scrV^\eta_{\vecs_n}\to\S_1^{d-1}
\end{equation}
and 
\begin{equation}
	B^-_{\rho,\vecs_n,\vecbeta^{(n-1)}}
:\scrV^\eta_{\vecs_n}\to\S_1^{d-1},
\end{equation}
respectively,
where if $n=1$ we understand $\vecbeta^{(0)}=\vecbeta$.
Hence the claims about $\overline{\vecbeta}^\pm$ %
are direct consequences of Lemma \ref{SHINELEM1} (iii) (since $\ve_1<\ve$).

{\em Proof of} (iii). 
When $n=1$ the statement %
follows directly from Lemma \ref{SHINELEM1} (ii);
thus from now on we assume $n\geq 2$. Let us write $\mu^{(k)}=V^{(k)}_*\mu$
for $k=1,\ldots,n$, so that $\mu^{(n)}=\Phi_*\mu$.
We know that for each $\vecv\in V^{(1)}(\Delta)$ the ray
$\vecm_1+\rho\vecbeta^{(1)}(\vecv)+\R_{>0}\vecv$ hits
$\vecm_2+\scrB_\rho^d$. Hence $\varphi(\vecv,\vecs_2)<\pi C\rho$
by the same argument that led to \eqref{VMCLOSEINOMEGA}.
It follows that 
\begin{equation}
	\varphi(\vecv,\vecv')<2\pi C\rho\leq c_{\eta,\ve_1}
\end{equation}
for all $\vecv,\vecv'\in V^{(1)}(\Delta)$, by our choice of $\rho$.
Hence Lemma \ref{SHINELEM1} (ii), using $\rho<\trho_0(\eta,C,\ve_1)$,
together with Lemma \ref{COMPARABILITYLEM} imply that
\begin{equation}
	\mu^{(1)}=(\overline{V}_{\rho,\vecs_1,\vecbeta})_*(\mu)
=g_1\cdot\nu_1
\end{equation}
where $\nu_1$ is the $\S_1^{d-1}$ volume measure
restricted to $V^{(1)}(\Delta)$ and rescaled to be a probability measure,
and $g_1$ is some continuous function from $V^{(1)}(\Delta)$ to
$[(1-\ve_1)^2,(1+\ve_1)^2]$.

Repeating the same argument, using 
\begin{equation}
	\mu^{(k+1)}=
(\overline{V}_{\rho,\vecs_{k+1},\vecbeta^{(k)}})_*(\mu^{(k)}),
\end{equation}
we obtain $\mu^{(k)}=g_k\cdot\nu_k$ for each $k=2,\ldots,n-1$, 
where $\nu_k$ is the $\S_1^{d-1}$ volume measure
restricted to $V^{(k)}(\Delta)$ and rescaled to be a probability measure,
and $g_k$ is some continuous function from $V^{(k)}(\Delta)$ to
$[(1-\ve_1)^{2k},(1+\ve_1)^{2k}]$.
Using this fact for $k=n-1$ together with one more application of 
Lemma \ref{SHINELEM1} (ii), we finally obtain
$\Phi_*\mu=\mu^{(n)}=g\cdot\lambda_{\vecs_n}^\eta$,
where $g$ is some continuous function from
$V^{(n)}(\Delta)=\scrV_{\vecs_n}^\eta$ to
$[(1-\ve_1)^{2n-1},(1+\ve_1)^{2n-1}]$, thus proving the desired claim.
\end{proof}

\section{Loss of memory}\label{secLoss}

\subsection{Statement of the main theorem}

For $\vecv_0\in\S_1^{d-1}$, we define the probability density
$p_{\bn,\vecbeta_{\vecv_{0}}^+}(\vecv_{1},\xi,\vecv_{2})$ on 
$\S_1^{d-1}\times\R_{\geq 0}\times\S_1^{d-1}$ by 
\begin{equation}  \label{p0betdefrep}
p_{\bn,\vecbeta_{\vecv_{0}}^+}(\vecv_{1},\xi,\vecv_{2})
\,d\!\vol_{\S_1^{d-1}}(\vecv_{2})
=
\begin{cases}
\Phi_\bn\bigl(\xi,\vecw,(\vecbeta_{\vecv_{0}}^+(\vecv_{1})K(\vecv_{1}))_\perp\bigr)\,d\vecw
& \text{if $\vecv_{1}\in\scrV_{\vecv_{0}}$, $\vecv_{{2}}\in\scrV_{\vecv_{1}}$} \\
0 & \text{otherwise,}
\end{cases}
\end{equation}
with $\vecw=-\vecbeta_{\vece_{1}}^-(\vecv_{2} K(\vecv_{1}))_\perp\in
\{0\}\times\scrB_1^{d-1}$.

\begin{remark} \label{PALFBETPLUSCONTREM}
As in Remark \ref{PALFBETCONTREM}, the density
$p_{\bn,\vecbeta_{\vecv_{0}}^+}(\vecv_{1},\xi,\vecv_{2})$ is independent
of the choice of the function $K(\vecv_1)$;
also
$p_{\bn,\vecbeta_{\vecv_{0}}^+}(\vecv_{1},\xi,\vecv_{2})$ is continuous
at each point $(\vecv_1,\xi,\vecv_2)\in\S_1^{d-1}\times\R_{\geq 0}\times\S_1^{d-1}$
with $\vecv_1\in\scrV_{\vecv_0}$, $\vecv_2\in\scrV_{\vecv_1}$,
except possibly when $d=2$, $\xi>0$ and
$\vecbeta_{\vecv_1}^-\bigl(\vecv_2\bigr)=
-\vecbeta_{\vecv_0}^+(\vecv_1)R_{\vecv_1}$
($\Leftrightarrow$ $\vecv_2=\Theta_1(\vecv_1,
-\vecbeta_{\vecv_0}^+(\vecv_1)R_{\vecv_1})$).
Here $R_{\vecv_1}$ %
denotes reflection in the line $\R\vecv_1$.
\end{remark}
\begin{remark}\label{rot-inv}
Due to the spherical symmetry of the scattering map $\Theta$ (cf.~Section
\ref{secScatt}) we have
\begin{equation}
        p_{\bn,\vecbeta_{\vecv_{0} K}^+}(\vecv_{1} K,\xi,\vecv_{2} K)
        =p_{\bn,\vecbeta_{\vecv_{0}}^+}(\vecv_{1},\xi,\vecv_{2})
\end{equation}
for any $K\in\O(d)$.
\end{remark}
\begin{remark} \label{PALFBETPLUSEXPLICITREM}
The explicit formula in Remark \ref{PALFBETEXPLICITREM} carries
over directly to the present case.
\end{remark}
\begin{remark} \label{PNORMREMARK}
By \eqref{PALFINTEQ1} we have, whenever $\vecv_1\in\scrV_{\vecv_0}$,
\begin{equation} \label{p-norm}
	\int_{\S_1^{d-1}\times\RR_{>0}} p_{\bn,\vecbeta_{\vecv_{0}}^+}(\vecv_{1},\xi,\vecv_{2})
\, d\xi \,d\!\vol_{\S_1^{d-1}}(\vecv_{2})	= 1.
\end{equation}
\end{remark}
For the definition of $p_{\vecalf,\vecbeta_0}$, recall \eqref{exactpos2-1hit-tpdef}. The analogue of relation \eqref{p-norm} of course also holds for $p_{\vecalf,\vecbeta_0}$, again by \eqref{PALFINTEQ1}.

\begin{thm}\label{pathwayThm}
Fix a lattice $\scrL=\ZZ^d M_0$ and a point $\vecq\in\RR^d$,
set $\vecalf=-\vecq M_0^{-1}$,
and let $\vecbeta_0:\S_1^{d-1}\to\R^d$ be a $\C^1$ function.
If $\vecq\in\scrL$ %
we assume that
$(\vecbeta_0(\vecv)+\R_{>0}\vecv)\cap\scrB_1^d=\emptyset$ for all 
$\vecv\in\S_1^{d-1}$.
Then for any Borel probability measure $\lambda_0$ on $\S_1^{d-1}$ which is
absolutely continuous with respect to $\vol_{\S_1^{d-1}}$,
and for any bounded continuous function $f_0:\S_1^{d-1}\times (\R_{>0} \times \S_1^{d-1})^n\to \RR$, we have
\begin{align} \notag
\lim_{\rho\to 0}  \int_{\S_1^{d-1}} f_0\big(\vecv_0, \rho^{d-1} \tau_1(\vecq_{\rho,\vecbeta_0}(\vecv_0),\vecv_0;\rho),\vecv_1(\vecq_{\rho,\vecbeta_0}(\vecv_0),\vecv_0;\rho), \ldots, \vecv_n(\vecq_{\rho,\vecbeta_0}(\vecv_0),\vecv_0;\rho) \big) \, d\lambda_0(\vecv_0) &
\\ \label{pathwayEq}
=  \int_{\S_1^{d-1}\times (\R_{>0}\times \S_1^{d-1})^n}  
f_0\big(\vecv_0,\xi_1,\vecv_1,\ldots,\xi_n,\vecv_n \big)  p_{\vecalf,\vecbeta_0}(\vecv_0,\xi_1,\vecv_1) p_{\bn,\vecbeta_{\vecv_0}^+}(\vecv_1,\xi_2,\vecv_2) &
\\ \notag
\cdots p_{\bn,\vecbeta_{\vecv_{n-2}}^+}(\vecv_{n-1},\xi_n,\vecv_n)
\, d\lambda_0(\vecv_0)\,\prod_{k=1}^n  d\xi_k \,
d\!\vol_{\S_1^{d-1}}(\vecv_k) & .
\end{align}
\end{thm}

Note that the case $n=1$ specializes to the statement of Theorem 
\ref{exactpos2-1hit}. Hence from now on we will assume $n\geq 2$.

The proof of Theorem \ref{pathwayThm} is given in Section \ref{pathwayThmproofsec}, and is based on an iterative application of Theorem \ref{exactpos2-1hit}. The idea is that at the 
$n$th step, for a given, small $\rho>0$, we apply Theorem~\ref{exactpos2-1hit} once for each possible sequence of balls $\vecm_1+\scrB_\rho^d$, $\vecm_2+\scrB_\rho^d$, $\ldots$, 
$\vecm_{n-1}+\scrB_\rho^d$ ($\vecm_k\in\scrL$) causing the first $n-1$ 
collisions in the orbit of the flow $\varphi_t$.
For each such sequence $\{\vecm_k\}_{k=1}^{n-1}$ we apply 
Theorem \ref{exactpos2-1hit} with $\lambda$ as the probability measure on 
$\S_1^{d-1}$ describing the random variable 
$\vecv_{n-1}(\vecq_{\rho,\vecbeta_0}(\vecv_0),\vecv_0;\rho)$
\textit{conditioned} on $\vecv_0$ leading to $\{\vecm_k\}_{k=1}^{n-1}$,
and with $\vecbeta$ as the function
``$\vecv_{n-1}(\vecq_{\rho,\vecbeta_0}(\vecv_0),\vecv_0;\rho)
\mapsto\vecw_{n-1}(\vecq_{\rho,\vecbeta_0}(\vecv_0),\vecv_0;\rho)$''
(this makes sense once we restrict to $\vecv_0$ leading to the
fixed sequence $\{\vecm_k\}_{k=1}^{n-1}$, see Prop.\ \ref{REFLPROP}).
Since $\rho$ is small, $\vecbeta$ is well approximated by
$\vecbeta_{\vecv_{n-2}}^+$ (see Prop.\ \ref{REFLPROP}),
and the resulting limiting density is 
$p_{\vecnull,\vecbeta_{\vecv_{n-2}}^+}(\vecv_{n-1},\xi_n,\vecv_n)$.
Note that in Theorem \ref{exactpos2-1hit}, $p_{\vecalf,\vecbeta}$ does not 
depend on the choice of $\lambda$; hence our limiting density
$p_{\vecnull,\vecbeta_{\vecv_{n-2}}^+}(\vecv_{n-1},\xi_n,\vecv_n)$
is independent of the measure $d\lambda(\vecv_{n-1})$,
and thus independent of 
$\xi_1,\ldots,\xi_{n-1}$ and $\vecv_0,\ldots,\vecv_{n-3}$.

To make the above argument rigorous, we need to use the uniform version
Theorem \ref{exactpos2-1hit-unif} in place of Theorem \ref{exactpos2-1hit}, 
for note that both $\lambda$ and $\vecbeta$ depend on the hit sequence
$\{\vecm_k\}_{k=1}^{n-1}$, as well as on $\rho$.
For the application of Theorem \ref{exactpos2-1hit-unif}
we need to prove that all our $\lambda$'s are contained in some
equismooth family, and furthermore that $\vecbeta$ is indeed sufficiently well approximated
by $\vecbeta_{\vecv_{n-2}}^+$ as stated above.

\subsection{Sets of good initial velocities} \label{GOODSETSSECTION}
In this section we will define sets of initial velocities 
$\vecv_0$ with good properties,
and prove two lemmas which will later allow us to see that the complement of 
these sets have small $\lambda_0$-measure.

Let $n\geq 2$ and $\vecbeta_0$ be given as in Theorem \ref{pathwayThm}, and set
\begin{align} \label{GOODSETSSECTIONCchoice}
C=10 \max\Bigl(1,\sup_{\S_1^{d-1}}\|\vecbeta_0\|,
\sup_{\vech\in \T^1(\S_1^{d-1})} \|D_\vech \vecbeta_0\|,
\sup_{\vecy\in\scrL\setminus\{\bn\}} \|\vecy\|^{-1},
\sup_{\vecy\in(\scrL-\vecq)\setminus\{\bn\}} \|\vecy\|^{-1}\Bigr).
\end{align}
Let $\eta$ and $\rho$ be arbitrary numbers
subject to the conditions
\begin{align} \label{ETARHOCONDITIONS}
0<\eta<\sfrac{\pi-B_\Theta}{100}\quad\text{and}\quad
0<\rho<\min\bigl(\sfrac{\eta}{C^3C_{\eta}},\rho_0(n,\eta,C,\eta)\bigr),
\end{align}
where $\rho_0(n,\eta,C,\eta)$ is as in Proposition \ref{REFLPROP},
and $C_\eta$ as in \eqref{C1THETAETADEF}.
These two numbers $\eta$ and $\rho$ will be kept fixed throughout the present 
section.
Recall \eqref{FUEPSDEF}, and define $\omega(\eta)\in(0,\frac\pi 2)$
to be the angle such that 
\begin{align} \label{OMEGAETADEF}
\fU_{\omega(\eta)}=\fU_{2\sqrt\eta}\,\cup\,
-\vecbeta^-_{\vece_1}\Bigl(\scrV_{\vece_1}\setminus
\overline{\scrV_{\vece_1}^{20\eta}}\Bigr).
\end{align} 
We also fix a partition $\S_1^{d-1}=\bigsqcup_{j=1}^N D_j$ of the sphere
$\S_1^{d-1}$   %
into Borel subsets $D_1,\ldots,D_N$, each of positive volume 
and boundary of measure zero, and with diameter
$<\eta/C_{\eta}$ with respect to the metric $\varphi$.

For $\vecv\in\S_1^{d-1}$ and $k\geq 1$ we let
\begin{equation}
	\vecw_k(\vecv)=\vecw_k(\vecq_{\rho,{\vecbeta_0}}(\vecv),\vecv;\rho)\in\S_1^{d-1}, \qquad
	\vecm_k(\vecv)=\vecm_k(\vecq_{\rho,{\vecbeta_0}}(\vecv),\vecv;\rho)\in\scrL
\end{equation}
be the impact position and the ball label at the $k$th collision
for initial condition $(\vecq+\rho\vecbeta_0(\vecv),\vecv)\in\T^1(\scrK_\rho)$.
Let $\vecv_k(\vecv)=\vecv_k(\vecq_{\rho,{\vecbeta_0}}(\vecv),\vecv;\rho)\in\S_1^{d-1}$
be the velocity directly after this collision, and 
$\tau_k(\vecv)=\tau_k(\vecq_{\rho,{\vecbeta_0}}(\vecv),\vecv;\rho)\in\R_{>0}$
be the time elapsed between the $(k-1)$th and $k$th collision.
(Cf.\ \eqref{MNWNDEF} and \eqref{TAUNDEF}.)
We also set $\vecm_0=\vecq\in\R^d$ and
\begin{equation}
	\vecs_k(\vecv)=\vecs_k(\vecq_{\rho,{\vecbeta_0}}(\vecv),\vecv;\rho)
=\vecm_k(\vecv)-\vecm_{k-1}(\vecv)
\end{equation}
for $k\geq 1$.
Note that $\|\vecs_k(\vecv)\|\geq 10/C$ always holds,
by our choice of $C$. %

Given any $\veca\in\S_1^{d-1}$ we let $[\veca]$ be the unique set $D_j$
for which $\veca\in D_j$.
Set $\scrU^{(0)}=\S_1^{d-1}$, and define the subsets
$\scrU^{(0)}\supset\scrU^{(1)}\supset\scrU^{(2)}\supset\ldots$
recursively as follows:
For each $k\geq 1$, we let
$\scrU^{(k)}$ be the set of all $\vecv\in\scrU^{(k-1)}$ satisfying the
following three conditions: 
\begin{enumerate}
	\item[(I)] $\tau_k(\vecv)<\infty$, 
	\item[(II)] $[\vecv_k(\vecv)]\subset
\scrV_{\vecs_k}^{10\eta}$, and
	\item[(III)] for each
$\vecu\in [\vecv_k(\vecv)]$ 
there is some $\vecv'\in\scrU^{(k-1)}$ satisfying
\begin{align} \label{FULLYLIGHTEDK}
\begin{cases} \vecm_\ell(\vecv')=
\vecm_\ell(\vecv), \quad
\ell=1,\ldots,k;
\\
\vecv_k(\vecv')=\vecu.
\end{cases}
\end{align} 
\end{enumerate}
If $\vecv\in\scrU^{(k)}$ and
$\vecu\in[\vecv_k(\vecv)]$,
then every $\vecv'\in\scrU^{(k-1)}$
satisfying \eqref{FULLYLIGHTEDK} will actually lie in $\scrU^{(k)}$.

Given a vector $\vecv\in\scrU^{(k)}$ ($1\leq k\leq n$)
we let $M_\vecv^k$ be the set of 
those indices $i$ for which \eqref{FULLYLIGHTEDK} holds for some
$\vecu\in D_i$ and $\vecv'\in\scrU^{(k)}$. Then note that 
$\cup_{i\in M_\vecv^k} D_i\subset\scrV_{\vecs_k}^{10\eta}$ and for 
\textit{each} $\vecu\in \cup_{i\in M_\vecv^k} D_i$ there is some 
$\vecv'\in\scrU^{(k)}$ such that \eqref{FULLYLIGHTEDK} holds.
Write $\vec{\vecm}^k_\vecv:=(\vecm_0,\ldots,\vecm_k)$ where
$\vecm_0=\vecq$ and 
$\vecm_\ell=\vecm_\ell(\vecq_{\rho,{\vecbeta_0}}(\vecv),\vecv;\rho)$
($\ell=1,\ldots,k$) as before.
For each $\ell\in\{1,\ldots,k-1\}$ we have
$\vecv_\ell\in\scrV^{10\eta}_{\vecs_\ell}$ since $\vecv\in\scrU^{(\ell)}$,
and also, since the ray $\vecm_\ell+\rho\vecbeta^+_{\vecv_{\ell-1}}(\vecv_\ell)
+\R_{>0}\vecv_\ell$
hits $\vecm_{\ell+1}+\scrB_\rho^d$ we have
\begin{align} \label{ARCSININEQ}
\varphi(\vecv_\ell,\vecs_{\ell+1})\leq
\arcsin \frac{2\rho}{\|\vecs_{\ell+1}\|}
\leq \arcsin \frac{C\rho}{5}
<\sfrac 12 C\rho<\eta;
\end{align}
hence $\varphi(\vecs_\ell,\vecs_{\ell+1})>B_\Theta+9\eta$.
Also recall $\rho<\rho_0(n,\eta,C,\eta)$.
Hence the data $\vec{\vecm}_\vecv^k$, $\rho$ and $\vecbeta_0$ satisfy
all the assumptions of Proposition~\ref{REFLPROP}.
Now taking $\Phi=\Phi_{\rho,\vec{\vecm}^k_\vecv,\vecbeta_0}:\Delta
\dto\scrV^\eta_{\vecs_k}$ to be the
diffeomorphism in that proposition, it follows 
that for every $\vecu\in \cup_{i\in M_\vecv^k} D_i$ there is a {\em unique}
$\vecv'\in\scrU^{(k)}$ for which \eqref{FULLYLIGHTEDK} holds,
namely $\vecv'=\Phi^{-1}(\vecu)$.
Proposition \ref{REFLPROP} also implies that for any such
pair $\vecu,\vecv'$, we have regarding the starting point of the $\vecu$-ray 
leaving the $\vecm_k$-ball: 
$\vecbeta^+_{\vecv_{k-1}(\vecq_{\rho,{\vecbeta_0}}(\vecv'),\vecv';\rho)}(\vecu)
=\overline{\vecbeta}^+(\vecu)$
with $\overline{\vecbeta}^+=
\overline{\vecbeta}^+_{\rho,\vec{\vecm}^k_\vecv,\vecbeta_0}
:\scrV_{\vecs_k}^\eta\to\S_1^{d-1}$. 

\begin{lem} \label{UDIFFLEM}
Assume $1\leq k\leq n-1$, $\vecv\in\scrU^{(k)}$ and 
$\overline{\vecbeta}^+=
\overline{\vecbeta}^+_{\rho,\vec{\vecm}^k_\vecv,\vecbeta_0}$.
If $\vecv\notin \scrU^{(k+1)}$ then one of the following statements holds:
\begin{enumerate} \label{SCATTERINGCOND1TO4}
	\item[(i)] $\tau_{k+1}(\vecv)=\infty$,
	\item[(ii)] $\vecv_k(\vecv)\in\partial_{C\rho} \bigl(\cup_{i\in M_\vecv^k} D_i\bigr)$,
	\item[(iii)] $-\vecw_1\bigl(\rho\overline{\vecbeta}^+(\vecv_k),\vecv_k;\rho\bigr)K(\vecv_k)
\in \fU_{\omega(\eta)}$,
	\item[(iv)] $-\vecw_1\Bigl(\rho(\overline{\vecbeta}^+(\vecv_k)+3\vecv_k),
\vecv_k;(1+\eta)\rho\Bigr)K(\vecv_k)\in\fU_{\omega(\eta)}$.
\end{enumerate}
\end{lem}

\begin{proof}
Fix any vector $\vecv\in\scrU^{(k)}\setminus\scrU^{(k+1)}$, 
assume $\tau_{k+1}(\vecv)<\infty$ so that 
$\vecw_{k+1}=\vecw_{k+1}(\vecv)$ and
$\vecv_{k+1}=\vecv_{k+1}(\vecv)$ exist,
and write $D_j=[\vecv_{k+1}]$.
Now by the definition of $\scrU^{(k+1)}$ we have either
$D_j\not\subset\scrV_{\vecs_{k+1}}^{10\eta}$,
or else there is some $\vecu\in D_j$ such that there does \textit{not}
exist any $\vecv'\in\scrU^{(k)}$ with
\begin{align} \label{FULLYLIGHTEDKP1}
\begin{cases} \vecm_\ell(\vecv')= \vecm_\ell(\vecv), \quad
\ell=1,\ldots,k+1;
\\
\vecv_{k+1}(\vecv')=\vecu.
\end{cases}
\end{align}

First assume that $D_j\not\subset\scrV_{\vecs_{k+1}}^{10\eta}$, i.e.\
there is some $\vecu\in D_j$ with $\vecu\notin\scrV_{\vecs_{k+1}}^{10\eta}$.
Then $\varphi(\vecu,\vecs_{k+1})\leq B_\Theta+10\eta$;
also $\varphi(\vecv_{k+1},\vecu)<\eta$ since $D_j$ has diameter less than 
$\eta$, and $\varphi(\vecs_{k+1},\vecv_k)<\frac 12 C\rho<\eta$ as in 
\eqref{ARCSININEQ}.
Hence %
$\varphi(\vecv_{k+1},\vecv_k)<B_\Theta+12\eta$, and using 
$\vecw_{k+1}=\vecbeta^-_{\vecv_k}(\vecv_{k+1})$ and the spherical
invariance of $\vecbeta^-$ this implies
$-\vecw_{k+1} K(\vecv_k)\in\fU_{\omega(\eta)}$ (cf.\ \eqref{OMEGAETADEF}).
But 
\begin{equation}
	\vecw_{k+1}
=\vecw_1(\vecm_k+\rho\overline{\vecbeta}^+(\vecv_k),\vecv_k;\rho)
=\vecw_1(\rho\overline{\vecbeta}^+(\vecv_k),\vecv_k;\rho),
\end{equation}
since
$\scrL-\vecm_k=\scrL$. %
Hence (iii) in Lemma \ref{UDIFFLEM} holds.

It remains to consider the case when $D_j\subset\scrV_{\vecs_{k+1}}^{10\eta}$
but there is some $\vecu\in D_j$ (which we now consider as fixed)
such that \eqref{FULLYLIGHTEDKP1} does
not hold for any $\vecv'\in\scrU^{(k)}$.
As in the discussion preceding the lemma we have
$\varphi(\vecs_\ell,\vecs_{\ell+1})>B_\Theta+9\eta$
for all $\ell=1,\ldots,k$, so that 
the data $\vec{\vecm}_\vecv^{k+1}$, $\rho$ and $\vecbeta_0$ satisfy
all the assumptions of Proposition \ref{REFLPROP}.
Now $\vecu\in D_j\subset\scrV_{\vecs_{k+1}}^{10\eta}$ lies in the
range of $\Phi_{\rho,\vec{\vecm}_\vecv^{k+1},\vecbeta_0}$;
we set $\vecv'=\Phi_{\rho,\vec{\vecm}_\vecv^{k+1},\vecbeta_0}^{-1}(\vecu)
\in \S_1^{d-1}$.
Then by ``(b)$\Rightarrow$(a)'' in Proposition~\ref{REFLPROP},
the ray $\vecq_{\rho,{\vecbeta_0}}(\vecv')+\R_{>0}\vecv'$ hits
$\vecm_1+\scrB_\rho^d$; after scattering in this ball we get a ray which hits
$\vecm_2+\scrB_\rho^d$, and so on for 
$\vecm_3+\scrB_\rho^d,\ldots,\vecm_{k+1}+\scrB_\rho^d$,
and after the final scattering we get a ray with direction $\vecu$.
(So far we make no claim for any $j$ on whether or not the ray leaving
$\vecm_j+\scrB_\rho^d$ passes through any \textit{other}
ball in $\scrL+\scrB_\rho^d$ before hitting $\vecm_{j+1}+\scrB_\rho^d$.)

Let $\vecv_\ell'$ be the direction of the $\ell$th ray in the above
sequence ($\ell\in\{0,\ldots,k+1\}$, thus $\vecv_0'=\vecv'$, 
$\vecv_{k+1}'=\vecu$), so that
$\vecm_k+\rho\vecbeta^+_{\vecv'_{k-1}}(\vecv'_k)+\R_{>0}\vecv'_k$ is
the ray which leaves $\vecm_k+\scrB_\rho^d$ and hits 
$\vecm_{k+1}+\scrB_\rho^d$.
We have $\varphi(\vecs_{k+1},\vecv'_k)
<\frac 12 C \rho<\eta$
(cf.\ \eqref{ARCSININEQ}) 
and thus 
\begin{equation}
	\varphi(\vecs_k,\vecv'_k)\geq\varphi(\vecs_k,\vecs_{k+1})
-\varphi(\vecs_{k+1},\vecv'_k)>B_\Theta+9\eta-\eta
\end{equation}
so that
$\vecv'_k\in\scrV_{\vecs_k}^\eta$. Hence by ``(a)$\Rightarrow$(b)'' in Proposition \ref{REFLPROP},
$\vecv'$ also belongs to the domain of
$\Phi_{\rho,\vec{\vecm}_\vecv^{k},\vecbeta_0}$, and
$\vecv'_k=\Phi_{\rho,\vec{\vecm}_\vecv^{k},\vecbeta_0}(\vecv')$ and 
\begin{equation}
	\vecbeta^+_{\vecv'_{k-1}}(\vecv'_k)=\overline{\vecbeta}^+(\vecv'_k)
=\overline{\vecbeta}^+_{\rho,\vec{\vecm}_\vecv^k,\vecbeta_0}(\vecv'_k).
\end{equation}
From $\varphi(\vecs_{k+1},\vecv_k)<\frac 12 C \rho$ and
$\varphi(\vecs_{k+1},\vecv'_k)<\frac 12 C \rho$ we get
$\varphi(\vecv_k,\vecv'_k)<C \rho$.
Note that $\vecv_k\in\cup_{i\in M_\vecv^k} D_i$ since $\vecv\in\scrU^{(k)}$.
Now if $\vecv'_k$ lies outside $\cup_{i\in M_\vecv^k} D_i$ then there is
some point on the geodesic between $\vecv_k$ and $\vecv'_k$ which
lies in $\partial(\cup_{i\in M_\vecv^k} D_i)$, and hence %
$\vecv_k\in\partial_{C\rho}(\cup_{i\in M_\vecv^k} D_i)$,
i.e.\ (ii) in Lemma \ref{UDIFFLEM} holds.

It remains to consider the case $\vecv'_k\in\cup_{i\in M_\vecv^k} D_i$.
As in the discussion preceding the lemma,
$\vecv'_k\in\cup_{i\in M_\vecv^k} D_i$ and 
$\vecv'=\Phi_{\rho,\vec{\vecm}_\vecv^{k},\vecbeta_0}^{-1}(\vecv'_k)$
imply that $\vecv'\in\scrU^{(k)}$ and that
\eqref{FULLYLIGHTEDK} holds with $\vecv'_k$ in place of $\vecu$; thus
$\vecv'_\ell=\vecv_\ell(\vecq_{\rho,{\vecbeta_0}}(\vecv'),\vecv';\rho)$
for all $\ell=1,\ldots,k$.
We know from above that the ray 
$\vecm_k+\rho\vecbeta^+_{\vecv'_{k-1}}(\vecv'_k)+\R_{>0}\vecv'_k$
hits $\vecm_{k+1}+\scrB_\rho^d$;
let us choose $\tau'>0$ and $\vecw'_{k+1}\in\S_1^{d-1}$ so that the
point of impact is
\begin{equation}
	\vecm_k+\rho\vecbeta^+_{\vecv'_{k-1}}(\vecv'_k)+\tau'\vecv'_k
=\vecm_{k+1}+\rho\vecw'_{k+1}.
\end{equation}
After scattering in the ball $\vecm_{k+1}+\scrB_\rho^d$ 
we get a ray with direction $\vecu$. It follows that the ray 
\begin{equation}
	\vecm_k+\rho\vecbeta^+_{\vecv'_{k-1}}(\vecv'_k)+\R_{>0}\vecv'_k
\end{equation}
must intersect some \textit{other} ball in $\scrL+\scrB_\rho^d$ before it hits
$\vecm_{k+1}+\scrB_\rho^d$, for otherwise we would conclude
$\vecm_{k+1}(\vecq_{\rho,{\vecbeta_0}}(\vecv'),\vecv';\rho)=\vecm_{k+1}$
and 
$\vecv_{k+1}(\vecq_{\rho,{\vecbeta_0}}(\vecv'),\vecv';\rho)=\vecu$,
i.e.\ \eqref{FULLYLIGHTEDKP1} would hold for our $\vecv'\in\scrU^{(k)}$, 
contrary to our assumption. Thus there is some point
\begin{equation}
	\vecm'=\vecm_{k+1}(\vecq_{\rho,{\vecbeta_0}}(\vecv'),\vecv';\rho)\in\scrL, \qquad \vecm'\neq\vecm_{k+1},
\end{equation}
and some $t'\in (0,\tau')$ such that
$\vecm_k+\rho\vecbeta^+_{\vecv'_{k-1}}(\vecv'_k)+t'\vecv'_k
\in\vecm'+\scrB_\rho^d$.
Note that $\|\vecm'-\vecm_{k+1}\|\geq 10/C>10^3\rho$,
and hence $t'<\tau'-100\rho$.

Using $\varphi(\vecv_k,\vecv_k')<C\rho<\eta$ and
$\vecv_k\in %
\scrV_{\vecs_k}^{10\eta}$ we see that the ($\S_1^{d-1}$-)geodesic
from $\vecv_k$ to $\vecv_k'$ is contained inside 
$\scrV_{\vecs_k}^{\eta}$; hence from
$\sup_{\vech\in\T^1(\scrV_{\vecs_k}^\eta)} \|D_\vech \overline{\vecbeta}^+\|
<C_{\eta}$ (see Proposition \ref{REFLPROP}) we deduce
\begin{equation}
	\|\overline{\vecbeta}^+(\vecv_k)-\overline{\vecbeta}^+(\vecv'_k)\|
<C_{\eta} C\rho<\eta.
\end{equation}
Similarly, writing
$\overline{\vecbeta}^-_{k+1}:=\overline{\vecbeta}^-_{\rho,
\vec{\vecm}_\vecv^{k+1},\vecbeta_0}$ and using
$\vecw_{k+1}=\overline{\vecbeta}^-_{k+1}(\vecv_{k+1})$,
$\vecw_{k+1}'=\overline{\vecbeta}^-_{k+1}(\vecu)$
and $\vecv_{k+1},\vecu\in D_j\subset\scrV_{\vecs_{k+1}}^{10\eta}$
(thus $\varphi(\vecv_{k+1},\vecu)<\frac{\eta}{C_{\eta}}$), we obtain
\begin{equation}
	\|\vecw_{k+1}-\vecw'_{k+1}\|<\eta .
\end{equation}
Thus the line segment 
$\vecm_k+\rho\overline{\vecbeta}^+(\vecv_k')+(0,\tau')\vecv_k'$
has both its endpoints
at distance $<\rho\eta$ from the corresponding endpoints of the line segment
$\vecm_k+\rho\overline{\vecbeta}^+(\vecv_k)+(0,\tau_{k+1})\vecv_k$.
Hence %
$|\tau'-\tau_{k+1}|<2\rho\eta$,
and there exists some $t>0$ with $|t-t'|<2\rho\eta$ such that
\begin{equation}
	\bigl\|(\vecm_k+\rho\overline{\vecbeta}^+(\vecv_k)+t\vecv_k)-
(\vecm_k+\rho\overline{\vecbeta}^+(\vecv_k')+t'\vecv'_k)\bigr\|<\rho\eta.
\end{equation}
It follows from this that $\vecm_k+\rho\overline{\vecbeta}^+(\vecv_k)+t\vecv_k
\in\vecm'+\scrB_{\trho}^d$, with $\trho=(1+\eta)\rho$, and therefore
$\tau_1'':=\tau_1(\vecm_k+\rho(\overline{\vecbeta}^+(\vecv_k)+3\vecv_k),\vecv_k,\trho)$
satisfies $\tau_1''<t-3\rho$.
Take $\vecm''\in\scrL$ and set
$\vecw''_1=\vecw_1(\vecm_k+\rho(\overline{\vecbeta}^+(\vecv_k)+3\vecv_k),\vecv_k,\trho)
\in\S_1^{d-1}$, so that
\begin{align}
\vecm_k+\rho(\overline{\vecbeta}^+(\vecv_k)+3\vecv_k)+\tau_1''\vecv_k=\vecm''+\trho\vecw''_1.
\end{align}

Note that if the ray 
$\vecm_k+\rho\overline{\vecbeta}^+(\vecv_k)+\R_{>0}\vecv_k$ would also
intersect the slightly smaller ball $\vecm''+\scrB_\rho^d$ then 
$\tau_1(\vecm_k+\rho(\overline{\vecbeta}^+(\vecv_k)+3\vecv_k),\vecv_k,\rho)<\tau_1''+3\rho$,
and we would get a contradiction:
\begin{align} \label{CONTRINEQSEQ}
\tau_{k+1} %
=\tau_1(\vecm_k+\rho(\overline{\vecbeta}^+(\vecv_k)+3\vecv_k),\vecv_k,\rho)
+3\rho<\tau_1''+6\rho<t+3\rho<t'+4\rho &
\\ \notag
<\tau'-96\rho<\tau_{k+1}-95\rho & .
\end{align}
Thus the ray $\vecm_k+\rho\overline{\vecbeta}^+(\vecv_k)+\R_{>0}\vecv_k$ 
intersects $\vecm''+\scrB_\trho^d$ but \textit{not} $\vecm''+\scrB_\rho^d$.
This implies 
$-\vecw''_1 K(\vecv_k)\in\fU_{2\sqrt\eta}\subset\fU_{\omega(\eta)}$.
Hence (iv) in Lemma \ref{UDIFFLEM} holds.
\end{proof}

We next give the analogue of Lemma \ref{UDIFFLEM} in the case $k=0$.
Let us define $\tbe_0\in \C^1(\S_1^{d-1})$ by
\begin{align} \label{TBE0DEF}
\tbe_0(\vecv)=(1+\eta)^{-1}\bigl(\vecbeta_0(\vecv)+(C+3)\vecv\bigr).
\end{align}

\begin{lem} \label{UDIFFLEM0}
Set $\trho=(1+\eta)\rho$. For each $\vecv\in\S_1^{d-1}\setminus\scrU^{(1)}$, one of the following statements holds:
\begin{enumerate}
	\item[(i)] $\tau_{1}(\vecv)=\infty$,
	\item[(ii)] $-\vecw_1\bigl(\vecq+\rho\vecbeta_0(\vecv),\vecv;\rho\bigr)K(\vecv)
\in \fU_{\omega(\eta)}$,
	\item[(iii)] $-\vecw_1\Bigl(\vecq+\trho\tbe_0(\vecv),\vecv;\trho\Bigr)K(\vecv)
\in\fU_{\omega(\eta)}$.
\end{enumerate}
\end{lem}

\begin{proof}
This is very similar to the proof of Lemma \ref{UDIFFLEM}, but easier
in certain respects.
\end{proof}

\subsection{Proof of Theorem \ref{pathwayThm}}\label{pathwayThmproofsec}

Let $n\geq 2,\lambda_0,\vecbeta_0,f_0$ be given as in the statement of the 
theorem.

\subsubsection{Some reductions}
Set $X=\S_1^{d-1}\times(\R_{>0}\times\S_1^{d-1})^n$ and let 
$\nu$ be the Borel measure on $X$ which appears in the
right hand side of \eqref{pathwayEq} in Theorem~\ref{pathwayThm}, i.e.\ set
\begin{align} \label{pathwaymeasure}
\nu(M)=\int_M p_{\vecalf,\vecbeta_0}(\vecv_0,\xi_1,\vecv_1) 
\Bigl(\prod_{j=1}^{n-1}
p_{\bn,\vecbeta_{\vecv_{j-1}}^+}(\vecv_j,\xi_{j+1},\vecv_{j+1})\Bigr)
\, d\lambda_0(\vecv_0)\, d\xi_1\cdots d\!\vol_{\S_1^{d-1}}(\vecv_n)
\end{align}
for every Borel subset $M\subset X$. Note that repeated application of \eqref{p-norm}, and the analogous relation for $p_{\vecalf,\vecbeta_0}(\vecv_0,\xi_1,\vecv_1)$, yields $\nu(X)=1$.
Set 
\begin{align}
X_0=\{(\vecv_0,\xi_1,\ldots,\vecv_n)\in X \col
\varphi(\vecv_{j-1},\vecv_{j})>B_\Theta,\: \forall j\in\{1,2,\ldots,n\}\}.
\end{align}
This is an open subset of $X$ with $\nu(X\setminus X_0)=0$, i.e.\ $\nu(X_0)=1$.
Hence we may now assume, without loss of generality, that \textit{$f_0$ has
compact support contained inside $X_0$.}
The general case of Theorem~\ref{pathwayThm} then follows by a
standard approximation argument.

We define functions $f_m: \S_1^{d-1}\times(\R_{>0}\times\S_1^{d-1})^{n-m}
\to\R$ for $m=1,\ldots,n-1$ by the following recursive formula:
\begin{align} \notag
f_m\bigl(\vecv_0,\xi_1,\vecv_1,\ldots,\xi_{n-m},\vecv_{n-m}\bigr)
:=\int_{\S_1^{d-1}}\int_{\R_{>0}} f_{m-1}\bigl(
\vecv_0,\xi_1,\vecv_1,\ldots,\vecv_{n-m},\xi,\vecv\bigr)\, \hspace{50pt} &
\\ \label{FMRECDEF}
\times p_{\bn,\vecbeta_{\vecv_{n-m-1}}^+}(\vecv_{n-m},\xi,\vecv)
\, d\xi\, d\!\vol_{\S_1^{d-1}}(\vecv) & .
\end{align}
The point of this %
is that now %
the right hand side of \eqref{pathwayEq} in Theorem~\ref{pathwayThm}
can be expressed as
\begin{align} \label{POINTOFRECDEF}
\int_X f_0\, d\nu=
\int_{\S_1^{d-1}}\int_{\R_{>0}}\int_{\S_1^{d-1}} 
f_{n-1}(\vecv_0,\xi_1,\vecv_1)
p_{\vecalf,\vecbeta_0}(\vecv_0,\xi_1,\vecv_1)\,d\lambda_0(\vecv_0)
\, d\xi_1\, d\!\vol_{\S_1^{d-1}}(\vecv_1).
\end{align}

Since $f_0$ has compact support contained in $X_0$ there exists
a constant $\delta>0$ such that $f_0(\vecv_0,\xi_1,\ldots,\vecv_n)=0$
unless $\varphi(\vecv_{j-1},\vecv_j)>B_\Theta+\delta$ for all $j=1,2,\ldots,n$.
Hence by induction each $f_m$ has compact support, and we have
$f_m(\vecv_0,\xi_1,\ldots,\vecv_{n-m})=0$ unless 
$\varphi(\vecv_{j-1},\vecv_j)>B_\Theta+\delta$ for all $j=1,2,\ldots,n-m$.
Using this fact together with Remark \ref{PALFBETSPECREM}
(rewriting \eqref{FMRECDEF} via \eqref{p0betdefrep})
one shows that \textit{each $f_m$ is continuous.}
Furthermore, by repeated use of \eqref{p-norm}
we see that $\sup |f_m|\leq \sup |f_0|$ for all $m=0,1,\ldots,n-1$.

Finally,
by a standard approximation argument we may assume without loss of generality
that $\lambda_0$ has a continuous density, i.e.\ that 
\begin{align} \label{LAMBDA0GDEF}
\lambda_0=g\cdot\vol_{\S_1^{d-1}}\quad\text{for some continuous function }
\: g:\S_1^{d-1}\to\R_{\geq 0}.
\end{align}
Indeed, note that for $\lambda_0=g\cdot\vol_{\S_1^{d-1}}$, $g\in \L^1$,
the right hand side of \eqref{pathwayEq} depends linearly on $g$
and is bounded in absolute value by $(\sup |f_0|)\cdot \|g\|_{\L^1}$;
the same is true for the left hand side of \eqref{pathwayEq}, 
for each fixed $\rho$.
We can thus mimic the first paragraph in the proof of Theorem~1.2 
in \cite[Sec.\ 9.2]{partI}.

\subsubsection{Choosing $\eta,\rho_0$ and $F_1,F_2,F_3$}
Let $\ve>0$ be given. On the next few pages we will describe how
to choose auxiliary positive numbers $\eta$ and $\rho_0$,
as well as families $F_1$, $F_2$, $F_3$ of measures and functions
to use in applications of Theorem \ref{exactpos2-1hit-unif}.
The goal is to set up things so that we will be able to prove
that the two sides in \eqref{pathwayEq} in Theorem \ref{pathwayThm}
differ by at most $O(\ve)$ for every $\rho\in (0,\rho_0)$;
see \eqref{pathwayproofclaim} below.

Recall the definition of $\omega(\eta)$, \eqref{OMEGAETADEF},
and note that $\lim_{\eta\to 0}\omega(\eta)=0$.
As in the discussion leading to \eqref{STEP1VEPRIMECHOICEIMPL} and
\eqref{STEP1TCHOICEIMPL} 
we see that we may fix 
$0<\eta<\min\bigl(\frac{\pi-B_\Theta}{100},\ve\bigr)$ 
so small that $2\omega(\eta)<\frac\pi 2$ and so that
the following two inequalities hold for 
all absolutely continuous probability measures $\lambda$ on $\S_1^{d-1}$,
all continuous functions $\vecbeta:\S_1^{d-1}\to\R^d$,
any $\vecalf'\in\{\bn,\vecalf\}$, 
and any two measurable functions
$\vecz_1,\vecz_2:\supp(\lambda)\to\{0\}\times\overline{\scrB_1^{d-1}}$
satisfying $\|\vecz_1(\vecv)-\vecz_2(\vecv)\|\leq \eta$ for all 
$\vecv\in\supp(\lambda)$:
\begin{align} \label{pwetachoice2}
& \int_0^\infty \int_{(\fU_{2\omega(\eta)})_\perp} \int_{\S_1^{d-1}} 
\Phi_{\vecalf'}\bigl(\xi,\vecw,
(\vecbeta(\vecv)K(\vecv))_\perp\bigr)\,d\lambda(\vecv)d\vecw d\xi<\ve;
\\ \label{VEPRIMECHOICEIMPL}
& \int_0^\infty \int_{\{0\}\times\scrB^{d-1}_1} \int_{\supp(\lambda)} 
\Bigl | \Phi_\bn\bigl(\xi,\vecw,\vecz_1(\vecv)\bigr)
-\Phi_\bn\bigl(\xi,\vecw,\vecz_2(\vecv)\bigr)\Bigr |
\,d\lambda(\vecv)\,d\vecw\,d\xi<\ve.
\end{align}

As in Section \ref{GOODSETSSECTION} we fix a partition
$\S_1^{d-1}=\bigsqcup_{j=1}^N D_j$ of the sphere
$\S_1^{d-1}$   %
into Borel subsets $D_1,\ldots,D_N$, each of positive volume 
and boundary of measure zero, and with diameter
$<\eta/C_{\eta}$ with respect to the metric $\varphi$.
Given $\vecs\in\RR^d\setminus\{\bn\}$ 
and any subset $M\subset\{1,\ldots,N\}$ with
$\lambda_\vecs(\cup_{j\in M} D_j)>0$, we let
$\lambda_{M,\vecs}$ be the restriction of the measure $\lambda_\vecs$
(cf.\ Sec.\ \ref{reflmapssec})
to $\cup_{j\in M} D_j$, rescaled to be a probability measure:
\begin{align}
\lambda_{M,\vecs}=\lambda_\vecs\bigl(\cup_{M} D_j\bigr)^{-1}
\cdot {\lambda_\vecs}_{|(\cup_{M} D_j)}.
\end{align}
Let $F_1$ be the following family of probability measures:
\begin{align} \label{F1DEF}
F_1:=\{\lambda_0\} \bigcup 
\bigl\{\lambda_{M,\vecs}\col \emptyset\neq M\subset\{1,\ldots,N\},\:
\vecs\in\S_1^{d-1},\:
(\cup_{M} D_j) \subset \scrV_\vecs^{10\eta}\bigr\}.
\end{align}
Note that $\lim_{h\to 0^+}\vol_{\S_1^{d-1}}(\partial_h D_j)=0$ for each $j$,
since $D_j$ has boundary of measure zero.
Also note that $\scrV_\vecs^{10\eta}$ has compact closure in $\scrV_\vecs$,
and recall \eqref{LAMBDA0DEF}.
From these facts together with \eqref{LAMBDA0GDEF}, 
it follows that \textit{$F_1$ is equismooth}.

Let us fix a continuous function 
$H:\S_1^{d-1}\times\R_{>0}\times\S_1^{d-1}\to [0,1]$ of the form
$H(\vecv_0,\xi,\vecv_1)=H_0(\varphi(\vecv_0,\vecv_1))$ such that
$H(\vecv_0,\xi,\vecv_1)=1$ whenever $\varphi(\vecv_0,\vecv_1)\leq B_\Theta$ or
$\varphi(\vecv_0,-\vecbeta_{\vecv_0}^-(\vecv_1))\geq\frac{\pi}2-\omega(\eta)$,
and $H(\vecv_0,\xi,\vecv_1)=0$ whenever 
$\varphi(\vecv_0,-\vecbeta_{\vecv_0}^-(\vecv_1))\leq\frac{\pi}2-2\omega(\eta)$.
This is possible since it follows from our assumptions on $\Theta_1$ 
in Sec.\ \ref{secScatt} that
$\varphi(\vecv_0,-\vecbeta_{\vecv_0}^-(\vecv_1))$ only depends on
$\varphi(\vecv_0,\vecv_1)$, as a strictly decreasing $\C^1$ function.
Now let $F_2$ be the following family of functions on
$\S_1^{d-1}\times\R_{>0}\times\S_1^{d-1}$,
\begin{align} \notag
F_2=\bigl\{H\bigr\} \cup \Bigl\{(\vecv,\xi,\vecv')\mapsto
f_m(\vecv_0,\xi_1,\ldots,\xi_{n-m-1},\vecv,\xi,\vecv')
\col \hspace{140pt} &
\\ \label{F2DEF}
0\leq m \leq n-1,\:
(\vecv_0,\xi_1,\ldots,\xi_{n-m-1})\in (\S_1^{d-1}\times\R_{>0})^{n-m-1}\Bigr\}.
\end{align}
Using the fact that each $f_m$ has compact support we see that
$F_2$ is equicontinuous and uniformly bounded.

Let $C$ be fixed as in \eqref{GOODSETSSECTIONCchoice},
and let, for any $\vecq'\in\R^d$:
\begin{align} \notag
F_3^{(\vecq')}=\Bigl\{\vecbeta:\S_1^{d-1}\to\R^d \:\:\col\:\:
\vecbeta\text{ is }\C^1, \:\:
\sup \|\vecbeta\|\leq 3C, \:\:
\sup_{\vech\in \T^1(\S_1^{d-1})} \|D_\vech \vecbeta\|\leq 
2C+\sfrac 4\eta+C_{\eta} , \hspace{10pt} &
\\ \label{F3DEF}
\text{if $\vecq'\in\scrL$:} \quad
(\vecbeta(\vecv)+\R_{>0}\vecv)\cap\scrB_1^d=\emptyset, 
\:\forall\vecv\in\S_1^{d-1}\Bigr\}. & %
\end{align}

Let $\rho_0(n,\eta,C,\eta)$ be as in Proposition \ref{REFLPROP}.
Now fix $0<\rho_0<\min\Bigl(\frac{\eta}{C^3C_{\eta}},
\rho_0(n,\eta,C,\eta)\Bigr)$
so small that for $\vecq'\in\{\bn,\vecq\}$ and 
all $\rho\in(0,2\rho_0)$, 
$\lambda\in F_1$, $f\in F_2$ and $\vecbeta\in F_3^{(\vecq')}$, we have
(here and below we write $\vecalf'=\bn$ if $\vecq'=\bn$; $\:$
$\vecalf'=\vecalf$ if $\vecq'=\vecq$)
\begin{align} \notag
\biggl |
\int_{\S_1^{d-1}} f\big(\vecv, \rho^{d-1} \tau_1(\vecq'+\rho\vecbeta(\vecv),\vecv;\rho), 
\vecv_1(\vecq'+\rho\vecbeta(\vecv),\vecv;\rho)\big) \, d\lambda(\vecv) 
\hspace{100pt} & 
\\ \label{exactpos2-1hit-unif-used}
-\int_{\S_1^{d-1}} \int_{\R_{>0}} \int_{\S_1^{d-1}} f\big(\vecv,\xi,\vecv_1\big) 
p_{\vecalf',{\vecbeta}}(\vecv,\xi,\vecv_1) \, d\lambda(\vecv)\, 
d\xi\,d\!\vol_{\S_1^{d-1}}(\vecv_1)\biggr | < %
\ve. &
\end{align}
This is possible by Theorem \ref{exactpos2-1hit-unif}.
Let us shrink $\rho_0$ further if necessary, so that 
(with $K_\eta$ as in Lemma \ref{COMPARABILITYLEM})
\begin{align} 
\label{RHOVOLBDRYASS}
& \vol_{\S_1^{d-1}}\bigl(\partial_{C\rho_0}(D_j)\bigr)<
K_\eta^{-1}\ve\vol_{\S_1^{d-1}}(D_j), \qquad 
\forall j\in\{1,\ldots,N\};
\\ \label{GRHOCONTASS}
& \forall \vecv,\vecv'\in\S_1^{d-1}: \qquad
\varphi(\vecv,\vecv')\leq 2C^2\rho_0
\Longrightarrow |g(\vecv)-g(\vecv')|<\ve;
\end{align}
and also, for all $m\in\{0,\ldots,n-1\}$ and all points
$(\vecv_0,\xi_1,\ldots,\vecv_{n-m})$ and
$(\vecv_0',\xi_1',\ldots,\vecv_{n-m}')$ in 
$\S_1^{d-1}\times(\R_{>0}\times\S_1^{d-1})^{n-m}$:
\begin{align} \notag
 \varphi(\vecv_j',\vecv_j)\leq C^2\rho_0 \quad (j=0,\ldots,n-m), \qquad
|\xi_j'-\xi_j|\leq 2C\rho_0^d \quad (j=1,\ldots,n-m) &
\\ \label{FCONTRHOASS}
 \Longrightarrow \Bigl |f_m\bigl(\vecv_0',\xi_1',\ldots,\vecv_{n-m}'\bigr)
-f_m\bigl(\vecv_0,\xi_1,\ldots,\vecv_{n-m}\bigr)\Bigr |
< & \ve.
\end{align}
Here %
\eqref{FCONTRHOASS} can be achieved since each $f_m$
is continuous with compact support.

Having thus fixed a choice of $\rho_0$, we now claim that for all 
$\rho\in(0,\rho_0)$ we have
\begin{align} \label{pathwayproofclaim}
& \biggl | \int_{\S_1^{d-1}} f_0\big(\vecv_0, \rho^{d-1} \tau_1(\vecq_{\rho,\vecbeta_0}(\vecv_0),\vecv_0;\rho), \ldots, \vecv_n(\vecq_{\rho,\vecbeta_0}(\vecv_0),\vecv_0;\rho) \big) d\lambda_0(\vecv_0) 
-\int_X f_0 \, d\nu \biggr | \ll \ve. \hspace{-10pt}
\end{align}
Here and in any $\ll$ bound in the remainder of the proof,
the implied constant depends only on $f_0$, $\lambda_0$, $d$, $n$, $C$
and $\Theta$ (the scattering map).
Since $\ve>0$ was arbitrary, the bound \eqref{pathwayproofclaim} 
will complete the proof of Theorem \ref{pathwayThm}.

\subsubsection{Bounding $\lambda_0(\S_1^{d-1}\setminus\scrU^{(n)})$}\label{BOUNDINGLAMBDA0DIFF}
Take an arbitrary $\rho\in(0,\rho_0)$, and keep $\rho$ fixed for the rest 
of this proof.
Note that \eqref{ETARHOCONDITIONS} holds.
We now define the subsets
$\S_1^{d-1}=\scrU^{(0)}\supset\scrU^{(1)}\supset\scrU^{(2)}\supset\ldots
\supset\scrU^{(n)}$ as in Section \ref{GOODSETSSECTION}.
We will prove that $\lambda_0(\S_1^{d-1}\setminus\scrU^{(n)})$ is small.

Let us first make explicit the conclusion from
\eqref{exactpos2-1hit-unif-used} in the case $f=H$.
In this case, by changing variables (via \eqref{exactpos2-1hit-tpdef},
\eqref{exactpos2-1hit-subst})
in the triple integral in \eqref{exactpos2-1hit-unif-used},
and using the definition of $H$,
we see that the triple integral is less than the expression in
\eqref{pwetachoice2}, and hence $<\ve$.
Note also that if $\vecv\in\S_1^{d-1}$ is such that
$\vecw_1=\vecw_1(\vecq'+\rho\vecbeta(\vecv),\vecv;\rho)$
satisfies $-\vecw_1K(\vecv)\in\fU_{\omega(\eta)}$, then
$\vecv_1=\vecv_1(\vecq'+\rho\vecbeta(\vecv),\vecv;\rho)
=\Theta_1(\vecv,\vecw_1)$ satisfies $H(\vecv,\xi,\vecv_1)=1$ for all $\xi>0$.
Hence \eqref{exactpos2-1hit-unif-used} implies
that for all $\vecq'\in\{\bn,\vecq\}$, $\rho\in(0,2\rho_0)$, 
$\lambda\in F_1$ and $\vecbeta\in F_3^{(\vecq')}$, we have
\begin{align} \label{LAMBDAEDGESMALL}
\lambda\bigl(\bigl\{\vecv\in\S_1^{d-1}\col 
-\vecw_1(\vecq'+\rho\vecbeta(\vecv),\vecv;\rho)K(\vecv)\in\fU_{\omega(\eta)}
\bigr\}\bigr)<2\ve.
\end{align}

We now apply Lemma \ref{UDIFFLEM0} to prove that $\lambda_0(\S_1^{d-1}\setminus
\scrU^{(1)})$ is small. 
The set $\{\vecv\in\S_1^{d-1}\col\tau_1=\infty\}$ has measure zero with respect
to $\vol_{\S_1^{d-1}}$ (see Sec.\ \ref{firstcoll}), and hence also with 
respect to $\lambda_0$.
Note $\trho<2\rho_0$ and $\vecbeta_0,\tbe_0
\in F_3^{(\vecq)}$, cf.\ \eqref{TBE0DEF} and \eqref{F3DEF}; hence we may apply
\eqref{LAMBDAEDGESMALL} with
$\vecq'=\vecq$, $\lambda=\lambda_0$ and $\vecbeta_0$, $\rho$ 
resp.\ $\tbe_0$, $\trho$. 
This implies that the set of $\vecv$'s %
which satisfy (ii) or (iii)
in Lemma \ref{UDIFFLEM0} has $\lambda_0$-measure $<4\ve$.
Thus
\begin{align} \label{LAMBDAUU1SMALL}
\lambda_0(\S_1^{d-1}\setminus\scrU^{(1)})<4\ve.
\end{align}

Next take $k\in\{1,\ldots,n-1\}$;
we will apply Lemma \ref{UDIFFLEM} to prove that $\lambda_0(\scrU^{(k)}
\setminus\scrU^{(k+1)})$ is small.
We call any two vectors $\vecv,\vecv'\in\scrU^{(k)}$
equivalent if and only if $\vec{\vecm}_\vecv^k=\vec{\vecm}_{\vecv'}^k$.
Let $\scrU'\subset\scrU^{(k)}$ be any fixed equivalence class for this
relation. %
Then the subset $\emptyset\neq M_\vecv^k\subset\{1,\ldots,N\}$ 
and the functions $\Phi_{\rho,\vec{\vecm}_\vecv^{k},\vecbeta_0}$,
and $\overline{\vecbeta}^+_{\rho,\vec{\vecm}_\vecv^{k},\vecbeta_0}$ 
are independent of $\vecv\in\scrU'$, and we will write simply
$M$, $\Phi$ and $\overline{\vecbeta}^+$ for these.
Note in particular $\scrU'=\Phi^{-1}(\cup_{i\in M} D_i)$.

We need to modify $\overline{\vecbeta}^+$ to get a function in 
$F_3^{(\bn)}$.
Let us fix $\C^1$ functions $c_1,c_2:\R_{>0}\to [0,1]$ with the properties
\begin{align} \label{CONDONC}
& c_1(\varphi)=0, \: \forall\varphi\leq B_\Theta+5\eta; \qquad
c_1(\varphi)=1, \: \forall\varphi\geq B_\Theta+6\eta; \qquad 
|c_1'(\varphi)|\leq 2\eta^{-1}, \: \forall\varphi>0; 
\\ \notag
& c_2(\varphi)=1, \: \forall\varphi\leq B_\Theta+6\eta; \qquad
c_2(\varphi)=0, \: \forall\varphi\geq B_\Theta+7\eta; \qquad 
|c_2'(\varphi)|\leq 2\eta^{-1}, \: \forall\varphi>0.
\end{align}
Now define $\ttbe=\ttbe_{\rho,\vec{\vecm}^k_\vecv,\vecbeta_0}\in
\C^1(\S_1^{d-1})$ by
\begin{align} \label{TTBEDEF}
\ttbe(\vecv)=\ttbe_{\rho,\vec{\vecm}^k_\vecv,\vecbeta_0}(\vecv)
=\begin{cases}
c_1(\varphi(\vecv,\vecs_k))\cdot\overline{\vecbeta}^+(\vecv)+
c_2(\varphi(\vecv,\vecs_k))\cdot\vecv
& \text{if } \: \vecv\in\scrV_{\vecs_k}^{\eta}
\\
\vecv & \text{otherwise}.
\end{cases}
\end{align}
We then have
\begin{align} \label{TTBINF3}
\ttbe_{\rho,\vec{\vecm}^k_\vecv,\vecbeta_0}(\vecv)=
\overline{\vecbeta}^+_{\rho,\vec{\vecm}^k_\vecv,\vecbeta_0}(\vecv),
\quad\forall\vecv\in\scrV_{\vecs_k}^{10\eta}.
\end{align}
Furthermore, for $\vecv\in\S_1^{d-1}$ with
$B_\Theta+4\eta<\varphi(\vecv,\vecs_k)<B_\Theta+8\eta$ we have, for any
$\vech\in T^1_\vecv(\S_1^{d-1})$,
\begin{equation} \label{DHTTBEBOUND}
	\bigl\|D_\vech\ttbe(\vecv)\bigr\|=
\bigl\|(D_\vech c_1(\varphi(\vecv,\vecs_k)))\cdot\overline{\vecbeta}^+(\vecv)
+c_1\cdot D_\vech\overline{\vecbeta}^+(\vecv)
+(D_\vech c_2(\varphi(\vecv,\vecs_k)))\cdot\vecv+c_2\cdot\vech\bigr\|
\leq \frac 4{\eta}+C_{\eta}+1
\end{equation}
by Proposition \ref{REFLPROP} and \eqref{CONDONC}.
From this one easily deduces
$\ttbe_{\rho,\vec{\vecm}^k_\vecv,\vecbeta_0}\in F_3^{(\bn)}.$

Using Lemma \ref{UDIFFLEM},
$\cup_{i\in M} D_i\subset\scrV_{\vecs_k}^{10\eta}$ and \eqref{TTBINF3}
we have
\begin{align} \notag
\Phi\bigl(\scrU'\setminus\scrU^{(k+1)}\bigr)
\quad \subset \quad \partial_{ C\rho}\bigl(\cup_{i\in M} D_i\bigr)\:\:
& \cup \:\:\bigl\{\vecv'\in\cup_{i\in M} D_i\col
\tau_1(\rho\ttbe(\vecv'),\vecv',\rho\bigr)=\infty\bigr\}
\\ \label{UDIFFCAPT}
& \cup \:\:\bigl\{\vecv'\in\cup_{i\in M} D_i\col
-\vecw_1\bigl(\rho\ttbe(\vecv'),\vecv',\rho\bigr)K(\vecv')
\in \fU_{\omega(\eta)}\bigr\}
\\ \notag
& \cup \:\: \bigl\{\vecv'\in\cup_{i\in M} D_i\col
-\vecw_1\bigl(\trho\tbe(\vecv'),\vecv',\trho\bigr)K(\vecv')
\in \fU_{\omega(\eta)}\bigr\},
\end{align}
where $\trho=(1+\eta)\rho$ and 
$\tbe(\vecv)=(1+\eta)^{-1}(\ttbe(\vecv)+3\vecv)$.
Here note that, by Lemma \ref{COMPARABILITYLEM} and \eqref{RHOVOLBDRYASS},
\begin{align}
\lambda_{M,\vecs_k}\Bigl(\partial_{ C\rho}
\bigl(\cup_{i\in M} D_i\bigr)\Bigr)\leq
K_{\eta}\vol_{\S_1^{d-1}}\Bigl(\cup_{i\in M} D_i\Bigr)^{-1}
\vol_{\S_1^{d-1}}\Bigl(\partial_{ C\rho}
\bigl(\cup_{i\in M} D_i\bigr)\Bigr) &
\\ \notag
\leq K_{\eta}\vol_{\S_1^{d-1}}\Bigl(\cup_{i\in M} D_i\Bigr)^{-1}
\sum_{i\in M} \vol_{\S_1^{d-1}}\Bigl(
\partial_{ C\rho}\bigl(D_i\bigr)\Bigr)
\leq \ve & .
\end{align}
The second set in the right hand side of \eqref{UDIFFCAPT}
has $\lambda_{M,\vecs_k}$-measure zero.
Next recall that \eqref{DHTTBEBOUND} led to $\ttbe\in F_3^{(\bn)}$;
by a similar argument we also verify $\tbe\in F_3^{(\bn)}$.
Hence from \eqref{LAMBDAEDGESMALL} applied with $\vecq'=\bn$, 
$\lambda=\lambda_{M,\vecs_k}$ and 
$\ttbe$, $\rho$, resp.\ $\tbe$, $\trho$, %
we see that each of the last two sets in \eqref{UDIFFCAPT}
have $\lambda_{M,\vecs_k}$-measure $<2\ve$. Hence
\begin{align}
\lambda_{M,\vecs_k}\Bigl(\Phi\bigl(\scrU'\setminus\scrU^{(k+1)}\bigr)\Bigr)
<5\ve.
\end{align}
Set $\mu=\vol_{\S_1^{d-1}}(\scrU')^{-1}\cdot (\vol_{\S_1^{d-1}})_{|\scrU'}$.
By Proposition \ref{REFLPROP},
and our assumption $\rho<\rho_0(n,\eta,C,\eta)$, $\Phi_{|\scrU'}$
is a diffeomorphism from $\scrU'$ onto $\cup_{M} D_i$,
which transforms the measure $\mu$ into
$c\tg\cdot\lambda_{M,\vecs_{k}}$ where $c>0$ is a constant and
$\tg$ is a continuous function from $\cup_{M} D_i$ to $[1-\eta,1+\eta]$.
Using $c\int_{\cup_{M} D_i} \tg(\vecv)\, d\lambda_{M,\vecs_{k}}(\vecv)=1$
we find $(1+\eta)^{-1}\leq c \leq (1-\eta)^{-1}$, and thus
$c\tg(\vecv)\in [1-3\eta,1+3\eta]$ for all $\vecv\in\cup_{M} D_i$.
Hence 
\begin{align} \label{UPRIMDIFFINEQ}
\mu\bigl(\scrU'\setminus\scrU^{(k+1)}\bigr)
<(1+3\eta)5\ve<10\ve.
\end{align}
In other words $\vol_{\S_1^{d-1}}\bigl(\scrU'\setminus\scrU^{(k+1)}\bigr)
<10\ve\vol_{\S_1^{d-1}}\bigl(\scrU'\bigr)$.
Adding this over all equivalence classes $\scrU'\subset\scrU^{(k)}$
we obtain $\vol_{\S_1^{d-1}}\bigl(\scrU^{(k)}\setminus\scrU^{(k+1)}\bigr)
<10\ve \vol_{\S_1^{d-1}}\bigl(\scrU^{(k)}\bigr)\ll\ve$.
Hence since $\lambda_0=g\cdot \vol_{\S_1^{d-1}}$ with $g$ bounded
(cf.\ \eqref{LAMBDA0GDEF}), we have
\begin{align} \label{LAMBDAUKUKPSMALL}
\lambda_0\bigl(\scrU^{(k)}\setminus\scrU^{(k+1)}\bigr)\ll\ve.
\end{align}
Adding this over $k=1,\ldots,n-1$ and combining
with \eqref{LAMBDAUU1SMALL} we finally obtain
\begin{align} \label{UKDIFFSMALL}
\lambda_0\bigl(\S_1^{d-1}\setminus\scrU^{(n)}\bigr)\ll\ve.
\end{align}

\subsubsection{Conclusion: Proof of (\ref{pathwayproofclaim})}
We intend to show that for each $m\in\{0,\ldots,n-2\}$ we have
\begin{align} \label{pathwayfmkey}
& \biggl | \int_{\S_1^{d-1}} f_m\big(\vecv_0, \rho^{d-1} \tau_1(\vecq_{\rho,{\vecbeta_0}}(\vecv_0),\vecv_0;\rho), \ldots, \vecv_{n-m}(\vecq_{\rho,{\vecbeta_0}}(\vecv_0),\vecv_0;\rho) \big) \, d\lambda_0(\vecv_0) 
\\ \notag
& -
\int_{\S_1^{d-1}} f_{m+1}\big(\vecv_0, \rho^{d-1} \tau_1(\vecq_{\rho,{\vecbeta_0}}(\vecv_0),\vecv_0;\rho), \ldots, \vecv_{n-m-1}(\vecq_{\rho,{\vecbeta_0}}(\vecv_0),\vecv_0;\rho) \big) \, d\lambda_0(\vecv_0) \biggr | \ll\ve.
\end{align}
This will imply \eqref{pathwayproofclaim} (and
thus complete our proof of Theorem \ref{pathwayThm}), for note that
\eqref{exactpos2-1hit-unif-used} applied with
$\vecq'=\vecq$, $\lambda=\lambda_0$, $f=f_{n-1}$ and $\vecbeta=\vecbeta_0$
gives, in view of \eqref{POINTOFRECDEF}:
\begin{align} \label{pathwayfmkey2}
& \biggl | \int_{\S_1^{d-1}} f_{n-1}\big(\vecv_0, \rho^{d-1} \tau_1(\vecq_{\rho,\vecbeta_0}(\vecv_0),\vecv_0;\rho), \vecv_{1}(\vecq_{\rho,\vecbeta_0}(\vecv_0),\vecv_0;\rho) \big) \, d\lambda_0(\vecv_0)
-\int_X f_0\, d\nu \biggr | <\ve.
\end{align}
Combining \eqref{pathwayfmkey} for $m=0,1,\ldots,n-2$ and
\eqref{pathwayfmkey2} we indeed obtain \eqref{pathwayproofclaim}, as desired.

Now to prove \eqref{pathwayfmkey} we fix $m\in\{0,\ldots,n-2\}$
and set $k=n-m-1\in\{1,\ldots,n-1\}$.
In the following %
we will use the shorthand notation
$\vecv_\ell=\vecv_\ell(\vecq_{\rho,\vecbeta_0}(\vecv_0),\vecv_0;\rho)$,
$\tau_\ell=\tau_\ell(\vecq_{\rho,\vecbeta_0}(\vecv_0),\vecv_0;\rho)$,
$\vecm_\ell=\vecm_\ell(\vecq_{\rho,\vecbeta_0}(\vecv_0),\vecv_0;\rho)$,
and
$\vecs_\ell=\vecm_\ell-\vecm_{\ell-1}$, and $\uvecs_\ell=\|\vecs_\ell\|^{-1}\vecs_\ell$.
As before   %
we call two vectors $\vecv,\vecv'\in\scrU^{(k)}$ %
equivalent if and only if
$\vec{\vecm}_\vecv^{k}=\vec{\vecm}_{\vecv'}^{k}$.
Let $\scrU'\subset\scrU^{(k)}$ be any fixed equivalence class for this
relation; thus by construction $\vecm_0,\vecm_1,\ldots,\vecm_k$ are 
\textit{constant} as $\vecv_0$ varies through $\scrU'$.
We write $M\subset\{1,\ldots,N\}$ and
$\Phi$, $\overline{\vecbeta}^+$, $\ttbe$ for the index set and functions
corresponding to our fixed class $\scrU'$, 
as in the discussion just above \eqref{UDIFFCAPT}.
For each $\vecv_0\in\scrU'$ we have
$\bigl|\tau_\ell-\|\vecs_\ell\|\bigr |\leq 2\rho$
and $\varphi(\vecv_{\ell-1},\vecs_{\ell})<\frac 12 C\rho$ 
for all $\ell=2,\ldots,k$ (cf.\ \eqref{ARCSININEQ}), and similarly using
$\sup \|\vecbeta_0\|\leq C$ we get
$\bigl|\tau_1-\|\vecs_1\|\bigr |\leq (C+1)\rho$
and $\varphi(\vecv_{0},\vecs_{1})<\frac 12 C(C+1)\rho$.
Hence by \eqref{FCONTRHOASS} we have the following approximation result for the
$\scrU'$-contribution to the left integral in \eqref{pathwayfmkey}:
\begin{align} \notag
\biggl | \int_{\scrU'} f_m\big(\vecv_0,\rho^{d-1}\tau_1,\vecv_1,\ldots,\vecv_k,
\rho^{d-1}\tau_{k+1},\vecv_{k+1}\big) \, d\lambda_0(\vecv_0) \hspace{150pt} &
\\ \label{pathwayintegral2}
-\int_{\scrU'} f_m
\Big(\uvecs_1,\rho^{d-1}\|\vecs_{1}\|,\ldots,\uvecs_k,
\rho^{d-1}\|\vecs_{k}\|,
\vecv_{k},\rho^{d-1}\tau_{k+1},\vecv_{k+1}\Big)\, d\lambda_0(\vecv_0)
\biggr | <\ve\lambda_0(\scrU') &. 
\end{align}
Recall $\lambda_0=g\cdot\vol_{\S_1^{d-1}}$, and 
set $\mu=\vol_{\S_1^{d-1}}(\scrU')^{-1}\cdot (\vol_{\S_1^{d-1}})_{|\scrU'}$.
From above we have $\varphi(\vecv_0,\vecs_1)<C^2\rho$
for all $\vecv_0\in\scrU'$; by our assumption \eqref{GRHOCONTASS} 
this implies
\begin{align}
\Bigl | g(\vecv_0)-\lambda_0(\scrU')\vol_{\S_1^{d-1}}(\scrU')^{-1}\Bigr |
=\Bigl | g(\vecv_0)-\int_{\scrU'} g(\vecv)\, d\mu(\vecv)\Bigr |<\ve,
\qquad \forall \vecv_0\in\scrU'.
\end{align}
Hence replacing ``$d\lambda_0(\vecv_0)$'' with
``$\lambda_0(\scrU')\, d\mu(\vecv_0)$'' in the last integral in
\eqref{pathwayintegral2}
causes an error $\leq \ve(\sup |f_m|)\vol_{\S_1^{d-1}}(\scrU')$.
Furthermore, by Proposition \ref{REFLPROP}, $\vecv_0\mapsto
\vecv_k=\vecv_{k}(\vecq_{\rho,\vecbeta_0}(\vecv_0),\vecv_0;\rho)=\Phi(\vecv_0)$
is a diffeomorphism from $\scrU'$ onto $\cup_{M} D_i$ which,
as we saw in the discussion preceding \eqref{UPRIMDIFFINEQ},
transforms the measure $\mu$ into $c\tg\cdot\lambda_{M,\vecs_{k}}$ 
where $c\tg(\vecv)\in [1-3\eta,1+3\eta]$ for all $\vecv\in\cup_{M} D_i$.
We also have
$\tau_{k+1}=\tau_1(\vecm_k+\rho\ttbe(\vecv_k),\vecv_k;\rho)
=\tau_1(\rho\ttbe(\vecv_k),\vecv_k;\rho)$
for all $\vecv_0\in\scrU'$, and similarly $\vecv_{k+1}
=\vecv_1(\rho\ttbe(\vecv_k),\vecv_k;\rho)$.
Hence, using also $\eta<\ve$,
\begin{align} \notag
\biggl | \int_{\scrU'} f_m
\Big(\uvecs_1,\rho^{d-1}\|\vecs_{1}\|,\ldots,\uvecs_k,
\rho^{d-1}\|\vecs_{k}\|,
\vecv_{k},\rho^{d-1}\tau_{k+1},\vecv_{k+1}\Big)\, d\lambda_0(\vecv_0) 
\hspace{100pt} &
\\ \notag
- \lambda_0(\scrU') \int_{\cup_{M} D_i} 
f_m\Big(\uvecs_1,\ldots,\rho^{d-1}\|\vecs_{k}\|,\vecv,
\rho^{d-1}\tau_1(\rho\ttbe(\vecv),\vecv;\rho),
\vecv_1(\rho\ttbe(\vecv),\vecv;\rho)\Bigr)\,
d\lambda_{M,\vecs_{k}}(\vecv) \biggr |  &
\\ \label{pathwayintegral4}
\leq \ve(\sup |f_m|)\vol_{\S_1^{d-1}}(\scrU')+ 3\eta(\sup |f_m|)
\lambda_0(\scrU')
\ll \ve\bigl(\vol_{\S_1^{d-1}}(\scrU')+\lambda_0(\scrU')\bigr). &
\end{align}
Now $\lambda_{M,\vecs_{k}}\in F_1$, 
$f_m(\uvecs_1,\ldots,\rho^{d-1}\|\vecs_{k}\|,
\cdot,\cdot,\cdot)\in F_2$ 
and $\ttbe\in F_3^{(\bn)}$, so that by \eqref{exactpos2-1hit-unif-used},
\begin{align} \notag
\biggl | \lambda_0(\scrU') \int_{\cup_{M} D_i} 
f_m\Big(\uvecs_1, %
\ldots, %
\rho^{d-1}\|\vecs_{k}\|,\vecv, 
\rho^{d-1}\tau_1(\rho\ttbe(\vecv),\vecv;\rho),
\vecv_1(\rho\ttbe(\vecv),\vecv;\rho)\Bigr)\,
d\lambda_{M,\vecs_{k}}(\vecv)  \hspace{20pt} &
\\ \notag
-\lambda_0(\scrU') \int_{\cup_{M} D_i} \int_{\R_{>0}} \int_{\S_1^{d-1}}
f_m\Big(\uvecs_1, \rho^{d-1}\|\vecs_1\|,\uvecs_2,\ldots,
\uvecs_{k},\rho^{d-1}\|\vecs_{k}\|,\vecv,\xi,\vecv'\Big) \hspace{20pt} &
\\ \label{pathwayintegral5}
\times p_{\bn,\ttbe}(\vecv,\xi,\vecv')
\,d\!\vol_{\S_1^{d-1}}(\vecv') \, d\xi\, d\lambda_{M,\vecs_{k}}(\vecv) 
\biggr | <\ve\lambda_0(\scrU') & .
\end{align}
(We have $\lambda_{M,\vecs_{k}}(A)=0$ for any 
$A\subset\S_1^{d-1}$ disjoint from $\cup_{M} D_i$, so that we can indeed
take the domain of integration for $\vecv$ in the 
triple integral to be $\cup_{M} D_i$ instead of $\S_1^{d-1}$.)
Next we wish to replace $\ttbe$ with 
$\vecbeta^+_{\uvecs_{k}}$ in the triple integral in 
\eqref{pathwayintegral5}, so that it can be rewritten in
terms of $f_{m+1}$ using \eqref{FMRECDEF}.
Using \eqref{exactpos2-1hit-tpdef} and \eqref{p0betdefrep},
we see that the error caused by this replacement is bounded above by
$\lambda_0(\scrU')\sup |f_m|$ times the integral in 
\eqref{VEPRIMECHOICEIMPL},
with $\lambda=\lambda_{M,\vecs_{k}}$,
$\vecz_1(\vecv)=(\ttbe(\vecv)K(\vecv))_\perp$ and
$\vecz_2(\vecv)=(\vecbeta^+_{\uvecs_k}(\vecv)K(\vecv))_\perp$.
Note that for all $\vecv\in\cup_{M} D_i\subset\scrV_{\vecs_{k}}^{10\eta}$
we have, by \eqref{TTBINF3} and Proposition~\ref{REFLPROP} 
and our assumption $\rho<\rho_0(n,\eta,C,\eta)$:
\begin{align}
\|\vecz_1(\vecv)-\vecz_2(\vecv)\|\leq
\|\ttbe(\vecv)-\vecbeta^+_{\uvecs_k}(\vecv)\|
=\|\overline{\vecbeta}^+(\vecv)-\vecbeta^+_{\uvecs_k}(\vecv)\|<\eta.
\end{align}
Hence the inequality in \eqref{VEPRIMECHOICEIMPL} is valid
for our choices of $\lambda,\vecz_1,\vecz_2$,
and we obtain:   %
\begin{align} \notag
\biggl | \lambda_0(\scrU') \int_{\cup_{M} D_i} \int_{\R_{>0}} \int_{\S_1^{d-1}}
f_m\Big(\uvecs_1, \rho^{d-1}\|\vecs_1\|,\uvecs_2,\ldots,
\uvecs_{k},\rho^{d-1}\|\vecs_{k}\|,\vecv,\xi,\vecv'\Big) 
\hspace{70pt} &
\\ \label{pathwayintegral6}
\times p_{\bn,\ttbe}(\vecv,\xi,\vecv')
\,d\!\vol_{\S_1^{d-1}}(\vecv') \, d\xi\, d\lambda_{M,\vecs_{k}}(\vecv) 
\hspace{30pt} &
\\ \notag
-\lambda_0(\scrU') \int_{\cup_{M} D_i}
f_{m+1}\Big(\uvecs_1, \rho^{d-1}\|\vecs_1\|,\uvecs_2,\ldots,
\uvecs_{k},\rho^{d-1}\|\vecs_{k}\|,\vecv\Big)
\, d\lambda_{M,\vecs_{k}}(\vecv) \biggr | \ll \ve\lambda_0(\scrU') & .
\end{align}
Next, by imitating the argument which led to \eqref{pathwayintegral4}, 
and then the argument which led to \eqref{pathwayintegral2},
we obtain
\begin{align} \notag
\biggl | \lambda_0(\scrU') \int_{\cup_{M} D_i}
f_{m+1}\Big(\uvecs_1, \rho^{d-1}\|\vecs_1\|,\uvecs_2,\ldots,
\uvecs_{k},\rho^{d-1}\|\vecs_{k}\|,\vecv\Big)
\, d\lambda_{M,\vecs_{k}}(\vecv) \hspace{35pt} &
\\ \notag
- \int_{\scrU'}
f_{m+1}\Big(\vecv_0, \rho^{d-1}\tau_1(\vecq_{\rho,\vecbeta_0}(\vecv_0),\vecv_0;\rho),\ldots,\vecv_{k}(\vecq_{\rho,\vecbeta_0}(\vecv_0),\vecv_0;\rho)\Big)
\, d\lambda_0(\vecv_0) \biggr |
\\ \label{pathwayintegral7}
\ll \ve\bigl(\vol_{\S_1^{d-1}}(\scrU')+\lambda_0(\scrU')\bigr). &
\end{align}
Combining \eqref{pathwayintegral2}, \eqref{pathwayintegral4}, 
\eqref{pathwayintegral5}, \eqref{pathwayintegral6}
and \eqref{pathwayintegral7}, and adding over all the 
equivalence classes $\scrU'\subset\scrU^{(k)}$, using %
$\vol_{\S_1^{d-1}}(\scrU^{(k)})\leq \vol(\S_1^{d-1})\ll 1$ and
$\lambda_0(\scrU^{(k)})\leq 1$, we get
\begin{align} \notag
\biggl | \int_{\scrU^{(k)}} f_m\big(\vecv_0, \rho^{d-1} \tau_1(\vecq_{\rho,\vecbeta_0}(\vecv_0),\vecv_0;\rho), \ldots, \vecv_{k+1}(\vecq_{\rho,\vecbeta_0}(\vecv_0),\vecv_0;\rho) \big) \, d\lambda_0(\vecv_0) \hspace{50pt} &
\\ \label{pathwayintegral9}
-\int_{\scrU^{(k)}} f_{m+1}\big(\vecv_0, \rho^{d-1} \tau_1(\vecq_{\rho,\vecbeta_0}(\vecv_0),\vecv_0;\rho), \ldots, \vecv_{k}(\vecq_{\rho,\vecbeta_0}(\vecv_0),\vecv_0;\rho) \big) d\lambda_0(\vecv_0) \biggr | 
\ll\ve. &
\end{align}
This implies \eqref{pathwayfmkey}, since 
$\lambda_0(\S_1^{d-1}\setminus\scrU^{(k)})\ll\ve$
(from \eqref{UKDIFFSMALL}).

Since \eqref{pathwayproofclaim} follows from \eqref{pathwayfmkey} and
\eqref{pathwayfmkey2}, the proof of Theorem \ref{pathwayThm} is now complete.
\hfill$\square$

\subsection{Proof of Theorem \ref{secThmMicro}}\label{ThmMicroproofsec}
The generalization of Theorem \ref{secThmMicro} to the case of a general
scattering map and a more general initial condition
$(\vecq_0+\rho\vecbeta(\vecv_0),\vecv_0)$ is as follows.
We set
\begin{align} \label{SCRBNDEF}
& \scrB_n:=\bigl\{(\vecS_1,\ldots,\vecS_n)\in(\RR^{d}\setminus\{\vecnull\})^n\col 
\varphi(\vecS_j,\vecS_{j+1})>B_\Theta \: (j=1,\ldots,n-1)\bigr\}.
\end{align}

\begin{thm}\label{secThmMicroGS}
Fix a lattice $\scrL=\ZZ^d M_0$ and a point $\vecq_0\in\RR^d$,
set $\vecalf=-\vecq_0 M_0^{-1}$,
and let $\vecbeta:\S_1^{d-1}\to\R^d$
be a $\C^1$ function.
If $\vecq_0\in\scrL$ %
we assume that
$(\vecbeta(\vecv)+\R_{>0}\vecv)\cap\scrB_1^d=\emptyset$ for all 
$\vecv\in\S_1^{d-1}$.
Then for each $n\in\ZZ_{>0}$ there exists a (Borel measurable) function 
$P_{\vecalf,\vecbeta}^{(n)}:\scrB_n\to\R_{\geq 0}$ such that, 
for any Borel probability measure $\lambda$ on $\S_1^{d-1}$ which is 
absolutely continuous with respect to $\vol_{\S_1^{d-1}}$,
and for any set $\scrA\subset\RR^{nd}$ with boundary of Lebesgue measure zero,
\begin{align} \notag
\lim_{\rho\to 0}\lambda\big( \big\{ \vecv_0\in \S_1^{d-1} : 
(\vecs_1(\vecq_0+\rho\vecbeta(\vecv_0),\vecv_0;\rho),\ldots,
\vecs_n(\vecq_0+\rho\vecbeta(\vecv_0),\vecv_0;\rho)) \in 
\rho^{-(d-1)}\scrA \big\} \big) \hspace{15pt} &
\\ \label{secThm-eqGS}
= \int_{\scrB_n\cap\scrA} P_{\vecalf,\vecbeta}^{(n)}(\vecS_1,\ldots,\vecS_n)\,
\lambda'(\uvecS_1)\,d\!\vol_{\RR^d}(\vecS_1)\cdots d\!\vol_{\RR^d}(\vecS_n) & ,
\end{align}
where $\lambda'\in \L^1(\S_1^{d-1})$ is the Radon-Nikodym derivative of
$\lambda$ with respect to $\vol_{\S_1^{d-1}}$.
Furthermore, there is a function $\Psi:\scrB_3\to\R_{\geq 0}$ 
which only depends on the scattering map, such that
\begin{equation} \label{jointlimdensGS}
	P_{\vecalf,\vecbeta}^{(n)}(\vecS_1,\ldots,\vecS_n) = P_{\vecalf,\vecbeta}^{(2)}(\vecS_1,\vecS_2)
	\prod_{j=3}^n \Psi(\vecS_{j-2},\vecS_{j-1},\vecS_j) 
\end{equation}
for all $n\geq 3$ and all $(\vecS_1,\ldots,\vecS_n)\in\scrB_n$.
\end{thm}

Explicit formulas for $P_{\vecalf,\vecbeta}^{(2)}$ and
$\Psi$ are given in the proof below.
As in \eqref{MUPROBMEAS} we define the probability measure corresponding 
to \eqref{secThm-eqGS} by 
\begin{equation} \label{MUPROBMEASGS}
\mu_{\vecalf,\vecbeta,\lambda}^{(n)}(\scrA)
:= \int_{\scrB_n\cap\scrA} P_{\vecalf,\vecbeta}^{(n)}(\vecS_1,\ldots,\vecS_n)
\lambda'(\uvecS_1)\,d\!\vol_{\RR^d}(\vecS_1)\cdots d\!\vol_{\RR^d}(\vecS_n).
\end{equation}

\begin{proof}[Proof of Theorem \ref{secThmMicroGS}]
By a standard approximation argument, using the 
absolute continuity of the limit measure and the assumption that
$\scrA$ has boundary of Lebesgue measure zero, one finds that
it suffices to prove the
corresponding statement for bounded continuous functions, i.e.\ to prove
that for each bounded continuous function $g:\R^{nd}\to\R_{\geq 0}$ we have
\begin{align} \label{secThm-eq-cont}
\lim_{\rho\to 0} \int_{\S_1^{d-1}} g\bigl(\rho^{d-1}
\vecs_1(\vecq_0+\rho\vecbeta(\vecv_0),\vecv_0;\rho),\ldots,
\rho^{d-1}\vecs_n(\vecq_0+\rho\vecbeta(\vecv_0),\vecv_0;\rho)\bigr) \, 
d\lambda(\vecv_0) \hspace{15pt} &
\\ \notag
= \int_{\scrB_n} g(\vecS_1,\ldots,\vecS_n)
P_{\vecalf,\vecbeta}^{(n)}(\vecS_1,\ldots,\vecS_n)\,  
\lambda'(\uvecS_1)\,d\!\vol_{\RR^d}(\vecS_1)\cdots d\!\vol_{\RR^d}(\vecS_n) & .
\end{align}
This, however, is a direct consequence of Theorem \ref{pathwayThm},
applied with $f_0(\vecv_0,\xi_1,\vecv_1,\ldots,\xi_n,\vecv_n)
:=g(\xi_1\vecv_0,\xi_2\vecv_1,\ldots,\xi_n\vecv_{n-1})$.
Indeed, substituting $\xi_j\vecv_{j-1}=\vecS_j$ 
in the right hand side of
\eqref{pathwayEq} in Theorem \ref{pathwayThm} gives exactly the right
hand side of \eqref{secThm-eq-cont}, with
\begin{align} \notag
P_{\vecalf,\vecbeta}^{(n)}(\vecS_1,\ldots,\vecS_n)
=\Bigl(\prod_{j=1}^n\|\vecS_j\|^{1-d}\Bigr)
p_{\vecalf,\vecbeta}(\uvecS_1,\|\vecS_1\|,\uvecS_2)\,
\prod_{j=1}^{n-2} 
p_{\bn,\vecbeta_{\uvecS_j}^+}(\uvecS_{j+1},\|\vecS_{j+1}\|,\uvecS_{j+2})
&
\\ \label{PNLQ0EXPL}
\times \int_{\S_1^{d-1}} 
p_{\bn,\vecbeta_{\uvecS_{n-1}}^+}(\uvecS_n,\|\vecS_n\|,\uvecS)
\, d\!\vol_{\S_1^{d-1}}(\uvecS) & 
\end{align}
if $n\geq 2$, and
\begin{align} \label{PNLQ0EXPLNONE}
P_{\vecalf,\vecbeta}^{(1)}(\vecS_1)
=\|\vecS_1\|^{1-d} \int_{\S_1^{d-1}} 
p_{\vecalf,\vecbeta}(\uvecS_1,\|\vecS_1\|,\uvecS)
\, d\!\vol_{\S_1^{d-1}}(\uvecS) & 
\end{align}
if $n=1$. 

Next we claim that \eqref{jointlimdensGS} holds (for $d\geq 3$) if we set 
\begin{align}
I(\vecS_1,\vecS_2)=\int_{\S_1^{d-1}} 
p_{\bn,\vecbeta_{\uvecS_{1}}^+}(\uvecS_{2},\|\vecS_{2}\|,\uvecS)
\, d\!\vol_{\S_1^{d-1}}(\uvecS)
\end{align}
and then define 
\begin{align} \label{PSIDEF}
\Psi(\vecS_1,\vecS_2,\vecS_3):=\|\vecS_3\|^{1-d}
p_{\bn,\vecbeta_{\uvecS_{1}}^+}(\uvecS_{2},\|\vecS_{2}\|,\uvecS_{3})
\frac{I(\vecS_2,\vecS_3)}{I(\vecS_1,\vecS_2)}
\end{align}
if $I(\vecS_1,\vecS_2)\neq 0$,
while setting $\Psi(\vecS_1,\vecS_2,\vecS_3)=0$ %
whenever $I(\vecS_1,\vecS_2)=0$.
To prove this claim it suffices to check
\begin{align} \label{PRODUCTIND}
P_{\vecalf,\vecbeta}^{(n)}(\vecS_1,\ldots,\vecS_n)
=P_{\vecalf,\vecbeta}^{(n-1)}(\vecS_1,\ldots,\vecS_{n-1})
\Psi(\vecS_{n-2},\vecS_{n-1},\vecS_{n})
\end{align} 
for all $n\geq 3$ and all $(\vecS_1,\ldots,\vecS_n)\in\scrB_n$.

If $I(\vecS_{n-2},\vecS_{n-1})\neq 0$ then \eqref{PRODUCTIND} is clear from 
\eqref{PNLQ0EXPL} and \eqref{PSIDEF}.
Now assume $I(\vecS_{n-2},\vecS_{n-1})=0$.
The function $f:\uvecS\mapsto
p_{\bn,\vecbeta_{\uvecS_{n-2}}^+}(\uvecS_{n-1},\|\vecS_{n-1}\|,\uvecS)$
is non-negative on all $\S_1^{d-1}$ and vanishes outside the open subset 
$\scrV_{\uvecS_{n-1}}$.
If $d\geq 3$ then $f$ is continuous on $\scrV_{\uvecS_{n-1}}$
(see Remark \ref{PALFBETPLUSCONTREM}),
so that $I(\vecS_{n-2},\vecS_{n-1})=0$ implies that
$f(\uvecS)=0$ holds for \textit{all} $\uvecS\in\scrV_{\uvecS_{n-1}}$;
in particular we have
$P_{\vecalf,\vecbeta}^{(n)}(\vecS_1,\ldots,\vecS_{n})=0$ as well as
$P_{\vecalf,\vecbeta}^{(n-1)}(\vecS_1,\ldots,\vecS_{n-1})=0$,
and hence \eqref{PRODUCTIND} holds.

To treat the remaining case, $d=2$, it is simplest to first
\textit{modify} the function $P_{\vecalf,\vecbeta}^{(n)}$ given by
\eqref{PNLQ0EXPL} by setting it to be zero at any point
$(\vecS_1,\ldots,\vecS_n)\in\scrB_n$ such that 
$\uvecS_{j+2}=\Theta_1(\uvecS_{j+1},-\vecbeta_{\uvecS_j}^+(\uvecS_{j+1})
R_{_{\vecS_{j+1}}})$ for some $1\leq j \leq n-2$,
where $R_{_{\vecS}}\in \O(2)$ denotes reflection in the line $\R\vecS$.
(This alteration only concerns a subset of $\scrB_n$ of Lebesgue measure zero,
and hence it does not affect the validity of \eqref{secThm-eqGS}.)
Now the proof can be completed as before, using the fact that when $d=2$,
the function $f$ considered in the last paragraph is continuous on $\scrV_{\uvecS_{n-1}}$
except possibly at the single point 
$\uvecS=\Theta_1(\uvecS_{n-1},
-\vecbeta_{\uvecS_{n-2}}^+(\uvecS_{n-1})R_{_{\vecS_{n-1}}})$,
again by Remark \ref{PALFBETPLUSCONTREM}.
\end{proof}

\begin{remark} \label{MUNPONEREL}
Note that $\mu_{\vecalf,\vecbeta,\lambda}^{(n+1)}(\scrA\times\RR^d)
=\mu_{\vecalf,\vecbeta,\lambda}^{(n)}(\scrA)$ for every 
measurable subset $\scrA\subset\R^{nd}$. This follows from
\eqref{MUPROBMEASGS}, \eqref{PNLQ0EXPL} and \eqref{PNLQ0EXPLNONE}, 
using \eqref{p-norm}.
\end{remark}

\begin{remark} \label{indi-remarkGS}
Recall that for $\vecalf\in\R^d\setminus\Q^d$ the function
$p_{\vecalf,\vecbeta}(\vecv_0,\xi_1,\vecv_1)$ is independent of
both $\vecalf$ and $\vecbeta$, cf.\ Remark \ref{PALFBETCONTREM}.
Hence by \eqref{PNLQ0EXPL}, \eqref{PNLQ0EXPLNONE}, 
the same is true for the function $P_{\vecalf,\vecbeta}^{(n)}=:P^{(n)}$. 
\end{remark}

\begin{remark} \label{secThmMicrocontrem}
If $d\geq 3$ then $P^{(n)}$ is continuous on all of $\scrB_n$,
and if also $\sup \|\vecbeta\|\leq 1$ then
$P_{\vecalf,\vecbeta}^{(n)}$ is continuous on $\scrB_n$ 
for every $\vecalf\in\R^d$.
Next suppose $d=2$, and 
let $\scrB_n'$ be the following dense open subset of $\scrB_n$:
\begin{align} \notag
\scrB_n':=\bigl\{(\vecS_1,\ldots,\vecS_n)\in\scrB_n \col
\vecbeta_{\vece_1}^-\bigl(\uvecS_2K(\uvecS_1)\bigr)_\perp\neq
(\vecbeta(\uvecS_1)K(\uvecS_1))_\perp, \hspace{95pt}&
\\ \label{SCRBNPRIMDEF}
\uvecS_{j+2}\neq\Theta_1(\uvecS_{j+1},-\vecbeta_{\uvecS_j}^+(\uvecS_{j+1})
R_{_{\vecS_{j+1}}}) 
\:\: (j=1,\ldots,n-2) \bigr\} & .
\end{align}
Then $P^{(n)}$ is continuous on all of $\scrB_n'$,
and if $\sup \|\vecbeta\|\leq 1$ then
$P_{\vecalf,\vecbeta}^{(n)}$ is continuous on $\scrB_n'$ 
for every $\vecalf\in\R^d$.
These statements follow from \eqref{PNLQ0EXPL}, \eqref{PNLQ0EXPLNONE},
Remark \ref{PALFBETCONTREM} and Remark \ref{PALFBETPLUSCONTREM}.
\end{remark}

\subsection{Proof of Theorem \ref{secThmMacro}}
We now prove the ``macroscopic'' version of Theorem \ref{secThmMicro},
i.e.\ Theorem \ref{secThmMacro},
in the case of a general scattering map.
The precise statement is as follows. Let $\scrB_n$ be as in 
\eqref{SCRBNDEF}.
\begin{thm}\label{secThmMacroGS}
Let $\Lambda$ be a Borel probability measure on $\T^1(\RR^d)$ which is absolutely continuous with respect to Lebesgue measure. Then, for each $n\in\ZZ_{>0}$, and for any set $\scrA\subset \RR^d\times\RR^{nd}$ with boundary of Lebesgue measure zero,
\begin{multline} \label{secThm-eq-macroGS}
	\lim_{\rho\to 0}\Lambda\big( \big\{ (\vecQ_0,\vecV_0)\in \T^1(\rho^{d-1}\scrK_\rho) : (\vecQ_0,\vecS_1(\vecQ_0,\vecV_0;\rho),\ldots,\vecS_n(\vecQ_0,\vecV_0;\rho)) \in \scrA \big\} \big) \\
= \int_{\scrA\cap(\R^d\times\scrB_n)} 
P^{(n)}(\vecS_1,\ldots,\vecS_n)\, \Lambda'\big(\vecQ_0,\uvecS_1\big)\, d\!\vol_{\RR^d}(\vecQ_0)\, d\!\vol_{\RR^d}(\vecS_1)\cdots d\!\vol_{\RR^d}(\vecS_n) ,
\end{multline}
with $P^{(n)}$ as in Remark \ref{indi-remarkGS},
and where $\Lambda'$ is the Radon-Nikodym derivative of
$\Lambda$ with respect to $\vol_{\R^d}\times\vol_{\S_1^{d-1}}$.
\end{thm}

The probability measure corresponding to the above limiting distribution is defined by
\begin{equation} \label{muLambdandefGS}
	\mu_{\Lambda}^{(n)}(\scrA)
:= \int_{\scrA\cap(\R^d\times\scrB_n)} P^{(n)}(\vecS_1,\ldots,\vecS_n)\, \Lambda'\big(\vecQ_0,\uvecS_1\big)\, d\!\vol_{\RR^d}(\vecQ_0)\, d\!\vol_{\RR^d}(\vecS_1)\cdots d\!\vol_{\RR^d}(\vecS_n) .
\end{equation}

\begin{proof}[Proof of Theorem \ref{secThmMacroGS}.]
We will use the same basic technique as in \cite[Sec.\ 9.2]{partI}.
Note that by a standard approximation argument,
using the fact that $\mu_\Lambda^{(n)}$ is %
absolutely continuous with respect to Lebesgue measure,
it suffices to prove the corresponding statement for continuous functions
of compact support, i.e.\ 
to prove that for any continuous function $g:\R^d\times(\R^d)^n\to\R$
of compact support we have
\begin{align} \notag
\int_{(\vecQ_0,\vecV_0)\in\T^1(\rho^{d-1}\scrK_\rho)}
g(\vecQ_0,\vecS_1(\vecQ_0,\vecV_0;\rho),\ldots,\vecS_n(\vecQ_0,\vecV_0;\rho))
\hspace{115pt} &
\\ \label{THMMACROTESTFUNC2}
\times
\Lambda'\big(\vecQ_0,\vecV_0\big)\, d\!\vol_{\RR^d}(\vecQ_0)\,
d\!\vol_{\S_1^{d-1}}(\vecV_0)
\to\int_{\R^d\times\R^{nd}} g \, d\mu_\Lambda^{(n)}&
\end{align}
as $\rho\to 0$.
By a further approximation argument, we may also assume $\Lambda'$ to be 
continuous and of compact support
(keeping $\Lambda'\geq 0$ and $\int_{\T^1(\R^d)} \Lambda'
\,d\!\vol_{\R^d} d\!\vol_{\S_1^{d-1}}=1$).

Recalling $\vecS_n(\vecQ_0,\vecV_0;\rho)
=\rho^{d-1}\vecs_n(\rho^{-(d-1)}\vecQ_0,\vecV_0;\rho)$
where $\vecs_n$ is $\scrL$-periodic in its first variable,
the left hand side of \eqref{THMMACROTESTFUNC2} may be expressed as
\begin{align} \notag
\int_{F\cap\scrK_\rho} \int_{\S_1^{d-1}}
\Bigl\{\rho^{d(d-1)} \sum_{\vecq\in\vecq_0+\scrL}
g\bigl(\rho^{d-1}\vecq,\rho^{d-1}\vecs_1(\vecq_0,\vecv_0;\rho),\ldots,
\rho^{d-1}\vecs_n(\vecq_0,\vecv_0;\rho)\bigr) \hspace{50pt} &
\\ \label{THMMACROTESTFUNC2STEP1}
\times \Lambda'(\rho^{d-1}\vecq,\vecv_0)\Bigr\}\,d\!\vol_{\S_1^{d-1}}(\vecv_0)
\, d\!\vol_{\R^d}(\vecq_0) & ,
\end{align}
where $F\subset\R^d$ is a fundamental parallelogram for $\R^d/\scrL$.
Because of our assumptions on $g$ and $\Lambda'$ we have
\begin{align*}
\rho^{d(d-1)}
\sum_{\vecq\in\vecq_0+\scrL} g(\rho^{d-1}\vecq,\veca_1,\ldots,\veca_n)
\Lambda'(\rho^{d-1}\vecq,\vecv_0)
\to \int_{\R^d} g(\vecq,\veca_1,\ldots,\veca_n) \Lambda'(\vecq,\vecv_0)
\,d\!\vol_{\R^d}(\vecq) 
\end{align*}
as $\rho\to 0$, uniformly over all $\veca_1,\ldots,\veca_n\in\R^d$,
$\vecv_0\in\S_1^{d-1}$ and $\vecq_0$ in compact subsets of $\R^d$.
Hence, using also Fubini's theorem, \eqref{THMMACROTESTFUNC2STEP1} equals
\begin{align} \notag
\int_{F} \int_{\R^d} \int_{\S_1^{d-1}}
 g\bigl(\vecq,\rho^{d-1}\vecs_1(\vecq_0,\vecv_0;\rho),\ldots,
\rho^{d-1}\vecs_n(\vecq_0,\vecv_0;\rho)\bigr) \hspace{120pt} &
\\ \label{THMMACROTESTFUNC2STEP2}
\times\Lambda'(\vecq,\vecv_0)
\,d\!\vol_{\S_1^{d-1}}(\vecv_0)\,d\!\vol_{\R^d}(\vecq)
\,d\!\vol_{\R^d}(\vecq_0)+o(1), &
\end{align}
where the innermost integral should be interpreted as zero whenever
$\vecq_0\notin\scrK_\rho$.
Now by Theorem~\ref{secThmMicroGS} reformulated in the context of
continuous test functions, cf.\ \eqref{secThm-eq-cont},
for almost all $(\vecq_0,\vecq)\in F\times\R^d$
the  innermost integral in \eqref{THMMACROTESTFUNC2STEP2} tends to
\begin{align}
\int_{\scrB_n}g(\vecq,\vecS_1,\ldots,\vecS_n)
P^{(n)}(\vecS_1,\ldots,\vecS_n)\,  
\Lambda'(\vecq,\uvecS_1)\, d\!\vol_{\RR^d}(\vecS_1)\cdots d\!\vol_{\RR^d}(\vecS_n).
\end{align}
Hence \eqref{THMMACROTESTFUNC2} follows by applying 
Lebesgue's Bounded Convergence Theorem to \eqref{THMMACROTESTFUNC2STEP2}
(with $(\vecq_0,\vecq)\mapsto
(\sup |g|)\int_{\S_1^{d-1}} \Lambda'(\vecq,\vecv_0)\,d\vecv_0$ as a majorant
function), and using $\vol(F)=1$.
\end{proof}

\section{Convergence to the stochastic process $\Xi(t)$}\label{secConvergence}

In this section we will prove Theorem \ref{secThmMicro2},
generalized to the %
situation of a general scattering
map, and with $(\vecq(t),\vecv(t))$ being the orbit of the flow $\varphi_t$
for initial data of the form $(\vecq_0+\rho\vecbeta(\vecv_0),\vecv_0)$.

The probability 
$\PP_{\vecalf,\vecbeta,\lambda}^{(\vecn)}\big(\Xi(t_1)\in\scrD_1
,\ldots, \Xi(t_M)\in\scrD_M\big)$ is defined in a similar way 
as before:
We again write 
$T_0:=0$, $T_n:=\sum_{j=1}^n \|\vecS_j\|$,
and given any $\vecn=(n_1,\ldots,n_M)\in\Z_{\geq 0}^M$ we set
\begin{multline} \label{PNLQ0LDEFGS}
	\PP_{\vecalf,\vecbeta,\lambda}^{(\vecn)}\big(\Xi(t_1)\in\scrD_1,\ldots, \Xi(t_M)\in\scrD_M \text{ and } T_{n_1}\leq t_1< T_{n_1+1},\ldots, T_{n_M}\leq t_M< T_{n_M+1}  \big) \\
	:=
\mu_{\vecalf,\vecbeta,\lambda}^{(n+1)}\big(\big\{(\vecS_1,\ldots,\vecS_{n+1}): \Xi_{n_j}(t_j) \in\scrD_j,\; T_{n_j}\leq t_j< T_{n_j+1}\;(j=1,\ldots,M) \big\}\big) 
\end{multline}
with $n=\max(n_1,\ldots,n_M)$ and
$\Xi_n(t):= \big( \sum_{j=1}^n \vecS_j + (t-T_n) \uvecS_{n+1}, \uvecS_{n+1}\big) ,$
and with the measure $\mu_{\vecalf,\vecbeta,\lambda}^{(\vecn)}$ now being given by
Theorem \ref{secThmMicroGS} and \eqref{MUPROBMEASGS}.
(Thus $\mu_{\vecalf,\vecbeta,\lambda}^{(\vecn)}$ depends on the given scattering 
map and the given function $\vecbeta$.)
We then set
\begin{multline} \label{PPLQ0LDEFGS}
		\PP_{\vecalf,\vecbeta,\lambda}\big(\Xi(t_1)\in\scrD_1,\ldots, \Xi(t_M)\in\scrD_M\big) 
		\\
	  := \sum_{\vecn\in\ZZ_{\geq 0}^M} \PP_{\vecalf,\vecbeta,\lambda}^{(\vecn)}\big(\Xi(t_1)\in\scrD_1,\ldots, \Xi(t_M)\in\scrD_M \text{ and } T_{n_1}\leq t_1< T_{n_1+1},\ldots, T_{n_M}\leq t_M< T_{n_M+1}  \big).
\end{multline}
Recall that given any set $\scrD\subset\T^1(\R^d)$ we say that $t\geq 0$ is
\textit{$\scrD$-admissible} if 
$(t\uvecS_1,\uvecS_1)\notin\partial\scrD$ holds for 
($\vol_{\S_1^{d-1}}$-)almost all $\uvecS_1\in\S_1^{d-1}$,
and we write $\adm(\scrD)$ for the set of all 
$\scrD$-admissible numbers $t\geq 0$.

The generalization of Theorem \ref{secThmMicro2} now reads as follows.
\begin{thm} \label{secThmMicro2GS}
Fix a lattice $\scrL=\ZZ^d M_0$ and a point $\vecq_0\in\RR^d$,
let $\lambda$ be a Borel probability measure on $\S_1^{d-1}$ which is absolutely continuous with respect to Lebesgue measure, and let
$\vecbeta:\S_1^{d-1}\to\R^d$
be a $\C^1$ function. If $\vecq_0\in\scrL$ we assume that
$(\vecbeta(\vecv)+\R_{>0}\vecv)\cap\scrB_1^d=\emptyset$ for all 
$\vecv\in\S_1^{d-1}$.
Then, for any subsets $\scrD_1,\ldots,\scrD_M\subset \T^1(\RR^{d})$ 
with boundary of Lebesgue measure zero, and any numbers
$t_j\in\adm(\scrD_j)$ ($j=1,\ldots,M$),
\begin{multline} \label{secThm-eq2GS}
	\lim_{\rho\to 0}\lambda\big( \big\{ \vecv_0\in \S_1^{d-1} : (\rho^{d-1}\vecq(\rho^{-(d-1)}t_j),\vecv(\rho^{-(d-1)} t_j)) \in\scrD_j, \; j=1,\ldots,M \big\} \big) \\
= \PP_{\vecalf,\vecbeta,\lambda}\big(\Xi(t_1)\in\scrD_1,\ldots, \Xi(t_M)\in\scrD_M\big) .
\end{multline}
The convergence is uniform for $(t_1,\ldots,t_M)$ in compact subsets of 
$\adm(\scrD_1)\times\cdots\times\adm(\scrD_M)$. 
\end{thm}

\subsection{Four lemmas}\label{FOURLEMMASSEC}

As a preparation for the proof of Theorem \ref{secThmMicro2GS} we will
require the following four lemmas.

\begin{lem}\label{diffLem}
Given any $t>0$ and $\ve>0$ there exist $\rho_0>0$ and $N\in\Z_{>0}$ 
such that, for all $\rho<\rho_0$, 
\begin{equation}\label{lemN}
\lambda\big( \{ \vecv_0\in\S_1^{d-1} \col
\|\vecs_1\|+\|\vecs_2\|+\ldots+\|\vecs_N\|\leq \rho^{-(d-1)}t\} \big) < \ve,
\end{equation}
where $\vecs_k=\vecs_k(\vecq_0+\rho\vecbeta(\vecv_0),\vecv_0;\rho)$.
\end{lem}

\begin{proof}
For $t>0$ and $N\in\Z_{>0}$ we have by Theorem \ref{pathwayThm}
(coupled with a simple approximation argument)
\begin{align} \label{diffLemproof1}
\lim_{\rho\to 0}\lambda & \big( \{ \vecv_0\in\S_1^{d-1} \col
\|\vecs_1\|+\|\vecs_2\|+\ldots+\|\vecs_N\|\leq \rho^{-(d-1)}t\} \big)
\\ \notag
& =\nu\bigl(\bigl\{(\vecv_0,\xi_1,\vecv_1,\ldots,\xi_N,\vecv_N)\in
\S_1^{d-1}\times (\R_{>0}\times \S_1^{d-1})^N\col
{\textstyle\sum_{j=1}^N}\xi_j<t\bigr\}\bigr),
\end{align}
where $\nu$ is the measure defined in \eqref{pathwaymeasure}.
Transforming the integral in \eqref{pathwaymeasure}
via \eqref{p0betdefrep} and \eqref{exactpos2-1hit-tpdef},
and then using the fact that $\Phi_\vecalf(\xi,\vecw,\vecz)$
and $\Phi_\bn(\xi,\vecw,\vecz)$ are uniformly bounded by $1$,
we see that \eqref{diffLemproof1} is bounded from above by
\begin{align}
\vol(\scrB_1^{d-1})^N 
\int_{\substack{\xi_1+\ldots+\xi_N<t\\\xi_1,\ldots,\xi_N>0}} 
d\xi_1\cdots d\xi_N
=\vol(\scrB_1^{d-1})^N \frac{t^N}{N!}.
\end{align}
The lemma follows from the fact that this expression tends
to zero as $N\to\infty$, for any fixed $t>0$.
\end{proof}

Given any $\scrD\subset\T^1(\R^d)=\R^d\times\S_1^{d-1}$ and $\delta>0$, 
we write $\partial_\delta\scrD$ and $\scrN_\delta\scrD$ for the 
$\delta$-neighborhoods of $\partial \scrD$ and $\scrD$, respectively,
viz.:
\begin{align} \label{PARTDELTA}
& \partial_\delta\scrD=\bigl\{(\vecp,\vecv)\in\R^d\times\S_1^{d-1}\col
\exists (\vecp_1,\vecv_1)\in\partial\scrD:\: 
\|\vecp_1-\vecp\|+\varphi(\vecv_1,\vecv)
<\delta\bigr\}.
\\ \label{NEIGHDELTA}
& \scrN_\delta\scrD=\scrD\cup\partial_\delta\scrD
=\bigl\{(\vecp,\vecv)\in\R^d\times\S_1^{d-1}\col
\exists (\vecp_1,\vecv_1)\in\scrD:\: 
\|\vecp_1-\vecp\|+\varphi(\vecv_1,\vecv)
<\delta\bigr\}.
\end{align}

\begin{lem} \label{NZEROKEYLEMMA}
Let $\scrD\subset\T^1(\R^d)$ and $t\in\adm(\scrD)$. Then 
\begin{align} \label{NZEROKEYLEMMAEQ}
\lim_{\delta\to 0}\vol_{\S_1^{d-1}}\bigl(\bigl\{\uvecS_1\in\S_1^{d-1}\col
(t\uvecS_1,\uvecS_1)\in\partial_\delta\scrD\bigr\}\bigr)=0.
\end{align}
\end{lem}

\begin{proof}
By assumption the set
$M=\bigl\{\uvecS_1\in\S_1^{d-1}\col
(t\uvecS_1,\uvecS_1)\in\partial\scrD\bigr\}$ has volume zero;
hence for any given $\ve>0$ there is an
open set $U\subset\S_1^{d-1}$ with $M\subset U$ and
$\vol_{\S_1^{d-1}}(U)<\ve$. Then $C=\{(t\uvecS_1,\uvecS_1)\col
\uvecS_1\in\S_1^{d-1}\setminus U\}$ is a compact subset of $\T^1(\R^d)$,
disjoint from $\partial \scrD$.
Since $\partial\scrD$ is closed, there is some
$\delta_0>0$ such that $\|\vecp_1-\vecp_2\|+\varphi(\vecv_1,\vecv_2)>\delta_0$ 
for all $(\vecp_1,\vecv_1)\in C$, $(\vecp_2,\vecv_2)\in \partial\scrD$.
Hence for each $\delta\leq\delta_0$ the set
$M_\delta=\{\uvecS_1\in\S_1^{d-1}\col
(t\uvecS_1,\uvecS_1)\in\partial_\delta\scrD\}$ is contained in $U$,
and $\vol_{\S_1^{d-1}}(M_\delta)\leq\vol_{\S_1^{d-1}}(U)<\ve$.
\end{proof}

\begin{lem} \label{NPOSKEYLEMMA}
Given any $n\geq 1$, $t\geq 0$ and $\ve>0$, 
there exists some $\delta>0$ such that,
for every measurable subset $\scrD\subset\R^d\times\S_1^{d-1}$ with
$\bigl[\vol_{\R^d}\times\vol_{\S_1^{d-1}}\bigr](\scrD)\leq\delta$,
\begin{align} \notag
\bigl[(\vol_{\R^d})^n\times\vol_{\S_1^{d-1}}\bigr]
\Bigl(\Bigl\{(\vecS_1,\ldots,\vecS_{n},\uvecS_{n+1})\in
\R^{nd}\times\S_1^{d-1}\col
\sum_{k=1}^{n}\|\vecS_k\|< t, \hspace{60pt} &
\\ \label{NPOSKEYLEMMAINEQ}
\Bigl(\sum_{k=1}^{n}\vecS_k+\Bigl(t-\sum_{k=1}^{n}\|\vecS_k\|\Bigr)\uvecS_{n+1},
\uvecS_{n+1}\Bigr)\in\scrD\Bigr\}\Bigr) <\ve. &
\end{align}
\end{lem}
\begin{proof}
We may assume $t>0$ since otherwise the left hand side of 
\eqref{NPOSKEYLEMMAINEQ} is zero.
Let us first treat the case $n=1$. Set 
\begin{align}
U_t=\{(\vecS_1,\uvecS_2)\in\scrB_t^d\times\S_1^{d-1}\col
\vecS_1\notin\R_{\geq 0}\uvecS_2\}
\end{align} 
and let $\Theta_t$ be the map
\begin{align}
\Theta_t:U_t\to\R^d\times\S_1^{d-1};\qquad
(\vecS_1,\uvecS_2)\mapsto
\bigl(\vecS_1+(t-\|\vecS_1\|)\uvecS_2,\uvecS_2\bigr).
\end{align}
Then since $\scrB_t^d\times\S_1^{d-1}\setminus U_t$ has measure zero,
the left hand side of \eqref{NPOSKEYLEMMAINEQ} equals
$\vol\bigl(\Theta_t^{-1}(\scrD)\bigr)$,
where we write $\vol:=\vol_{\R^d}\times\vol_{\S_1^{d-1}}$.
But one verifies that $\Theta_t$ is a diffeomorphism from $U_t$ onto
$\scrB_t^d\times\S_1^{d-1}$. Hence $\vol\circ\,\Theta_t^{-1}$
is a bounded measure on $\R^d\times\S_1^{d-1}$
which is absolutely continuous with respect to $\vol$.
This implies the desired claim (cf.\ \cite[Thm.\ 6.11]{Rudin}).

In the remaining case $n\geq 2$, the left hand side of 
\eqref{NPOSKEYLEMMAINEQ} can be expressed as
\begin{align}
\int_{\substack{(\vecS_1,\ldots,\vecS_{n-1})\in\R^{(n-1)d}\\ T_{n-1}<t}}
\vol\Bigl(\Theta_{t-T_{n-1}}^{-1}\Bigl(\scrD-\sum_{j=1}^{n-1}\vecS_j
\Bigr)\Bigr)
\,d\!\vol_{\R^d}(\vecS_1)\cdots d\!\vol_{\R^d}(\vecS_{n-1}),
\end{align}
where $\scrD-\vecq:=\{(\vecp-\vecq,\vecv)\col(\vecp,\vecv)\in\scrD\}$
(also recall $T_{n-1}:=\sum_{j=1}^{n-1} \|\vecS_j\|$).
Hence it suffices to prove that for any given $\ve_1>0$ we can choose
$\delta>0$ so small that for every measurable subset 
$\scrD\subset\R^d\times\S_1^{d-1}$ with $\vol(\scrD)\leq\delta$ 
\textit{and for every $t'\in (0,t]$,}
we have $\vol\bigl(\Theta_{t'}^{-1}(\scrD)\bigr)<\ve_1$.
To prove this, for $t'$ so small that $\vol(U_{t'})<\ve_1$ we use the a priori 
bound $\vol\bigl(\Theta_{t'}^{-1}(\scrD)\bigr)\leq\vol(U_{t'})$;
the remaining $t'$-interval can then be treated using the above discussion 
of $\vol\circ\,\Theta_t^{-1}$, together with the relation 
$\Theta_{t'}^{-1}=L_{t'/t}\circ\Theta_t^{-1}\circ L_{t/t'}$
where $L_x:\R^d\times\S_1^{d-1}\to\R^d\times\S_1^{d-1}$ is the map
$(\vecp,\vecv)\mapsto(x\vecp,\vecv)$.
\end{proof}

Given any $\vecn\in\Z^M_{\geq 0}$ and subsets 
$\scrD_1',\ldots,\scrD_M'\subset\T^1(\R^d)$, we set
$n=\max(n_1,\ldots,n_M)$ and
\begin{align} \label{ANDDDDEF}
& \scrA_\vecn(\scrD_1',\ldots,\scrD_M')
:=\big\{(\vecS_1,\ldots,\vecS_{n+1})\col \Xi_{n_j}(t_j) \in\scrD_j',\; T_{n_j}\leq t_j< T_{n_j+1}\;(j=1,\ldots,M) \big\},
\end{align}
a subset of $\R^{(n+1)d}$. 

\begin{lem} \label{KEYLEMMAFORTHMMICRO2}
Let $\scrD_1,\ldots,\scrD_M$ and $t_1,\ldots,t_M$ be as in 
Theorem \ref{secThmMicro2GS}, and assume furthermore that each set
$\scrD_j$ is bounded.
Then for each $\vecn\in\Z^M_{\geq 0}$ and $\ve>0$, there exists a choice of
subsets $\scrD_j',\scrD_j''\subset\T^1(\R^d)$ with
$\scrD_j'\subset\scrD_j\subset\scrD_j''$ ($j=1,\ldots,M$), such that the
following holds:
\begin{enumerate}
	\item[(i)] There is some $\delta>0$ such that 
$\scrN_\delta\scrD_j'\subset\scrD_j$ and $\scrN_\delta\scrD_j\subset\scrD_j''$.
        \item[(ii)] 
\mbox{$\mu_{\vecalf,\vecbeta,\lambda}^{(n+1)}(\scrA_\vecn(\scrD_1'',\ldots,\scrD_M''))
-\ve
<\mu_{\vecalf,\vecbeta,\lambda}^{(n+1)}(\scrA_\vecn(\scrD_1,\ldots,\scrD_M))
<\mu_{\vecalf,\vecbeta,\lambda}^{(n+1)}(\scrA_\vecn(\scrD_1',\ldots,\scrD_M'))
+\ve.$}
        \item[(iii)] Both $\scrA_\vecn(\scrD_1',\ldots,\scrD_M')$ and
$\scrA_\vecn(\scrD_1'',\ldots,\scrD_M'')$ have boundaries of Lebesgue 
measure zero.
\end{enumerate}
\end{lem}
\begin{proof}
Given any $\delta>0$ we let $\scrF_1$ be the family of all cubes of the form
$C=\delta\vecm+[0,\delta]^d$ with $\vecm\in\Z^d$,
and fix $\scrF_2$ to be a finite family of closed subsets
$D\subset\S_1^{d-1}$ with diameter $<\delta$
(with respect to the metric $\varphi$) and boundary of measure zero, 
such that $\cup_{D\in\scrF_2} D=\S_1^{d-1}$.
Let $\scrF$ be the family of all sets $B=C\times D\subset\T^1(\R^d)$
with $C\in\scrF_1$ and $D\in\scrF_2$.
Now for each $j\in\{1,\ldots,M\}$ we define 
$\scrD_j'$ as the union of 
all $B\in\scrF$ with $\scrN_\delta B\subset\scrD_j$;
and define $\scrD_j''$ as the union of 
all $B\in\scrF$ with $B\cap\scrN_\delta\scrD_j\neq\emptyset.$
Note that both these unions are finite, since $\scrD_j$ is bounded, and 
(i) holds by construction.
Also note that $\scrD_j\setminus\scrD_j'\subset\partial_{(d+2)\delta}\scrD_j$
and $\scrD_j''\setminus\scrD_j\subset\partial_{(d+2)\delta}\scrD_j$.

We now prove that (ii) holds provided that $\delta$ is sufficiently small.
Every $(\vecS_1,\ldots,\vecS_{n+1})\in
\scrA_\vecn(\scrD_1,\ldots,\scrD_M)\setminus
\scrA_\vecn(\scrD_1',\ldots,\scrD_M')$ satisfies
$\Xi_{n_j}(t_j)\in\scrD_j\setminus\scrD_j'$ for some $j$, and hence,
using also Remark \ref{MUNPONEREL} repeatedly,
\begin{align} \notag
&\mu_{\vecalf,\vecbeta,\lambda}^{(n+1)}
\Bigl(\scrA_\vecn(\scrD_1,\ldots,\scrD_M)\setminus
\scrA_\vecn(\scrD_1',\ldots,\scrD_M')\Bigr)
\\ \label{KEYLEMMAFORTHMMICRO2JSUM}
& \leq\sum_{j=1}^M \mu_{\vecalf,\vecbeta,\lambda}^{(n_j+1)}
\bigl(\bigl\{
(\vecS_1,\ldots,\vecS_{n_j+1})\in\R^{(n_j+1)d}\col
\Xi_{n_{j}}(t_{j})\in\partial_{(d+2)\delta}\scrD_j,\:
T_{n_{j}}\leq t_{j}\bigr\}\bigr).
\end{align}
For any $j\in\{1,\ldots,M\}$, if $n_j=0$ then 
(via \eqref{PNLQ0EXPLNONE})
the $j$th term in the above sum equals
\begin{align}
\lambda\bigl(\bigl\{\uvecS_1\in\S_1^{d-1}\col (t_j\uvecS_1,\uvecS_1)\in
\partial_{(d+2)\delta}\scrD_j\bigr\}\bigr),
\end{align}
which tends to zero as $\delta\to 0$ by Lemma \ref{NZEROKEYLEMMA}.
On the other hand if $n_j\geq 1$, then the $j$th term in 
\eqref{KEYLEMMAFORTHMMICRO2JSUM}
equals 
$\mu_{\vecalf,\vecbeta,\lambda}^{(n_j+1)}\bigl(\varpi^{-1}(\scrM)\bigr)$
where $\varpi$ is the projection
\begin{align}
\varpi:\R^{n_jd}\times(\R^d\setminus\{\bn\})\to\R^{n_jd}\times\S_1^{d-1};
\qquad (\vecS_1,\ldots,\vecS_{n+1})\mapsto
(\vecS_1,\ldots,\vecS_n,\uvecS_{n+1}),
\end{align}
and $\scrM$ is the same set as in the left hand side of
\eqref{NPOSKEYLEMMAINEQ}, with $n_j,$ $t_j$ and $\partial_{(d+2)\delta}\scrD_j$
in place of $n$, $t$, $\scrD$.
Now $\mu_{\vecalf,\vecbeta,\lambda}^{(n_j+1)}\circ \varpi^{-1}$ is a
Borel probability measure on $\R^{n_jd}\times\S_1^{d-1}$ which is
absolutely continuous with respect to 
$(\vol_{\R^d})^n\times\vol_{\S_1^{d-1}}$.
Also 
\begin{align}
\bigl[\vol_{\R^d}\times\vol_{\S_1^{d-1}}\bigr]
\bigl(\partial_{(d+2)\delta}\scrD_j\bigr)\to 0
\qquad\text{as }\:\delta\to 0,
\end{align}
since $\scrD_j$ is bounded and has boundary of measure zero.
Hence by Lemma \ref{NPOSKEYLEMMA}, 
\begin{align}
\mu_{\vecalf,\vecbeta,\lambda}^{(n_j+1)}\bigl(\varpi^{-1}(\scrM)\bigr)\to 0
\qquad\text{as }\:\delta\to 0.
\end{align}
In conclusion we see that for $\delta$ sufficiently small,
the sum in \eqref{KEYLEMMAFORTHMMICRO2JSUM} must be $<\ve$,
so that the second inequality in (ii) holds.
Similarly, using 
$\scrD_j''\setminus\scrD_j\subset\partial_{(d+2)\delta}\scrD_j$, 
one proves that also the first inequality in (ii) must hold
for $\delta$ sufficiently small.

The property (iii) is proved in a similar way, making use the fact that
each set $\scrD_j'$ and $\scrD_j''$ is a finite union of sets $B\in\scrF$,
and that, by construction, each such set $B$ satisfies\linebreak
$\vol_{\S_1^{d-1}}\bigl(
\bigl\{\uvecS_1\in\S_1^{d-1}\col (t\uvecS_1,\uvecS_1)\in\partial B\bigr\}\bigr)
=0$ and 
$\bigl[\vol_{\R^d}\times\vol_{\S_1^{d-1}}\bigr]\bigl(\partial B\bigr)=0$.
\end{proof}

\subsection{Proof of Theorem \ref{secThmMicro2GS}} \label{ThmMicro2GSpfsec}
\subsubsection{Convergence} \label{ThmMicro2GSpfsecsubsec1}
Note that both sides of \eqref{secThm-eq2GS} remain unchanged if we replace
$\scrD_j$ with $\{(\vecq,\vecv)\in\scrD_j\col \|\vecq\|<t_j+1\}$ for
each $j=1,\ldots,M$. Hence from now on we may assume without loss of generality
that \textit{each set $\scrD_j$ is bounded.}
Writing $\vecv_k=\vecv_k(\vecq_0+\rho\vecbeta(\vecv_0),\vecv_0;\rho)$,
$\vecs_k=\vecs_k(\vecq_0+\rho\vecbeta(\vecv_0),\vecv_0;\rho)$
and $T_n=T_n(\vecq_0+\rho\vecbeta(\vecv_0),\vecv_0;\rho)
:=\rho^{d-1}\sum_{k=1}^n\|\vecs_k\|$,
the left hand side of \eqref{secThm-eq2GS} may be expressed as
$\lim_{\rho\to 0}\sum_{\vecn\in\Z^M_{\geq 0}} \lambda(\scrS_{\vecn,\rho})$,
where
\begin{align} \notag
\scrS_{\vecn,\rho}=
\Big\{ \vecv_0\in \S_1^{d-1} \col
\Bigl(\rho^{d-1}\vecq_0+\rho^d\vecbeta(\vecv_0)
+\sum_{k=1}^{n_j}\rho^{d-1}\vecs_k
+ (t_j-T_{n_j}) \vecv_{n_j}, \vecv_{n_j}\Bigr) \in\scrD_j 
& ,
\\ \label{ANRHODEF}
T_{n_j}\leq t_j <T_{n_j+1}\: (j=1,\ldots,M) \Big\} & .
\end{align}
The right hand side of \eqref{secThm-eq2GS} is, by definition, 
\begin{align} \label{secThmMicroproof2}
\PP_{\vecalf,\vecbeta,\lambda}\big(\Xi(t_1)\in\scrD_1,\ldots, \Xi(t_M)\in\scrD_M\big)=
\sum_{\vecn\in\Z^M_{\geq 0}}
\mu_{\vecalf,\vecbeta,\lambda}^{(n+1)}(\scrA_\vecn(\scrD_1,\ldots,\scrD_M)).
\end{align}
By Lemma \ref{diffLem}, for any given $\ve>0$ there
is a finite subset $S\subset\Z_{\geq 0}^M$ such that 
$\sum_{\vecn\in \Z_{\geq 0}^M\setminus S} \lambda(\scrS_{\vecn,\rho})<\ve$ 
holds for all sufficiently small $\rho>0$.
Hence to prove \eqref{secThm-eq2GS}, and also deduce that the
sum in \eqref{secThmMicroproof2} is indeed convergent,
it suffices to prove that 
for every \textit{fixed} $\vecn\in\Z_{\geq 0}^M$, we have
\begin{align} \label{secThmMicroproof1}
\lim_{\rho\to 0} \lambda(\scrS_{\vecn,\rho})
= \mu_{\vecalf,\vecbeta,\lambda}^{(n+1)}(\scrA_\vecn(\scrD_1,\ldots,\scrD_M)).
\end{align}

To prove \eqref{secThmMicroproof1}, let $\ve>0$ be given.
Take $\scrD_j',\scrD_j''\subset\T^1(\R^d)$ and $\delta>0$ so that all the
claims in Lemma \ref{KEYLEMMAFORTHMMICRO2} hold. %
By Lemma \ref{KEYLEMMAFORTHMMICRO2} (i), if $\rho>0$ is so small that
$\|\rho^{d-1}\vecq_0+\rho^d\vecbeta(\vecv_0)\|<\delta$ for all 
$\vecv_0\in\S_1^{d-1}$, then
\begin{align} \notag
\big\{ \vecv_0\in \S_1^{d-1} : 
(\vecs_1,\ldots,\vecs_{n+1})\in\rho^{-(d-1)}
\scrA_\vecn(\scrD_1',\ldots,\scrD_M')\big\}
\hspace{130pt} &
\\ \label{INCLUSIONREL}
\subset \scrS_{\vecn,\rho} \subset \big\{ \vecv_0\in \S_1^{d-1} : 
(\vecs_1,\ldots,\vecs_{n+1})\in\rho^{-(d-1)}
\scrA_\vecn(\scrD_1'',\ldots,\scrD_M'')\big\} &.
\end{align}
Hence by Theorem \ref{secThmMicroGS}, using also
Lemma \ref{KEYLEMMAFORTHMMICRO2} (ii) and (iii), we have
\begin{align} \label{secThmMicroproof1a}
& \limsup_{\rho\to 0} \lambda(\scrS_{\vecn,\rho})
\leq \mu_{\vecalf,\vecbeta,\lambda}^{(n+1)}
(\scrA_\vecn(\scrD_1'',\ldots,\scrD_M''))
< \mu_{\vecalf,\vecbeta,\lambda}^{(n+1)}
(\scrA_\vecn(\scrD_1,\ldots,\scrD_M))+\ve;
\\ \label{secThmMicroproof1b}
& \liminf_{\rho\to 0} \lambda(\scrS_{\vecn,\rho})
\geq \mu_{\vecalf,\vecbeta,\lambda}^{(n+1)}
(\scrA_\vecn(\scrD_1',\ldots,\scrD_M'))
> \mu_{\vecalf,\vecbeta,\lambda}^{(n+1)}
(\scrA_\vecn(\scrD_1,\ldots,\scrD_M))-\ve.
\end{align}
Since this is true for each $\ve>0$ we have now proved 
\eqref{secThmMicroproof1}.

\subsubsection{Uniformity} \label{ThmMicro2GSpfsecsubsec2}
We now turn to the statement about uniformity in \eqref{secThm-eq2GS}. 
By similar arguments as above one proves that
the right hand side of \eqref{secThm-eq2GS} is continuous as a
function of $(t_1,\ldots,t_M)$
at each point of $\adm(\scrD_1)\times\ldots\times\adm(\scrD_M)$.
Hence the desired statement about 
uniform convergence for $(t_1,\ldots,t_M)$ in compact 
subsets of $\adm(\scrD_1)\times\ldots\times\adm(\scrD_M)$ 
will follow if we can prove that
\begin{multline} \label{secThm-eq2GS-unif}
	\lim_{\rho\to 0}\lambda\bigl( \bigl\{ \vecv_0\in \S_1^{d-1} : \bigl(\rho^{d-1}\vecq(\rho^{-(d-1)}t_j(\rho)),\vecv(\rho^{-(d-1)} t_j(\rho))\bigr) \in\scrD_j, \; j=1,\ldots,M \bigr\} \bigr) \\
= \PP_{\vecalf,\vecbeta,\lambda}\big(\Xi(t_1)\in\scrD_1,\ldots, \Xi(t_M)\in\scrD_M\big) .
\end{multline}
holds for any functions $t_j(\rho)$ from $\R_{>0}$ to $\R_{\geq 0}$ satisfying
$t_j=\lim_{\rho\to 0}t_j(\rho)\in\adm(\scrD_j)$.
By Lemma \ref{diffLem}, we see that it suffices to prove that
\eqref{secThmMicroproof1} 
holds for any fixed $\vecn\in\Z_{\geq 0}^M$ and with
$\scrS_{\vecn,\rho}$ redefined using \eqref{ANRHODEF} with each 
``$t_j$'' replaced by ``$t_j(\rho)$''.

To this end, let us define $\scrE_\rho$ to be the 
set of those $\vecv_0\in\S_1^{d-1}$ for which there is some $j$ such that
exactly one of the two numbers $t_j(\rho)$, $t_j$
lies in %
$\bigl[T_{n_j},T_{n_j+1}\bigr)$.
The point of this is that if we keep $\rho>0$ so small that
$\|\rho^{d-1}\vecq_0+\rho^d\vecbeta(\vecv_0)\|+|t_j(\rho)-t_j|<\delta$ 
for all $j$ and $\vecv_0$, then for our redefined $\scrS_{\vecn,\rho}$ the inclusions 
in \eqref{INCLUSIONREL} \textit{can only fail for 
vectors $\vecv_0\in\scrE_\rho$.}
Now $\lim_{\rho\to 0} \lambda(\scrE_\rho)=0$,
as we see by applying Theorem \ref{secThmMicroGS} to the sets
\begin{align}
\scrA_\delta:=  %
\cup_{j=1}^M \cup_{k\in\{n_j,n_j+1\}\setminus\{0\}}
\bigl\{(\vecS_1,\ldots,\vecS_{n+1})\in\R^{(n+1)d}\col T_k\in
(t_j-\delta,t_j+\delta)\bigr\}
\end{align}
for $\delta=\frac 11, \frac 12, \frac 13,\ldots$.
(Indeed, note that $\lim_{\delta\to 0}
\mu_{\vecalf,\vecbeta,\lambda}^{(n+1)}(\scrA_\delta)=0$,
since $\mu_{\vecalf,\vecbeta,\lambda}^{(n+1)}$ is bounded and absolutely
continuous with respect to $(\vol_{\R^d})^{n+1}$.)
Using these observations, the proof of \eqref{secThmMicroproof1}
carries over to the present situation.
This completes the proof of Theorem \ref{secThmMicro2GS}.\hfill$\square$

\subsection{A counterexample} \label{ADMREMARKSEC}
We now give an example to show that the condition
$t_j\in\adm(\scrD_j)$ in the statement of Theorem \ref{secThmMicro2GS}
(or Theorem \ref{secThmMicro2}) cannot be disposed with.
Suppose $M=1$, $t_1>0$ and $\scrD_1=\scrB_{t_1}^d\times\S_1^{d-1}$.
Then as %
in the above proof we have 
\begin{align} \label{ADMREMARKPROOF0}
\sum_{n \geq 1} \lambda(\scrS_{(n),\rho})\to\sum_{n\geq 1}
\mu_{\vecalf,\vecbeta,\lambda}^{(n+1)}(\scrA_{(n)}(\scrD_1)) \qquad
\text{as }\:\rho\to 0.
\end{align}
However, in general we have
\begin{align} \label{ADMREMARKPROOF1}
\lambda(\scrS_{(0),\rho})\not\to
\mu_{\vecalf,\vecbeta,\lambda}^{(1)}(\scrA_{(0)}(\scrD_1)) \qquad
\text{as }\:\rho\to 0.
\end{align}
Indeed, there are many choices of $\vecq_0$, $\vecbeta$ and 
$\lambda$ such that 
$\rho^{d-1}\vecq_0+\rho^d\vecbeta(\vecv_0)+t_1\vecv_0\in\scrB_{t_1}^d$ holds
for all sufficiently small $\rho>0$ and all 
$\vecv_0$ in the support of $\lambda$, 
and in this case we have
\begin{align}
\lambda(\scrS_{(0),\rho}) &=\lambda\bigl(\bigl\{\vecv_0\in\S_1^{d-1}\col
(\rho^{d-1}\vecq_0+\rho^d\vecbeta(\vecv_0)+t_1\vecv_0,\vecv_0)\in\scrD_1,
\: \|\vecs_1\|>\rho^{-(d+1)}t_1\bigr\}\bigr)
\\ \notag
&=\lambda\bigl(\bigl\{\vecv_0\in\S_1^{d-1}\col 
\rho^{d-1}\tau_1(\vecq_0+\rho\vecbeta(\vecv_0),\vecv_0;\rho)>t_1\bigr\}\bigr)
\to \int_{t_1}^\infty \Phi_{\vecalf,\vecbeta}(\xi)\,d\xi,
\end{align}
where $\vecalf=-\vecq_0 M_0^{-1}$, cf.\ \cite[Cor.\ 4.2]{partI}.
This limit is in general non-zero. On the other hand
we have $\scrA_{(0)}(\scrD_1)=\emptyset$ since
$(t_1\uvecS_1,\uvecS_1)\notin\scrD_1$ for all $\uvecS_1\in\S_1^{d-1}$,
and this proves \eqref{ADMREMARKPROOF1}.
Combining \eqref{ADMREMARKPROOF0} and \eqref{ADMREMARKPROOF1} we see that
the limit relation \eqref{secThm-eq2GS} \textit{fails} for our $M=1$, $\scrD_1$,
$t_1$ and many choices of $\vecq_0,\vecbeta,\lambda$.

Note that we may take $\vecbeta\equiv\bn$ above, i.e.\ there are many choices
of $\vecq_0$ and $\lambda$ such that 
$\rho^{d-1}\vecq_0+t_1\vecv_0\in\scrB_{t_1}^d$ holds
for all sufficiently small $\rho>0$ and all $\vecv_0$ in the support 
of $\lambda$; thus the above example applies in particular to the situation in 
Theorem \ref{secThmMicro2}.

\subsection{Macroscopic initial conditions}
Finally we discuss the proof of Theorem \ref{secThmMacro2},
again in the setting of a general scattering map.
(The statement of the theorem remains the same. In the definition of the
limiting stochastic process we use 
$\mu_{\Lambda}^{(n)}$ from \eqref{muLambdandefGS}.)

The basic strategy of the proof is to mimic the proof of 
Theorem \ref{secThmMicro2GS} given in 
sections \ref{FOURLEMMASSEC}--\ref{ThmMicro2GSpfsec},
using Theorem \ref{secThmMacroGS} in place of Theorem \ref{secThmMicroGS}.
We will here only point out the main differences.

Both Lemma \ref{NZEROKEYLEMMA} and Lemma \ref{NPOSKEYLEMMA} can in the 
present case be replaced by the much simpler
\begin{lem} \label{NPOSKEYLEMMAMACRO}
Given any $n\geq 0$, $t\geq 0$ and $\ve>0$, 
there exists some $\delta>0$ such that,
for every measurable subset $\scrD\subset\R^d\times\S_1^{d-1}$ with
$\bigl[\vol_{\R^d}\times\vol_{\S_1^{d-1}}\bigr](\scrD)\leq\delta$,
\begin{align} \notag
\bigl[(\vol_{\R^d})^n\times\vol_{\S_1^{d-1}}\bigr]
\Bigl(\Bigl\{(\vecQ_0,\vecS_1,\ldots,\vecS_{n},\uvecS_{n+1})\in
\R^d\times\R^{nd}\times\S_1^{d-1}\col
\sum_{k=1}^{n}\|\vecS_k\|\leq t, \hspace{40pt} &
\\ \label{NPOSKEYLEMMAMACROINEQ}
\Bigl(\vecQ_0+\sum_{k=1}^{n}\vecS_k+(t-\sum_{k=1}^{n}\|\vecS_k\|)\uvecS_{n+1},
\uvecS_{n+1}\Bigr)\in\scrD\Bigr\}\Bigr) <\ve. &
\end{align}
\end{lem}
\begin{proof}
By Fubini's Theorem the left hand side is equal to 
$\bigl[\vol_{\R^d}\times\vol_{\S_1^{d-1}}\bigr](\scrD)$ times a finite
constant which only depends on $n$ and $t$.
\end{proof}
We replace the definition \eqref{ANDDDDEF} by 
\begin{align} \label{ANDDDDEFMACRO}
& \scrA_\vecn(\scrD_1',\ldots,\scrD_M')
:=\big\{(\vecQ_0,\vecS_1,\ldots,\vecS_{n+1})\col \Xi_{n_j}(t_j) \in\scrD_j',\; T_{n_j}\leq t_j< T_{n_j+1}\;(j=1,\ldots,M) \big\},
\end{align}
a subset of $\R^d\times\R^{(n+1)d}$. 
(Recall that $\Xi_n(t)$ is now given by \eqref{XNDEFMACRO}.)

Now Lemma \ref{KEYLEMMAFORTHMMICRO2} and the 
discussion in Section \ref{ThmMicro2GSpfsec} carries over with only
few and obvious changes. 
By a simple approximation argument we may assume from the start
that $\Lambda$ has compact support; using this we may then also assume
without loss of generality that each set $\scrD_j$ is bounded, as before.
The proof of the convergence (Sec.\ \ref{ThmMicro2GSpfsecsubsec1}) 
takes a somewhat simpler shape in the present case, since
\eqref{ANRHODEF} is now replaced by the exact identity
\begin{align} \label{ANRHODEFMACRO}
\scrS_{\vecn,\rho}=
\Big\{(\vecQ_0,\vecV_0)\in\T^1(\rho^{d-1}\scrK_\rho)\col
(\vecQ_0,\vecS_1,\ldots,\vecS_{n+1})\in\scrA_\vecn(\scrD_1,\ldots,\scrD_M)
\Bigr\},
\end{align}
with $\vecS_k=\vecS_k(\vecQ_0,\vecV_0;\rho)$.

\section{A continuous-time Markov process}\label{secExtended}

As mentioned in Section \ref{secMacroscopic}, the operator $L_t$ describing the dynamics of a particle cloud in the Boltzmann-Grad limit does not form a semigroup, and thus the stochastic process $\Xi(t)$ is not Markovian. To overcome this difficulty, we set 
\begin{align}
\XX:=
\bigl\{(\vecQ,\vecV,\scrT,\vecV_+)\in
\T^1(\RR^d)\times\RR_{\geq 0}\times\S_1^{d-1}\col
\varphi(\vecV,\vecV_+)>B_\Theta\bigr\}
\end{align}
and extend phase space by the map
\begin{align}
\scrR:\T^1(\rho^{d-1}\scrK_\rho) \to \XX,
\qquad (\vecQ,\vecV) \mapsto  (\vecQ,\vecV,\scrT,\vecV_+),
\end{align}
where
\begin{equation}
	\scrT=\scrT(\vecQ,\vecV)=\|\vecS_1(\vecQ,\vecV;\rho)\|
\end{equation}
represents the free path length until the next scatterer, and $\vecV_+=\vecV_+(\vecQ,\vecV)$ is the velocity after the collision. 

For random initial data $(\vecQ_0,\vecV_0)$, distributed as before with respect to an absolutely continuous probability measure $\Lambda$, the dynamics in this extended phase space is again described, in the limit $\rho\to 0$, by a stochastic process $\widehat\Xi$ (cf.\ Theorem \ref{secThmMacro2-X} below) which is Markovian (Proposition \ref{SEMIGROUPLEM}). 

\subsection{Transcription of Theorem \ref{secThmMacro2}}

Set
\begin{align}
& \scrX^{(n)}=\bigl\{(\vecQ_0,\vecv_0,\xi_1,\ldots,\xi_n,\vecv_n)
\in\T^1(\RR^d)\times(\RR_{\geq 0}\times\S_1^{d-1})^n
\\ \notag
& \hspace{200pt}
\col
\varphi(\vecv_{j-1},\vecv_{j})>B_\Theta,\: j=1,\ldots,n\bigr\}
\end{align}
(thus $\scrX^{(1)}=\XX$)
and define the volume measure $\sigma_n$ on $\scrX^{(n)}$ by
\begin{align}
d\sigma_n(\vecQ_0,\vecv_0,\xi_1,\ldots,\xi_n,\vecv_n)
=d\!\vol_{\RR^d}(\vecQ_0) \, d\!\vol_{\S_1^{d-1}}(\vecv_0) \, d\xi_1
\,\cdots\,d\xi_n\, d\!\vol_{\S_1^{d-1}}(\vecv_n).
\end{align}
Given a probability density $f\in\L^1(\XX,\sigma_1)$ 
we define for every $n\geq 1$ the probability measure $\nu_{f}^{(n)}$ on $\scrX^{(n)}$ by
\begin{multline} \label{NUFNDEF}
	\nu_{f}^{(n)}(\scrA):=\int_{\scrA} f(\vecQ_0,\vecv_0,\xi_1,\vecv_1)\prod_{j=1}^{n-1} p_{\vecnull,\vecbeta^+_{\vecv_{j-1}}}(\vecv_j,\xi_{j+1},\vecv_{j+1}) 
\, d\sigma_n(\vecQ_0,\vecv_0,\xi_1,\ldots,\xi_n,\vecv_n),
\end{multline}
for any Borel subset $\scrA\subset\scrX^{(n)}$.
We will use the shorthand notation $T_n=\xi_1+\ldots+\xi_n$. Let us define
\begin{equation} \label{XNDEFMACRO-X}
	\widehat\Xi_n(t):= \bigg( \vecQ_0+ \sum_{j=1}^n \xi_j\vecv_{j-1} + (t-T_n) \vecv_n, \vecv_n, T_{n+1}-t, \vecv_{n+1} \bigg) .
\end{equation}
The stochastic process $\widehat\Xi(t)$ is now characterized via the probability
\begin{multline} \label{PPLQ0LDEF-X}
		\PP_{f}\big(\widehat \Xi(t_1)\in\scrD_1,\ldots, \widehat \Xi(t_M)\in\scrD_M\big) 
		\\
	  := \sum_{\vecn\in\ZZ_{\geq 0}^M} \PP_{f}^{(\vecn)}\big(\widehat \Xi(t_1)\in\scrD_1,\ldots, \widehat\Xi(t_M)\in\scrD_M \text{ and } T_{n_1}\leq t_1<T_{n_1+1},\ldots, T_{n_M}\leq t_M<T_{n_M+1}  \big),
\end{multline}	  
where $\scrD_j\subset\XX$ %
are Borel subsets, and
\begin{multline} \label{PNLQ0LDEF-X}
	\PP_{f}^{(\vecn)}\big(\widehat\Xi(t_1)\in\scrD_1,\ldots, \widehat\Xi(t_M)\in\scrD_M \text{ and } T_{n_1}\leq t_1<T_{n_1+1},\ldots, T_{n_M}\leq t_M<T_{n_M+1} \big) \\
	:=
	\nu_{f}^{(n+1)}\big(\big\{(\vecQ_0,\vecv_0,\xi_1,\vecv_1\ldots,\xi_{n+1},\vecv_{n+1})\col \widehat\Xi_{n_j}(t_j) \in\scrD_j,\; T_{n_j}\leq t_j\;(j=1,\ldots,M) \big\}\big) ,
\end{multline}
with $n:=\max(n_1,\ldots,n_M)$.
Note that we have automatically 
$t_j\leq T_{n_j+1}$ in the right hand side of \eqref{PNLQ0LDEF-X},
by the definition of $\XX$, and the subset corresponding to $t_j= T_{n_j+1}$ 
of course has measure zero with respect to $\nu_{f}^{(n+1)}$.

The following theorem is an extension of Theorem \ref{secThmMacro2}. 
Set $\widehat F_t:=\scrR\circ F_t: \T^1(\rho^{d-1}\scrK_\rho)\to\scrX$.

\begin{thm}\label{secThmMacro2-X}
Fix a lattice $\scrL$ and let $\Lambda$ be a Borel probability measure on $\T^1(\RR^d)$ which is absolutely continuous with respect to Lebesgue measure. Then, for any $t_1,\ldots,t_M\in\RR_{\geq 0}$, and any subsets $\scrD_1,\ldots,\scrD_M\subset \XX$ with 
$\sigma_1(\partial\scrD_1)=\ldots=\sigma_1(\partial\scrD_M)=0$,
\begin{multline} \label{secThm-eq2-macro-X}
	\lim_{\rho\to 0}\Lambda\big(\big\{(\vecQ_0,\vecV_0)\in \T^1(\rho^{d-1}\scrK_\rho) : \widehat F_{t_1}(\vecQ_0,\vecV_0) \in \scrD_1,\ldots, \widehat F_{t_M}(\vecQ_0,\vecV_0) \in \scrD_M \big\}\big) \\
= \PP_{f}\big(\widehat\Xi(t_1)\in\scrD_1,\ldots, \widehat\Xi(t_M)\in\scrD_M\big) 
\end{multline}
where
\begin{equation}
	f(\vecQ,\vecV,\xi,\vecV_+)=\Lambda'(\vecQ,\vecV) p(\vecV,\xi,\vecV_+),
\end{equation}
for $p(\vecV,\xi,\vecV_+)$ as in Remark \ref{PALFBETCONTREM}.
The convergence is uniform for $t_1,\ldots,t_M$ in compact subsets of $\RR_{\geq 0}$. 
\end{thm}

\begin{proof}%
This is analogous to the proof of Theorem \ref{secThmMacro2},
using in place of Theorem \ref{secThmMacroGS} the fact that,
for any subset $\scrA\subset\scrX^{(n)}$
with $\sigma_n(\partial\scrA)=0$,
\begin{align} \notag
\lim_{\rho\to 0}  \Lambda\bigl(\bigl\{
(\vecQ_0,\vecV_0)\in\T^1(\rho^{d-1}\scrK_\rho)\col
\big(\vecQ_0,\vecV_0,  %
\rho^{d-1}\tau_1(\rho^{1-d}\vecQ_0,\vecV_0;\rho),
\vecv_1(\rho^{1-d}\vecQ_0,\vecV_0;\rho),\ldots, \hspace{10pt}&
\\ \label{pathwayEq-macro-cor}
\rho^{d-1}\tau_n(\rho^{1-d}\vecQ_0,\vecV_0;\rho),
\vecv_n(\rho^{1-d}\vecQ_0,\vecV_0;\rho)\big)\in\scrA\bigr\}\bigr) &
\\ \notag
=  \int_{\scrA}
p(\vecv_0,\xi_1,\vecv_1) \prod_{j=1}^{n-1}
p_{\bn,\vecbeta_{\vecv_{j-1}}^+}(\vecv_j,\xi_{j+1},\vecv_{j+1})
\Lambda'(\vecQ_0,\vecv_0)\,d\sigma_n(\vecQ_0,\vecv_0,\xi_1,\ldots,\xi_n,
\vecv_n)&.
\end{align}
The proof of this statement is almost identical to that of
Theorem \ref{secThmMacroGS}, with Theorem \ref{secThmMicroGS} replaced by Theorem
\ref{pathwayThm}.

The second main ingredient of the proof of Theorem \ref{secThmMacro2-X} is a substitute
of Lemma \ref{NPOSKEYLEMMAMACRO}, which we formulate as follows. Given any $n\geq 0$,
$t\geq 0$, $\ve>0$, there exists some $\delta>0$
such that, for every measurable subset $\scrD\subset\XX$
with $\sigma_1(\scrD)\leq\delta$,
\begin{align} \label{KEYLEMMAMACRO-X-bound}
\sigma_{n+1}
\Bigl(\Bigl\{(\vecQ_0,\vecv_0,\xi_1,\ldots,\xi_{n+1},\vecv_{n+1})\in\scrX^{
(n+1)}
\col \widehat{\Xi}_n(t)\in\scrD,\: T_n\leq t\Bigr\}\Bigr)<\ve.
\end{align}
To establish \eqref{KEYLEMMAMACRO-X-bound}, note that the left hand side equals
\begin{align}
\frac{t^n}{n!}
\vol_{\S_1^{d-1}}(\scrV_{\vece_1})^n \sigma_1(\scrD).
\end{align}
\end{proof}

\subsection{A semigroup of propagators}

We write $\Ll(\XX,\sigma_1)$ for the space of measurable functions
$f:\XX\to\R$ satisfying $\int_C \bigl | f \bigr | \, d\sigma_1<\infty$ for any
compact subset \mbox{$C\subset\overline{\scrX}$.}
In order to show that $\widehat\Xi(t)$ is a Markov process, 
for each $t\geq 0$ we introduce the operator 
\begin{equation}
	K_t : \Ll(\XX,\sigma_1) \to \Ll(\XX,\sigma_1)
\end{equation}
by the relation
\begin{equation}\label{Kt}
	\int_\scrD [K_t f](\vecQ,\vecV,\xi,\vecV_+)\, d\sigma_1(\vecQ,\vecV,\xi,\vecV_+)
= \PP_{f}(\widehat\Xi(t)\in\scrD)
\end{equation}
for every $f\in\Ll(\XX,\sigma_1)$ and every bounded Borel subset
$\scrD\subset\XX$.
The right hand side of \eqref{Kt} is defined by extending the definition of 
$\PP_f$ in \eqref{NUFNDEF}, \eqref{PPLQ0LDEF-X}, \eqref{PNLQ0LDEF-X}
for $f\in\Ll(\XX,\sigma_1)$: Note that for a given 
$f\in\Ll(\XX,\sigma_1)$, \eqref{NUFNDEF} defines $\nu_f^{(n)}$
as a signed Borel measure on any bounded open subset of $\scrX^{(n)}$,
and then \eqref{PPLQ0LDEF-X}, \eqref{PNLQ0LDEF-X} define
$\scrD\mapsto\PP_{f}(\widehat\Xi(t)\in\scrD)$ 
as a signed Borel measure on any bounded open subset of $\XX$.
Since this signed measure is absolutely continuous with respect to $\sigma_1$,
there exists a unique $K_t f\in \Ll(\XX,\sigma_1)$ such that
\eqref{Kt} holds for each bounded Borel subset $\scrD\subset\XX$.

If $f\in\L^1(\XX,\sigma_1)$ then
$\scrD\mapsto\PP_{f}(\widehat\Xi(t)\in\scrD)$ is in fact a signed Borel measure
on all of $\XX$, of total variation $\leq \|f\|_{\L^1}$.
Hence the restriction of $K_t$ to $\L^1(\XX,\sigma_1)$ maps into
$\L^1(\XX,\sigma_1)$, and gives a bounded linear operator 
on $\L^1(\XX,\sigma_1)$ of norm $\leq 1$.

Note that $K_t$ commutes with the translation operators $\{ T_\vecR : \vecR\in\RR^d\}$,
\begin{equation}
        [T_\vecR f](\vecQ,\vecV,\xi,\vecV_+) := f(\vecQ-\vecR,\vecV,\xi,\vecV_+),
\end{equation}
and with the rotation operators $\{ R_K : K\in \O(d)\}$,
\begin{equation}
        [R_K f](\vecQ,\vecV,\xi,\vecV_+) := f(\vecQ K,\vecV K,\xi,\vecV_+ K) ,
\end{equation}
cf.~Remark \ref{rot-inv}.

\begin{prop}\label{eigenlemma}
Suppose that the flow $F_t$ preserves the Liouville measure $\nu$
(cf.\ Remark~\ref{PRESERVELIOUVILLEREMARK}).
Then the function $f(\vecQ,\vecV,\xi,\vecV_+)=p(\vecV,\xi,\vecV_+)$
satisfies
\begin{equation}
	K_t f = f 
\end{equation}
for all $t\geq 0$.
\end{prop}

\begin{proof}
Eq.~\eqref{secThm-eq2-macro-X} implies that for any {\em bounded} set $\scrD$ with boundary of Lebesgue measure zero,
\begin{equation} \label{secThm-eq2-macro-XX}
	\nu\big(\big\{(\vecQ_0,\vecV_0)\in \T^1(\rho^{d-1}\scrK_\rho) : \widehat F_{t}(\vecQ_0,\vecV_0)\in \scrD \big\}\big) 
\to \PP_{f}\big(\widehat\Xi(t)\in\scrD\big) 
\end{equation}
as $\rho\to 0$, %
with $f$ as assumed above. The $F_t$-invariance of $\nu$ implies that the left hand side of \eqref{secThm-eq2-macro-XX} equals
\begin{equation} \label{secThm-eq2-macro-XXX}
	\nu\big(\big\{(\vecQ_0,\vecV_0)\in \T^1(\rho^{d-1}\scrK_\rho) : (\vecQ_0,\vecV_0,\scrT(\vecQ_0,\vecV_0),\vecV_+(\vecQ_0,\vecV_0))\in \scrD \big\}\big) ,
\end{equation}
for all $t\geq 0$, and hence
\begin{equation}
	\PP_{f}\big(\widehat\Xi(t)\in\scrD\big)=\PP_{f}\big(\widehat\Xi(0)\in\scrD\big)
	=\nu_f^{(1)}(\scrD)
=\int_\scrD f(\vecQ,\vecV,\xi,\vecV_+)\, d\sigma_1(\vecQ,\vecV,\xi,\vecV_+).
\end{equation}
Hence by \eqref{Kt}, $K_tf=f\in \Ll(\XX,\sigma_1)$.
\end{proof}

Using \eqref{Kt} we can write, more explicitly,
\begin{equation} \label{KTEXPLICIT}
K_t = \sum_{n=0}^\infty K_t^{(n)}	,
\end{equation}
where
\begin{multline}
	\int_\scrD [K_t^{(n)} f](\vecQ,\vecv_n,\xi_{n+1},\vecv_{n+1})\,
	d\!\vol_{\RR^d}(\vecQ)\,d\!\vol_{\S_1^{d-1}}(\vecv_n)\,d\xi_{n+1}\,d\!\vol_{\S_1^{d-1}}(\vecv_{n+1}) \\
	= \nu_{f}^{(n+1)}\big(\big\{(\vecQ_0,\vecv_0,\xi_1,\vecv_1\ldots,\xi_{n+1},\vecv_{n+1}): \widehat\Xi_{n}(t) \in\scrD,\; T_{n}\leq t\big\}\big) .
\end{multline}
So in the case $n=0$,
\begin{equation} \label{KT0EXPLDEF}
[K_t^{(0)} f](\vecQ,\vecv_0,\xi,\vecv_{1}) =
f\big(\vecQ-t\vecv_0,\vecv_0,\xi+t,\vecv_1\big) ,
\end{equation}
and for $n\geq 1$,
\begin{multline} \label{KTNEXPLDEF}
[K_t^{(n)} f](\vecQ,\vecv_n,\xi,\vecv_{n+1}) =
\int_{T_n\leq t}  f\bigg(\vecQ-\bigg(\sum_{j=1}^n \xi_j\vecv_{j-1} + (t-T_n) \vecv_n\bigg),\vecv_0,\xi_1,\vecv_1\bigg)\\ 
\times\prod_{j=1}^{n} p_{\vecnull,\vecbeta^+_{\vecv_{j-1}}}(\vecv_j,\xi_{j+1},\vecv_{j+1}) \,
d\!\vol_{\S_1^{d-1}}(\vecv_0)\,d\xi_1\cdots d\!\vol_{\S_1^{d-1}}(\vecv_{n-1})\,d\xi_n 
\end{multline}
with the shorthand $\xi_{n+1}:=\xi+t-T_n$.
We remark that when restricting $K_t$ and $K_t^{(n)}$ to %
$\L^1(\XX,\sigma_1)$,
the right hand side of \eqref{KTNEXPLDEF} can be estimated from above
as in the proof of Lemma~\ref{diffLem}, yielding
\begin{align} \label{KTNNORMBOUND}
\bigl\|K_t^{(n)}\bigr\|_{\L^1}
\leq 
\frac{t^{n-1}\vol(\scrB_1^{d-1})^{n-1}}{(n-1)!}
\qquad \text{for }\: n\geq 1,
\end{align}
and hence the sum \eqref{KTEXPLICIT} is uniformly operator convergent
on $\L^1(\XX,\sigma_1)$.

The following proposition implies $\widehat\Xi(t)$ is Markovian.

\begin{prop} \label{SEMIGROUPLEM}
The family $\{ K_t : t\geq 0\}$ forms a semigroup on 
$\Ll(\XX,\sigma_1)$, and a contraction semigroup on $\L^1(\XX,\sigma_1)$.
\end{prop}

\begin{proof}
Note that, for $f\in \Ll(\XX,\sigma_1)$, $0\leq s\leq t$, $0\leq m<n$,
\begin{align} \notag
[K_{t-s}^{(n-m)}K_{s}^{(m)} f](\vecQ,\vecv_n,\xi,\vecv_{n+1})  =
\int_{\substack{T_n\leq t \\ T_m\leq s <T_{m+1}}}  f\bigg(\vecQ-\bigg(\sum_{j=1}^n \xi_j\vecv_{j-1} + (t-T_n) \vecv_n\bigg),\vecv_0,\xi_1,\vecv_1\bigg) &
\\ \label{KNMMMFREL}
\times\prod_{j=1}^{n} p_{\vecnull,\vecbeta^+_{\vecv_{j-1}}}(\vecv_j,\xi_{j+1},\vecv_{j+1})\,
d\!\vol_{\S_1^{d-1}}(\vecv_0)\,d\xi_1\cdots d\!\vol_{\S_1^{d-1}}(\vecv_{n-1})\,d\xi_n  & ,
\end{align}
and for $m=n>0$,
\begin{multline}
[K_{t-s}^{(0)}K_{s}^{(n)} f](\vecQ,\vecv_n,\xi,\vecv_{n+1}) =
\int_{T_n\leq s}  f\bigg(\vecQ-\bigg(\sum_{j=1}^n \xi_j\vecv_{j-1} + (t-T_n) \vecv_n\bigg),\vecv_0,\xi_1,\vecv_1\bigg)\\ 
\times\prod_{j=1}^{n} p_{\vecnull,\vecbeta^+_{\vecv_{j-1}}}(\vecv_j,\xi_{j+1},\vecv_{j+1})\,
d\!\vol_{\S_1^{d-1}}(\vecv_0)\,d\xi_1\cdots d\!\vol_{\S_1^{d-1}}(\vecv_{n-1})\,d\xi_n  .
\end{multline}
Therefore
\begin{equation}
	\sum_{m=0}^n K_{t-s}^{(n-m)} K_s^{(m)} = K_t^{(n)}
\end{equation}
and thus
\begin{equation}
	K_{t-s}K_s = \sum_{m,n=0}^\infty K_{t-s}^{(m)} K_s^{(n)}
	= \sum_{n=0}^\infty \sum_{m=0}^n K_{t-s}^{(n-m)} K_s^{(m)} =K_t. 
\end{equation}
This proves the semigroup property.

We now consider the action restricted to $\L^1(\XX,\sigma_1)$.
Since we have already noted that $\|K_t\|_{\L^1}\leq 1$ for all $t$, 
it only remains to prove that 
for any given $f\in\L^1(\XX,\sigma_1)$ the map
$\R_{\geq 0}\ni t\mapsto K_t f\in \L^1(\XX,\sigma_1)$ is continuous.
In view of \eqref{KTEXPLICIT} and \eqref{KTNNORMBOUND} it suffices to prove
that the map $t\mapsto K_t^{(n)}f$ is continuous, for each $n\geq 0$.
This is clear for $n=0$, thus we now assume $n\geq 1$.
Given any $s\geq 0$ and $h>0$ we split the integral \eqref{KTNEXPLDEF}
for $t=s+h$ as
\begin{align}
K_{s+h}^{(n)}f=I_1+I_2,
\end{align}
where $I_1$ corresponds to $\xi_1<h$ and $I_2$ corresponds to $\xi_1>h$. 
Repeated use of Remark \ref{PNORMREMARK} gives
\begin{align}
\|I_1\|_{\L^1}
\leq\int_{\XX\cap\{\xi<h\}} \bigl|f(\vecQ,\vecV,\xi,\vecV_+)\bigr|
\, d\sigma_1(\vecQ,\vecV,\xi,\vecV_+).
\end{align}
Furthermore, \eqref{KNMMMFREL} with $m=0$ gives
$I_2=K_s^{(n)}[K_h^{(0)} f]$, and hence 
\begin{align}
\bigl\|K_{s+h}^{(n)}f-K_s^{(n)}f\bigr\|%
\leq \|I_1\| + \bigl\|K_s^{(n)}([K_{h}^{(0)}f]-f)\bigr\|%
\leq \|I_1\| + \bigl\|[K_{h}^{(0)}f]-f\bigr\|\to 0
\end{align}
as $h\to 0^+$, uniformly with respect to $s\geq 0$.
This proves the desired continuity.
\end{proof}

\subsection{The Fokker-Planck-Kolmogorov equation}

We will now derive the Fokker-Planck-Kolmogorov equation of the Markov process 
$\widehat\Xi(t)$.

We introduce convenient spaces of continuous functions on $\XX$
and $\R_{\geq 0}\times\XX$ as follows.
Recall the definition of $J(\vecv_1,\vecv_2)$
in Remark \ref{PALFBETEXPLICITREM};
we consider $J$ as a function $\XX\to\R_{>0}$ by letting
$J(\vecQ,\vecV,\xi,\vecV_+):=J(\vecV,\vecV_+)$. 
We also write $\vecQ=(Q_1,\ldots,Q_d)\in\R^d$.
Now set
\begin{align} \notag
& \|f\|_J:={\sup}_{\XX} |f|/J \hspace{50pt} 
(\text{for }\:f:\XX\to\R);
\\ \notag
& \C_J(\XX):=\bigl\{f\in\C(\XX)\col
\|f\|_J%
<\infty\bigr\};
\\
& \C_J^1(\XX):=\bigl\{f\in\C_J(\XX)\col
\partial_{Q_1}f,\partial_{Q_2}f,\ldots,\partial_{Q_d}f,
\partial_\xi f\in\C_J(\XX)
\bigr\};
\\ \notag
& \C_J(\R_{\geq 0}\times\XX):=
\bigl\{f\in\C(\R_{\geq 0}\times\XX)\col
{\sup}_{t\in[0,T]} \|f(t,\cdot)\|_J<\infty,\:\forall T>0\bigr\};
\\ \notag
& \C_J^1(\R_{\geq 0}\times\XX)
:=\bigl\{f\in\C_J(\R_{\geq 0}\times\XX)\col
\partial_t f,\partial_{Q_1}f,\partial_{Q_2}f,\ldots,\partial_{Q_d}f
,\partial_\xi f 
\\ \notag
& \hspace{300pt} 
\in\C_J(\R_{\geq 0}\times\XX) \bigr\}.
\end{align}

If $d=2$ then we impose from now on the following 
extra assumption on the scattering map:
\begin{align} \label{ASSUMPTIONSTAR}
\vartheta_1'(\varphi)\neq\vartheta_2'(\varphi)\:
\text{ for (Lebesgue-)almost every $\varphi\in(-\sfrac\pi 2,\sfrac\pi 2)$,}
\end{align}
where $\vartheta_1,\vartheta_2$ are defined via \eqref{TAUJDEF}.
This assumption is always fulfilled 
in the case when the flow $F_t$ preserves the Liouville measure 
(cf.\ Remark \ref{PRESERVELIOUVILLEREMARK}).

\begin{thm}\label{FPK}
For any $f_0\in\C_J^1(\XX)$, the function 
$f(t,\vecQ,\vecV,\xi,\vecV_+):=K_t f_0(\vecQ,\vecV,\xi,\vecV_+)$ belongs to
$\C_J^1(\R_{\geq 0}\times\XX)$ and is the unique solution 
in $\C_J^1(\R_{\geq 0}\times\XX)$ of the differential equation
\begin{multline} \label{FPKEQ}
	\bigg[ \partial_t + \vecV\cdot\nabla_\vecQ - \partial_\xi \bigg] f(t,\vecQ,\vecV,\xi,\vecV_+) \\
	= \int_{\S_1^{d-1}}  f(t,\vecQ,\vecv_0,0,\vecV)
p_{\vecnull,\vecbeta^+_{\vecv_{0}}}(\vecV,\xi,\vecV_+) \,
d\!\vol_{\S_1^{d-1}}(\vecv_0) 
\end{multline}
with $f(0,\vecQ,\vecV,\xi,\vecV_+)\equiv f_0(\vecQ,\vecV,\xi,\vecV_+)$.
\end{thm}

\begin{remark} \label{ZREMARK}
The equation \eqref{FPKEQ} can also be expressed as
\begin{align} \label{FPKEQ2}
\partial_t f(t,\vecQ,\vecV,\xi,\vecV_+)
=[Z f(t,\cdot)](\vecQ,\vecV,\xi,\vecV_+) 
\end{align}
where the operator $Z$ (acting on functions on $\XX$)
is defined by
\begin{align} \label{ZDEF}
[Zg](\vecQ,\vecV,\xi,\vecV_+)
= & \big[\partial_\xi-\vecV\cdot\nabla_\vecQ \big] g(\vecQ,\vecV,\xi,\vecV_+) 
\\ \notag
& +\int_{\S_1^{d-1}}  g(\vecQ,\vecv_0,0,\vecV)
p_{\vecnull,\vecbeta^+_{\vecv_{0}}}(\vecV,\xi,\vecV_+) \,
d\!\vol_{\S_1^{d-1}}(\vecv_0).
\end{align}
Thus on a formal level we have $K_t=e^{tZ}$.
\end{remark}

To prepare for the proof of Theorem \ref{FPK} we first establish
a series of lemmas.
For any $\vecv\in\S_1^{d-1}$ and $\delta>0$ we write $\scrN_\delta(\vecv)$ for
the closed $\delta$-neighborhood of $\vecv$ in $\S_1^{d-1}$, i.e.
\begin{align}
\scrN_\delta(\vecv)=\{\vecw\in\S_1^{d-1}\col\varphi(\vecw,\vecv)\leq\delta\}.
\end{align}
Given $n\geq 1$ and $\vecv_n\in\S_1^{d-1}$ we write
\begin{align} \label{VVNNDEF}
\scrV^{[n]}_{\vecv_n}:=\bigl\{(\vecv_0,\ldots,\vecv_{n-1})\in (\S_1^{d-1})^n 
\col\varphi(\vecv_j,\vecv_{j+1})>B_\Theta,\:j=0,\ldots,n-1\bigr\}.
\end{align}
The point of this is that the integrand in \eqref{KTNEXPLDEF} vanishes for 
all $(\vecv_0,\ldots,\vecv_{n-1})$ outside $\scrV^{[n]}_{\vecv_n}$.
\begin{lem} \label{GOODCOMPACTLEM}
Given $n\geq 1$, $\ve>0$ and $\vecv_n^0,\vecv_{n+1}^0\in\S_1^{d-1}$ with 
$\varphi(\vecv_n^0,\vecv_{n+1}^0)>B_\Theta$, there exist some 
$\delta>0$
and a compact subset $D\subset\cap_{\vecv_n\in \scrN_\delta(\vecv_n^0)}
\scrV_{\vecv_n}^{[n]}$ such that
for all $\vecv_n\in \scrN_\delta(\vecv_n^0)$, 
$\vecv_{n+1}\in \scrN_\delta(\vecv_{n+1}^0)$, the following holds:
\begin{enumerate}
\item[(i)] $\varphi(\vecv_n,\vecv_{n+1})>B_\Theta$;
\item[(ii)] %
$\displaystyle{\int_{\scrV_{\vecv_n}^{[n]}\setminus D} \prod_{j=0}^n 
J(\vecv_j,\vecv_{j+1}) \, d\!\vol_{\S_1^{d-1}}(\vecv_0) 
\cdots d\!\vol_{\S_1^{d-1}}(\vecv_{n-1})<\ve;}$
\item[(iii)] if $d=2$ then %
$\vecbeta_{\vece_1}^-(\vecv_{j+1}K(\vecv_{j}))_\perp
\neq(\vecbeta_{\vecv_{j-1}}^+(\vecv_{j})K(\vecv_{j}))_\perp$
for all $j\in\{1,\ldots,n\}$ and all $(\vecv_0,\ldots,\vecv_{n-1})\in D.$
\end{enumerate}
\end{lem}

\begin{proof}[Proof of Lemma \ref{GOODCOMPACTLEM} if $d\geq 3$]
Note $\int_{\scrV_\vecv} J(\vecv',\vecv)\,d\!\vol_{\S_1^{d-1}}(\vecv')
=\vol(\scrB_1^{d-1})$
(cf.\ \eqref{exactpos2-1hit-tpdef}--\eqref{JFORMULA});
thus the integral in (ii) with $D=\emptyset$
is absolutely convergent.
Note also that if $K\in\SO(d)$ is any rotation with $\vecv_nK=\vecv_n^0$,
the left hand side in (ii) can be rewritten as
\begin{align}
& J(\vecv_n,\vecv_{n+1})
\int_{\scrV_{\vecv_n^0}^{[n]}\setminus DK} \Bigl(\prod_{j=0}^{n-2} 
J(\vecv_j,\vecv_{j+1})\Bigr)
J(\vecv_{n-1},\vecv_n^0) \, d\!\vol_{\S_1^{d-1}}(\vecv_0) 
\cdots d\!\vol_{\S_1^{d-1}}(\vecv_{n-1}).
\end{align}
Now we may fix a compact subset $D\subset\scrV_{\vecv_n^0}^{[n]}$
such that (ii) holds when $\vecv_n=\vecv_n^0$ and
$\vecv_{n+1}=\vecv_{n+1}^0$, and then by continuity we may choose
$\delta>0$ so small that (i) and (ii) hold for all
$\vecv_n\in \scrN_\delta(\vecv_n^0)$, 
$\vecv_{n+1}\in \scrN_\delta(\vecv_{n+1}^0)$.
\end{proof}

\begin{proof}[Proof of Lemma \ref{GOODCOMPACTLEM} if $d=2$] 
First take $\delta$ and 
$D$   %
as in the previous proof; it then suffices to show that we can make
(iii) hold by removing an open subset
of arbitrarily small ($\vol_{\S_1^{1}}^n-$)volume from $D$ and possibly 
shrinking $\delta$. 

First consider $j=n$ in (iii);
to deal with this case it suffices to prove
\begin{align} \label{GOODCOMPACTLEMD2STEP0}
\lim_{\delta\to 0} 
\vol_{\S_1^{1}}\Bigl(
{\bigcup}_{(\vecv_n,\vecv_{n+1})\in \scrN_\delta(\vecv_{n}^0)
\times\scrN_\delta(\vecv_{n+1}^0)} S_{\vecv_n,\vecv_{n+1}}\Bigr)=0,
\end{align}
where
\begin{align}
S_{\vecv_n,\vecv_{n+1}}
:=\bigl\{\vecv_{n-1}\in\scrV_{\vecv_{n}}\col
\vecbeta_{\vece_1}^-(\vecv_{n+1}K(\vecv_{n}))_\perp
=(\vecbeta_{\vecv_{n-1}}^+(\vecv_{n})K(\vecv_{n}))_\perp\bigr\}.
\end{align}
Set $K_\theta:=\smatr{\cos\theta}{\sin\theta}{-\sin\theta}{\cos\theta}$;
then \eqref{GOODCOMPACTLEMD2STEP0} will follow if we can prove
\begin{align} \label{GOODCOMPACTLEMD2STEP0a}
\lim_{\delta\to 0} 
\vol_{\S_1^{1}}\Bigl({\bigcup}_{|\theta|\leq\delta}
\Bigl({\bigcup}_{\vecv_{n+1}\in \scrN_{2\delta}(\vecv_{n+1}^0)} 
S_{\vecv_n^0,\vecv_{n+1}}\Bigr)K_\theta\Bigr)=0.
\end{align}
Here the inner union equals (cf.\ Remark \ref{PALFBETPLUSCONTREM})
\begin{align} \label{UNIONFORMULA}
& \bigl\{\vecv_{n-1}\in\scrV_{\vecv_{n}^0}\col
\Theta_1(\vece_1,-\vecbeta_{\vecv_{n-1}K(\vecv_n^0)}^+(\vece_1)R_{\vece_1})
\in\scrN_{2\delta}(\vecv_{n+1}^0 K(\vecv_n^0))\bigr\},
\end{align}
where $R_{\vece_1}=\smatr 100{-1}$ is the reflection in the line $\R\vece_1$.
Since \eqref{UNIONFORMULA} is a closed subset of $\scrV_{\vecv_n^0}$ for 
every $\delta>0$, it suffices to prove that the volume
of (\ref{UNIONFORMULA}) tends to zero as $\delta\to 0$;
and since $\Theta_1(\vece_1,\cdot)$ is a $\C^1$ diffeomorphism,
this will follow if we can show
\begin{align} \label{GOODCOMPACTLEMD2STEP2}
\lim_{\delta\to 0} \vol_{\S_1^1}\bigl(\bigl\{\vecV\in\scrV_{\vece_1}\col
\vecbeta_{\vecV}^+(\vece_1)\in\scrN_{\delta}(\vecW)\bigr\}\bigr)=0,
\qquad\forall\vecW\in\S_1^{d-1}.
\end{align}
But for $\vecV=-(\cos\varphi)\vece_1+(\sin\varphi)\vece_2$
with $|\varphi|<\pi-B_\Theta$ we have, by a computation,
\begin{align} \label{BETANUFORMULA}
\vecbeta_{\vecV}^+(\vece_1)=(\cos\nu(\varphi))\vece_1
+(\sin\nu(\varphi))\vece_2
\quad\text{with}\quad
\nu(\varphi)=\vartheta_2(\vartheta_1^{-1}(\varphi))-\varphi,
\end{align}
where $\vartheta_1^{-1}$ and $\vartheta_2$ are as in 
Remark \ref{PRESERVELIOUVILLEREMARK} and Remark \ref{PALFBETEXPLICITREM}
(thus the function $\nu(\varphi)$ is $\C^1$ for all 
$|\varphi|<\pi-B_\Theta$).
Now, by (\ref{ASSUMPTIONSTAR}), the set
$\{\nu'(\varphi)=0\}$ has Lebesgue measure zero.
This implies \eqref{GOODCOMPACTLEMD2STEP2}, and hence also
\eqref{GOODCOMPACTLEMD2STEP0}.

Next, to deal with the case $j=n-1$ (thus $n\geq 2$)
in (iii), it suffices to prove
\begin{align} \label{GOODCOMPACTLEMD2STEP3}
\lim_{\delta\to 0} 
\vol_{\S_1^{1}}^2
\Bigl({\bigcup}_{\vecv_n\in \scrN_\delta(\vecv_{n}^0)} S_{\vecv_n}\Bigr)=0,
\end{align}
where
\begin{align} \label{GOODCOMPACTLEMD2STEP4}
S_{\vecv_n}:=\bigl\{(\vecv_{n-2},\vecv_{n-1})\in\scrV_{\vecv_{n}}^{[2]}\col
\vecbeta_{\vece_1}^-(\vecv_{n}K(\vecv_{n-1}))_\perp
=(\vecbeta_{\vecv_{n-2}}^+(\vecv_{n-1})K(\vecv_{n-1}))_\perp\bigr\}.
\end{align}
However, \eqref{GOODCOMPACTLEMD2STEP3} follows from the fact that 
$S_{\vecv_n^0}$ is a closed subset of $\scrV_{\vecv_{n}^0}^{[2]}$
with $\vol_{\S_1^{1}}^2(\scrV_{\vecv_{n}^0}^{[2]})=0$, 
using rotational invariance in a similar way as in 
\eqref{GOODCOMPACTLEMD2STEP0a}.

Finally, the case $j\leq n-2$ 
is easy, since 
\begin{align}
\vol_{\S_1^{1}}^3\bigl(\bigl\{(\vecv_{j-1},\vecv_j,\vecv_{j+1})
\in(\S_1^{1})^3\col
(\vecv_{j-1},\vecv_j)\in S_{\vecv_{j+1}}\bigr\}\bigr)=0,
\end{align}
with $S_{\vecv_{j+1}}$ defined 
in analogy with \eqref{GOODCOMPACTLEMD2STEP4}.
\end{proof}

\begin{lem} \label{FTCONTLEMMA}
If $f_0\in \C_J(\XX)$ then the function 
$f(t,\vecQ,\vecV,\xi,\vecV_+):=K_t f_0(\vecQ,\vecV,\xi,\vecV_+)$ 
belongs to $\C_J(\R_{\geq 0}\times\XX)$.
\end{lem}

\begin{proof}
From %
\eqref{KTNEXPLDEF} we obtain,
using Remarks \ref{PALFBETEXPLICITREM}, \ref{PALFBETPLUSEXPLICITREM}
and the fact that $0\leq\Phi_\bn\leq 1$,
\begin{align} \notag
\bigl|[K_t^{(n)}f_0](\vecQ,\vecv_n,\xi,\vecv_{n+1})\bigr|
& \leq \|f_0\|_J \frac{t^n}{n!}\int_{\scrV_{\vecv_n}^{[n]}} \prod_{j=0}^n
J(\vecv_j,\vecv_{j+1})\, d\!\vol_{\S_1^{d-1}}(\vecv_0) 
\cdots d\!\vol_{\S_1^{d-1}}(\vecv_{n-1}) 
\\ \label{FTCONTLEMMASTEP2}
& = \|f_0\|_J \frac{t^n}{n!}\vol(\scrB_1^{d-1})^n J(\vecv_n,\vecv_{n+1}),
\hspace{50pt} \forall n\geq 1.
\end{align}
From \eqref{KT0EXPLDEF} we see that the same bound is also true for $n=0$.
It follows that the sum $[K_tf_0](\vecQ,\vecv_n,\xi,\vecv_{n+1})
=\sum_{n=0}^\infty[K_t^{(n)}f_0](\vecQ,\vecv_n,\xi,\vecv_{n+1})$
is uniformly absolutely convergent for $t$ and $(\vecv_n,\vecv_{n+1})$ 
in compacta (and $\vecQ,\xi$ unrestricted), and we have
$\|K_tf_0\|_J\leq e^{t\vol(\scrB_1^{d-1})}\|f_0\|_J$ for each $t\geq 0$.
It now only remains to prove that $f$ is continuous on 
$\R_{\geq 0}\times\XX$, and for this it %
suffices to prove that each function $K_t^{(n)}f_0$ %
is continuous on $\R_{\geq 0}\times\XX$. %
The case $n=0$ is trivial, thus from now on we fix some $n\geq 1$.

Let $(t^0,\vecQ^0,\vecv_n^0,\xi^0,\vecv_{n+1}^0)\in
\R_{\geq 0}\times\XX$ and $\ve>0$ be given.
Let $\delta>0$ and $D\subset\scrV_{\vecv_n^0}^{[n]}$ be as in 
Lemma \ref{GOODCOMPACTLEM}, with 
$\|f_0\|_J^{-1} \frac{n!}{(t^0+1)^n}\ve$ in the place of $\ve$.
By \eqref{KTNEXPLDEF} and the bounds leading to \eqref{FTCONTLEMMASTEP2}
we then get that, for all $(t,\vecQ,\vecv_n,\xi,\vecv_{n+1})\in\R_{\geq 0}\times\XX$
with $t\leq t^0+1$, $\vecv_n\in\scrN_\delta(\vecv_n^0)$ and
$\vecv_{n+1}\in\scrN_\delta(\vecv_{n+1}^0)$,
the value of $[K_t^{(n)}f_0](\vecQ,\vecv_n,\xi,\vecv_{n+1})$ differs by at 
most $\ve$ from
\begin{align} \notag
K_{D,t}^{(n)}f_0(\vecQ,\vecv_n,\xi,\vecv_{n+1}):=
\int_{T_n\leq t}\int_D f_0\bigg(\vecQ-\big(\sum_{j=1}^n \xi_j\vecv_{j-1} + (t-T_n) \vecv_n\big),\vecv_0,\xi_1,\vecv_1\bigg) \hspace{30pt}&
\\ \label{FTCONTLEMMASTEP1}
\times
\prod_{j=1}^{n} p_{\vecnull,\vecbeta^+_{\vecv_{j-1}}}(\vecv_j,\xi_{j+1},\vecv_{j+1}) 
\, d\!\vol_{\S_1^{d-1}}(\vecv_0)\cdots d\!\vol_{\S_1^{d-1}}(\vecv_{n-1})
\,d\xi_1\cdots d\xi_n & . 
\end{align}
By Remark \ref{PALFBETPLUSCONTREM}
(using Lemma \ref{GOODCOMPACTLEM} (iii) if $d=2$),
the integrand in \eqref{FTCONTLEMMASTEP1} depends jointly continuously on
all the variables
$(\xi_1,\xi_2,\ldots,\xi_n)\in(\R_{\geq 0})^n$,
$(\vecv_0,\ldots,\vecv_{n-1})\in D$ and
$(t,\vecQ,\vecv_n,\xi,\vecv_{n+1})\in\R_{\geq 0}\times\XX$,
so long as $T_n\leq t$, 
$\vecv_n\in\scrN_\delta(\vecv_n^0)$ and
$\vecv_{n+1}\in\scrN_\delta(\vecv_{n+1}^0)$.
Hence since the domain of integration in \eqref{FTCONTLEMMASTEP1} is compact,
we have for all 
$(t,\vecQ,\vecv_n,\xi,\vecv_{n+1})\in\R_{\geq 0}\times\XX$
sufficiently near $( t^0,\vecQ^0,\vecv_n^0,\xi^0,\vecv_{n+1}^0)$:
\begin{align}
\bigl|K_{D,t}^{(n)}f_0(\vecQ,\vecv_n,\xi,\vecv_{n+1})-
K_{D,t^0}^{(n)}f_0(\vecQ^0,\vecv_n^0,\xi^0,\vecv_{n+1}^0)\bigr|<\ve,
\end{align}
and thus
\begin{align}
\Bigl| [K_t^{(n)}f_0](\vecQ,\vecv_n,\xi,\vecv_{n+1})
-[K_{t^0}^{(n)}f_0](\vecQ^0,\vecv_n^0,\xi^0,\vecv_{n+1}^0)\Bigr|
<3\ve.
\end{align}
Hence $[K_t^{(n)}f_0](\vecQ,\vecv_n,\xi,\vecv_{n+1})$ is indeed continuous,
since $\ve>0$ was arbitrary. 
\end{proof}

\begin{lem} \label{FTC1LEMMA}
If $f_0\in\C_J^1(\XX)$ and 
$f(t,\vecQ,\vecV,\xi,\vecV_+):=K_t f_0(\vecQ,\vecV,\xi,\vecV_+)$,
then the derivatives
$\partial_{Q_j} f$ ($j=1,\ldots,d$) and $\partial_\xi f$ all exist
and belong to $\C_J(\R_{\geq 0}\times\XX)$.
\end{lem}
\begin{proof}
We start by considering $\partial_\xi K_t^{(n)}f_0$.
First assume $n\geq 2$. 
To differentiate \eqref{KTNEXPLDEF} with respect to $\xi$ we first move
the integral over $\xi_1$ to the innermost position; this integral will
then appear as $\int_0^{t-T_{2,n}}\cdots d\xi_1$ where
$T_{2,n}:=\sum_{j=2}^n\xi_j$,
and it %
may be differentiated with respect to $\xi$
by using the differentiability assumption on $f_0$ and the fact that 
$p_{\vecnull,\vecbeta^+_{\vecv_{n-1}}}(\vecv_n,\cdot,\vecv_{n+1})$
depends continuously on its second argument (except when
$d=2$ and $\vecbeta_{\vece_1}^-(\vecv_{n+1}K(\vecv_{n}))_\perp
=(\vecbeta_{\vecv_{n-1}}^+(\vecv_{n})K(\vecv_{n}))_\perp$).
Hence we obtain, at least formally:
\begin{align} \notag
[\partial_\xi K_t^{(n)} f_0](\vecQ,\vecv_n,\xi,\vecv_{n+1}) \hspace{320pt} &
\\ \notag
= \int_{\scrV_{\vecv_n}^{[n]}}
\int_{\substack{\xi_2,\ldots,\xi_n\geq 0\\ T_{2,n}\leq t}}
f_0\biggl(\vecQ-\sum_{j=2}^n\xi_j\vecv_{j-1}-(t-T_{2,n})\vecv_n
,\vecv_0,0,\vecv_1 \biggr) 
\prod_{j=1}^{n-1}
p_{\vecnull,\vecbeta^+_{\vecv_{j-1}}}(\vecv_j,\xi_{j+1},\vecv_{j+1}) 
\hspace{25pt} &
\\ \notag
\times 
p_{\vecnull,\vecbeta^+_{\vecv_{n-1}}}(\vecv_n,\xi+t-T_{2,n},\vecv_{n+1})
\, d\xi_2\ldots d\xi_n 
\, d\!\vol_{\S_1^{d-1}}(\vecv_0)\cdots d\!\vol_{\S_1^{d-1}}(\vecv_{n-1}) &
\\ \notag
-\int_{\scrV_{\vecv_n}^{[n]}}
\int_{\substack{\xi_2,\ldots,\xi_n\geq 0\\ T_{2,n}\leq t}}
f_0\biggl(\vecQ-(t-T_{2,n})\vecv_0-\sum_{j=2}^n\xi_j\vecv_{j-1}
,\vecv_0,t-T_{2,n},\vecv_1 \biggr)
\prod_{j=1}^{n-1}
p_{\vecnull,\vecbeta^+_{\vecv_{j-1}}}(\vecv_j,\xi_{j+1},\vecv_{j+1}) &
\\ \label{KTNXIDER}
\times 
p_{\vecnull,\vecbeta^+_{\vecv_{n-1}}}(\vecv_n,\xi,\vecv_{n+1}) 
\, d\xi_2\ldots d\xi_n
\, d\!\vol_{\S_1^{d-1}}(\vecv_0)\cdots d\!\vol_{\S_1^{d-1}}(\vecv_{n-1}) &
\\ \notag
+\int_{\scrV_{\vecv_n}^{[n]}}
\int_{\substack{\xi_1,\ldots,\xi_n\geq 0\\ T_{n}\leq t}}
\Bigl[(\vecv_n-\vecv_0)\cdot \nabla_\vecQ f_0+\partial_\xi f_0\Bigr]
\bigg(\vecQ-\bigg(\sum_{j=1}^n \xi_j\vecv_{j-1} + (t-T_n) \vecv_n\bigg),\vecv_0,\xi_1,\vecv_1\bigg) &
\\ \notag
\times\prod_{j=1}^{n} p_{\vecnull,\vecbeta^+_{\vecv_{j-1}}}(\vecv_j,\xi_{j+1},\vecv_{j+1})
\, d\xi_1\ldots d\xi_n
\, d\!\vol_{\S_1^{d-1}}(\vecv_0)\cdots d\!\vol_{\S_1^{d-1}}(\vecv_{n-1}). &
\end{align}
Here $\partial_\xi f_0$ and all components of $\nabla_\vecQ f_0$ lie in
$\C_J(\XX)$, since $f_0\in\C_J^1(\XX)$; 
hence as in %
Lemma~\ref{FTCONTLEMMA} one shows that the right hand side of \eqref{KTNXIDER} 
is a continuous function of 
$(t,\vecQ,\vecv_n,\xi,\vecv_{n+1})\in\R_{\geq 0}\times\XX$.
Now to validate the formula \eqref{KTNXIDER} it suffices to
prove that \eqref{KTNXIDER} holds true after integration with respect to $\xi$
over an arbitrary finite interval $[a,b]\subset\R_{\geq 0}$, 
and this is easily verified using Fubini's theorem.
The case $n\geq 1$ is very similar, and the case $n=0$ is trivial.

From \eqref{KTNXIDER} and its analogues for $n=1,0$, one obtains 
a pointwise bound on $\partial_\xi K_t^{(n)} f$
similar to \eqref{FTCONTLEMMASTEP2}, involving 
$\|f\|_J$, $\|\partial_\xi f\|_J$ and
$\|\partial_{Q_j}f\|_J$ for $j=1,\ldots,d$.
This bound immediately implies that $\partial_\xi f
=\sum_{n=0}^\infty \partial_\xi K_t^{(n)} f_0$, with uniform absolute
convergence on compact subsets of $\scrX$, and
$\partial_\xi f\in\C_J(\R_{\geq 0}\times\XX)$, as desired.

The proof of $\partial_{Q_j} f\in\C_J(\R_{\geq 0}\times\XX)$ 
follows the same steps but with simpler formulas;
in fact $\partial_{Q_j} K_t^{(n)} f_0=K_t^{(n)} \partial_{Q_j} f_0$
holds for each $n\geq 0$.
\end{proof}

\begin{lem} \label{FTC1LEMMA2}
If $f_0\in\C_J^1(\XX)$ and 
$f(t,\vecQ,\vecV,\xi,\vecV_+):=K_t f_0(\vecQ,\vecV,\xi,\vecV_+)$, then also 
$\partial_t f$ exists and belongs to $\C_J(\R_{\geq 0}\times\XX)$,
and the equation \eqref{FPKEQ} holds throughout $\R_{\geq 0}\times\XX$.
\end{lem}
\begin{proof}
Fix any point $(t,\vecQ,\vecV,\xi,\vecV_+)\in\R_{\geq 0}\times\XX$.
Using $K_{t+h}=K_h K_t$ and the bound 
\eqref{FTCONTLEMMASTEP2} (with $h$ in place of $t$) added over all $n\geq 2$,
we get as $h\to 0^+$:
\begin{align} 
& f(t+h,\vecQ,\vecV,\xi,\vecV_+) 
=\sum_{n=0}^\infty [K_h^{(n)}K_tf_0](\vecQ,\vecV,\xi,\vecV_+)
\\ \notag
& =f(t,\vecQ-h\vecV,\vecV,\xi+h,\vecV_+)+
\int_0^h\int_{\scrV_{\vecv_0}}  f\big(t,\vecQ-(\xi_1\vecv_{0} + (h-\xi_1) \vecV),\vecv_0,\xi_1,\vecV \big) 
\\ \notag
& \hspace{140pt}
\times p_{\vecnull,\vecbeta^+_{\vecv_{0}}}(\vecV,\xi+h-\xi_1,\vecV_+) 
\,d\!\vol_{\S_1^{d-1}}(\vecv_0)\,d\xi_1 
+O(h^2).
\end{align}
Using $f(t,\cdot)\in\C_J^1(\XX)$ (cf.\ Lemma \ref{FTC1LEMMA}),
and treating the integral term with a continuity argument as in the proof of 
Lemma \ref{FTCONTLEMMA}, we get (cf.\ \eqref{ZDEF})
\begin{align} \label{FTC1LEMMA2step1}
& =  f(t,\vecQ,\vecV,\xi,\vecV_+)+h[Zf(t,\cdot)](\vecQ,\vecV,\xi,\vecV_+)+o(h).
\end{align}
A similar argument shows that the function
$[Zf(t,\cdot)](\vecQ,\vecV,\xi,\vecV_+)$ lies in 
$\C_J(\R_{\geq 0}\times\XX)$.
Letting $h\to 0^+$ in \eqref{FTC1LEMMA2step1} we conclude that the right 
derivative $\partial_t^+ f(t,\vecQ,\vecV,\xi,\vecV_+)$ exists and equals
$[Zf(t,\cdot)](\vecQ,\vecV,\xi,\vecV_+)$.
Since the latter function is continuous, the relation now follows
for the two-sided derivative, i.e.\ \eqref{FPKEQ2}
($\Leftrightarrow$ \eqref{FPKEQ}) holds, and
$\partial_t f\in \C_J(\R_{\geq 0}\times\XX)$.
\end{proof}

\begin{lem} \label{PTWLIMITLEM}
Let $\{f_h\}_{h>0}$ be a family of functions 
in $\C_J(\XX)$ satisfying $\limsup_{h\to 0} \|f_h\|_J<\infty$
and $f_h(\vecQ,\vecV,\xi,\vecV_+)\to 0$ as $h\to 0$,
uniformly over $(\vecQ,\vecV,\xi,\vecV_+)$ in compact subsets of $\XX$.
Then for each $(\vecQ,\vecV,\xi,\vecV_+)\in\XX$ we have
$[K_tf_h](\vecQ,\vecV,\xi,\vecV_+)\to 0$ as $h\to 0$,
uniformly over $t$ in compact subsets of $\R_{\geq 0}$.
\end{lem}

\begin{proof}
Take $C,h_0>0$ so that $\|f_h\|_J\leq C$ for all $h\in(0,h_0]$.
Given some $(\vecQ,\vecV,\xi,\vecV_+)\in\XX$ and %
$T>0$, $\eta>0$, $\ve>0$, 
after possibly shrinking $h_0$ we may assume that
$\bigl|f_h(\vecQ',\vecV',\xi',\vecV_+')\bigr|<\ve J(\vecV',\vecV_+')$ for all
$h\in(0,h_0]$ and all $(\vecQ',\vecV',\xi',\vecV_+')\in\XX$ with 
$\|\vecQ'-\vecQ\|\leq T$, $\xi'\leq\xi+T$ and 
$\vecV'\in\overline{\scrV^\eta_{\vecV'}}_{\hspace{-5pt}_+}$ 
(cf.\ \eqref{WVPDEF}).
Then for each $n\geq 1$, $t\in[0,T]$ and $h\in(0,h_0]$,
by mimicking %
\eqref{FTCONTLEMMASTEP2}
but splitting the integral over $\scrV_{\vecV}^{[n]}$ according to the
two cases $\vecv_0\in\scrV_{\vecv_1}^\eta$ and
$\vecv_0\notin\scrV_{\vecv_1}^\eta$, we obtain 
(with $C_1:=\vol(\scrB_1^{d-1})$)
\begin{align} \label{PTWLIMITLEMSTEP1}
\bigl| [K_t^{(n)}f_h](\vecQ,\vecV,\xi,\vecV_+)\bigr|
\leq  
\frac{C_1^{n}t^n}{n!}\biggl\{\ve+
\frac{C}{C_1}
\int_{\scrV_{\vece_1}\setminus\scrV_{\vece_1}^{\eta}} J(\vecv_0,\vece_1)
\,d\!\vol_{\S_1^{d-1}}(\vecv_0)\biggr\} J(\vecV,\vecV_+).
\end{align}
This bound also holds for $n=0$ (without the 
$\int_{\scrV_{\vece_1}\setminus\scrV_{\vece_1}^{\eta}}$-term), so long as 
$\eta<\varphi(\vecV,\vecV_+)-B_\Theta$.
The expression within the brackets in \eqref{PTWLIMITLEMSTEP1}
can be made arbitrarily small
by taking $\ve$ and $\eta$ to be sufficiently small.
Now the desired conclusion follows by adding \eqref{PTWLIMITLEMSTEP1}
over $n\geq 0$.
\end{proof}

\begin{lem} \label{COMMUTELEMMA}
We have $K_tZf_0=ZK_tf_0$, %
for all $f_0\in\C_J^1(\XX)$ and $t\geq 0$.
\end{lem}

\begin{proof}
Let $f_0\in\C_J^1(\XX)$ be given.
Then by Lemma \ref{FTC1LEMMA2} (cf.\ \eqref{FPKEQ2}) we have 
$[\partial_t K_tf_0](\vecQ,\vecV,\xi,\vecV_+)\equiv
[ZK_t f_0](\vecQ,\vecV,\xi,\vecV_+)$, and this function belongs
to $\C_J(\R_{\geq 0}\times\XX)$. Hence
\begin{align} \label{COMMUTELEMMASTEP1}
h^{-1}[K_hf_0-f_0](\vecQ,\vecV,\xi,\vecV_+)=
h^{-1}\int_0^h [ZK_t f_0](\vecQ,\vecV,\xi,\vecV_+)\, dt
\end{align}
for all $h>0$ and all $(\vecQ,\vecV,\xi,\vecV_+)\in\XX$,
and this expression tends to $[Zf_0](\vecQ,\vecV,\xi,\vecV_+)$ as $h\to 0^+$,
uniformly over $(\vecQ,\vecV,\xi,\vecV_+)$ in compact subsets of $\XX$.
Using \eqref{COMMUTELEMMASTEP1} we also get
$\sup_{0<h\leq 1}\bigl\|h^{-1}[K_hf_0-f_0]\bigr\|_J
\leq \sup_{0<t\leq 1}\bigl\|ZK_t f_0\bigr\|_J<\infty$.
Hence by Lemma \ref{PTWLIMITLEM} we have
\begin{align}
\lim_{h\to 0^+} h^{-1}[K_t[K_hf_0-f_0]](\vecQ,\vecV,\xi,\vecV_+)=
[K_t Zf_0](\vecQ,\vecV,\xi,\vecV_+),
\end{align}
for every $(t,\vecQ,\vecV,\xi,\vecV_+)\in\R_{\geq 0}\times\XX$.
Using $K_tK_hf_0=K_{t+h}f_0$ and \eqref{FPKEQ2} this implies
\begin{align} 
[K_tZf_0](\vecQ,\vecV,\xi,\vecV_+)
=[\partial^+_t K_tf_0](\vecQ,\vecV,\xi,\vecV_+)
=[ZK_tf_0](\vecQ,\vecV,\xi,\vecV_+).
\end{align}
\end{proof}

\begin{proof}[Proof of Theorem \ref{FPK}]
By Lemma \ref{FTC1LEMMA} and Lemma \ref{FTC1LEMMA2}, 
it only remains to prove the uniqueness.
Thus assume that $f\in\C_J^1(\R_{\geq 0}\times\XX)$ satisfies \eqref{FPKEQ};
we then need to prove that
$f(t,\vecQ,\vecV,\xi,\vecV_+)=[K_t f(0,\cdot)](\vecQ,\vecV,\xi,\vecV_+)$ 
for all $(t,\vecQ,\vecV,\xi,\vecV_+)\in\R_{\geq 0}\times\XX$.

Fix an arbitrary point 
$(\vecQ,\vecV,\xi,\vecV_+)\in\XX$ and some $a>0$, and set
\begin{align*}
F(s)=[K_{a-s} f(s,\cdot)](\vecQ,\vecV,\xi,\vecV_+)\qquad\text{for }\:
s\in [0,a].
\end{align*}
Then $F(s)$ is continuous, by Lemma \ref{PTWLIMITLEM}.
We will prove that the right derivative $D^+F(s)$ vanishes for each 
$s\in (0,a)$.
This will conclude the proof, since it implies $F(0)=F(a)$, 
which is the desired relation at 
$(t,\vecQ,\vecV,\xi,\vecV_+)=(a,\vecQ,\vecV,\xi,\vecV_+)$.

Thus fix $s\in (0,a)$.
By Lemma \ref{FTC1LEMMA2} applied with $f_0=f(s,\cdot)$ we have
\begin{align} \label{FPKUNIQSTEP2}
\lim_{h\to 0^+} \frac{\bigl[K_{a-s-h}-K_{a-s}\bigr]\bigl[f(s,\cdot)\bigr]
(\vecQ,\vecV,\xi,\vecV_+)}h
=-\bigl[ZK_{a-s} f(s,\cdot)\bigr](\vecQ,\vecV,\xi,\vecV_+).
\end{align}
Next, using the identity
\begin{align}
\frac{f(s+h,\vecQ',\vecV',\xi',\vecV_+')-f(s,\vecQ',\vecV',\xi',\vecV_+')}h
=\frac 1h\int_s^{s+h}\partial_tf(t,\vecQ',\vecV',\xi',\vecV_+')\,dt
\end{align}
and $\partial_tf\in\C_J(\R_{\geq 0}\times\XX)$ we see that
$h^{-1}[f(s+h,\cdot)-f(s,\cdot)]$ has uniformly bounded $\|\cdot\|_J$-norm
for $h\in(0,a-s]$, and approaches
$\partial_sf(s,\cdot)$ uniformly on compact subsets of $\XX$ as $h\to 0^+$.
Hence, using Lemma \ref{PTWLIMITLEM} and the fact that the function
$\R_{\geq 0}\ni b\mapsto[K_b\partial_s f(s,\cdot)](\vecQ,\vecV,\xi,\vecV_+)$
is continuous (cf.\  Lemma \ref{FTCONTLEMMA}), we obtain, at our fixed point
$(\vecQ,\vecV,\xi,\vecV_+)\in\XX$:
\begin{align} \label{FPKUNIQSTEP1}
\lim_{h\to 0^+} \frac{K_{a-s-h}\bigl[f(s+h,\cdot)-f(s,\cdot)\bigr]
(\vecQ,\vecV,\xi,\vecV_+)}h
=\bigl[K_{a-s}\partial_sf(s,\cdot)\bigr]
(\vecQ,\vecV,\xi,\vecV_+).
\end{align}
But we are assuming that $f$ satisfies \eqref{FPKEQ2};
thus $[\partial_sf](s,\cdot)\equiv Zf(s,\cdot)$.
Hence, when adding \eqref{FPKUNIQSTEP2} and \eqref{FPKUNIQSTEP1}
and using Lemma \ref{COMMUTELEMMA}, we obtain $D^+F(s)=0$, as desired.
\end{proof}

\begin{remark} \label{PPPINTREM}
Proposition \ref{eigenlemma} and Theorem \ref{FPK} imply that
if the flow $F_t$ preserves the Liouville measure, then
\begin{equation} \label{PPPINT}
-\partial_\xi\, p(\vecV,\xi,\vecV_+) = \int_{\S_1^{d-1}} J(\vecv_0,\vecV)
\, p_{\vecnull,\vecbeta^+_{\vecv_{0}}}(\vecV,\xi,\vecV_+) 
\,d\!\vol_{\S_1^{d-1}}(\vecv_0) ,
\end{equation}
since $p(\vecv_0,0,\vecV)=J(\vecv_0,\vecV)$.
Using \eqref{PLBETACRITERION} and \eqref{PALFBETEXPLICITREMFORMULA}
this equation can be reformulated as
\begin{align}
-\partial_\xi\Phi_\vecalf(\xi,\vecw,\vecz)
=\int_{\{0\}\times\scrB_1^{d-1}}
\Phi_\bn\bigl(\xi,\vecw,\vecz')\,d\vecz'
\end{align}
for $\vecalf\notin\Q^d$. 
For an alternative proof of this equation working directly from the 
definition of $\Phi_\vecalf$, see \cite[(8.32), (8.37)]{partI}
and note %
$\Phi_\bn(\xi,\vecw,\vecz')=\Phi_\bn(\xi,\vecz',\vecw)$.
\end{remark}

\begin{remark} \label{ORIGINALPROPAGATORREM}
Given any $\overline f_0\in\L^1(\T^1(\R^d))$ we define %
$f_0\in\L^1(\XX,\sigma_1)$ by $f_0(\vecQ,\vecV,\xi,\vecV_+)
=\overline f_0(\vecQ,\vecV) p(\vecV,\xi,\vecV_+)$.
We then have the relation
\begin{equation}
[L_t \overline f_0](\vecQ,\vecV) = \int_0^\infty \int_{\S_1^{d-1}} 
[K_t f_0](\vecQ,\vecV,\xi,\vecV_+)\,
	 d\!\vol_{\S_1^{d-1}}(\vecV_+)\, d\xi,
\end{equation}
where $L_t$ is the propagator of the original stochastic process $\Xi(t)$,
cf.\ Section \ref{secMacroscopic}.
If we furthermore impose that $\overline f_0$ is bounded continuous,
with bounded continuous derivatives $\partial_{Q_j}\overline f_0$
for $j=1,\ldots,d$, then $f_0$ satisfies the assumption of 
Theorem \ref{FPK}, so that $K_tf_0$ satisfies the Fokker-Planck-Kolmogorov 
equation, \eqref{FPKEQ}.
\end{remark}

\end{document}